\documentclass[11pt,twoside]{article}
\usepackage[margin=3cm,footskip=30pt,headheight=30pt]{geometry}
\usepackage{amsmath,amsthm}
\usepackage[all]{xy}
\usepackage{amssymb,mathrsfs,bbm}
\usepackage{amscd,xspace}
\usepackage[enableskew]{youngtab}
\usepackage{ytableau}
\usepackage{comment}
\usepackage{nicefrac,todonotes}
\usepackage{fancyhdr}
\usepackage{color}
\usepackage{hyperref}

\newtheorem{theorem}{Theorem}[section]
\newtheorem{lemma}[theorem]{Lemma}
\newtheorem{proposition}[theorem]{Proposition}
\newtheorem{corollary}[theorem]{Corollary}

\newtheorem{conjecture*}[theorem]{Conjecture*}
\newtheorem{itheorem}{Theorem}

\newtheorem{icorollary}{Corollary}

\numberwithin{equation}{section}

{
\theoremstyle{definition}
\newtheorem{definition}[theorem]{Definition}

\newtheorem{remark}[theorem]{Remark}

\newtheorem{question*}[theorem]{Question*}
}

  \newcommand{\nc}{\newcommand}
  \newcommand{\renc}{\renewcommand}

\newcommand{\arxiv}[1]{\href{http://arxiv.org/abs/#1}{\tt\nolinkurl{arXiv:#1}}}

\nc{\ep}{\epsilon}
\nc{\hh}{h}
\nc{\fg}{\mathfrak g}
\nc{\fk}{\mathfrak k}
\nc{\fh}{\mathfrak h}
\nc{\kb}{\mathbb{C}}
\nc{\C}{\mathbb{C}}
\nc{\cO}{\mathcal{O}}
\nc{\cC}{\mathcal{C}}
\nc{\cI}{\mathcal{I}}
\nc{\cB}{\mathcal{B}}
\nc{\cP}{\mathcal{P}}
\nc{\cS}{\mathcal{S}}
\nc{\B}{\cB}
\nc{\tc}{\widetilde{\mathbf{c}}}
\nc{\cM}{\mathcal{M}}
\nc{\cN}{\mathcal{N}}
\nc{\mmod}{\operatorname{-mod}}
\nc{\PS}{\mathcal{P\!S}}
\nc{\cA}{\mathcal{A}}
\nc{\cT}{\mathcal{T}}
\nc{\br}{\mathbf{r}}
\nc{\bx}{\mathbf{x}}
\nc{\bd}{\mathbf{d}}
\nc{\bc}{\mathbf{c}}
\nc{\bb}{\mathbf{b}}
\nc{\bh}{\mathbf h}
\nc{\bB}{\mathbf{B}}
\nc{\ba}{\mathbf{a}}
\nc{\bA}{\mathbf{A}}
\nc{\bs}{\mathbf{s}}
\nc{\bR}{\mathbf{R}}
\nc{\tbR}{\widetilde{\bR}}
\nc{\bS}{\mathbf{S}}
\nc{\bV}{\mathbf{V}}
\nc{\wL}{\widetilde{L}}

\nc{\Sym}{\operatorname{Sym}}
\nc{\hSym}{\operatorname{\widehat{Sym}}}
\nc{\Hom}{\operatorname{Hom}}

\nc{\fmod}{\operatorname{-fmod}}
\nc{\om}{\omega}
\nc{\Z}{\mathbb{Z}}
\nc{\si}{\sigma}
\nc{\la}{\lambda}
\nc{\al}{\alpha}
\nc{\Gr}{\mathsf{Gr}}
\nc{\Gp}{G[t]}
\nc{\Gm}{{G_1[[t^{-1}]]}}
\nc{\gp}{\mathfrak{g}[t]}
\nc{\gm}{t^{-1}\mathfrak{g}[[t^{-1}]]}
\nc{\Grlbar}{\overline{\Gr^\lambda}}
\nc{\Grlmbar}{\Gr^{\overline{\lambda}}_\mu}
\nc{\Grblmbar}{\Gr^{\bla}_\mu}
\renc{\O}{\mathcal{O}}
\nc{\val}{\operatorname{val}}
\nc{\pp}{\mathbb{P}^1}
\nc{\excise}[1]{}
\nc{\bla}{{\underline{\boldsymbol{\la}}}}
\nc{\bz}{\mathbf{z}}
\nc{\vlam}{\vec{\lambda}}
\nc{\wY}{\widetilde{Y}}
\nc{\col}{\operatorname{col}}

\nc{\defD}{\mathscr{\tilde{D}}}
\nc{\K}{\mathbbm{k}}
\nc{\defE}{\mathscr{\tilde{E}}}
\nc{\defA}{{\tilde{A}}}
\nc{\wK}{\widetilde{K}}
\nc{\wT}{\widetilde{T}}
\nc{\wc}{\widetilde{c}}
\nc{\wphi}{\widetilde{\phi}}
\nc{\Yml}{Y_\mu^\lambda}
\nc{\gr}{\operatorname{gr}}
\nc{\BB}{\mathbb{B}}
\nc{\Ann}{\operatorname{Ann}}
\nc{\Spec}{\operatorname{Spec}}
\nc{\ord}{\operatorname{ord}}
\nc{\End}{\operatorname{End}}
\nc{\rt}{\gamma}
\nc{\brt}{\boldsymbol{\gamma}}

\nc{\bG}{\mathbf{G}}	
\nc{\bN}{\mathbf{N}}

\nc{\RowR}{Row_{\bR}(\pi)}
\nc{\RowRt}{Row_{\tbR}(\pi)}
\nc{\RowRtc}{Row_{\tbR}(\pi)^{\circ}}

\nc{\Norm}{\Theta}

\nc{\bysame}{\leavevmode\hbox to3em{\hrulefill}\thinspace}

\newcommand{\fp}{\mathfrak{p}}
\newcommand{\fl}{\mathfrak{l}}

\newcommand{\fX}{\mathfrak{X}}
\newcommand{\fu}{\mathfrak{u}}
\newcommand{\fm}{\mathfrak{m}}

\newcommand{\fb}{\mathfrak{b}}
\renc{\al}{\alpha}
\newcommand{\thetitle}{On the quantum Mirkovi\'c-Vybornov isomorphism}
\newcommand{\theauthors}{Ben Webster, Alex Weekes \& Oded Yacobi}

\newcommand{\acom}[1]{\todo[inline,color=green!20]{ Alex: #1 }}
\newcommand{\ocom}[1]{\todo[inline,color=magenta!20]{Oded: #1}}
\newcommand{\bcom}[1]{\todo[inline,color=yellow!20]{Ben: #1}}

\begin{document}
\thispagestyle{plain}

\begin{center} {\Large \bf A quantum Mirkovi\'c-Vybornov isomorphism}
\end{center}
\bigskip

\noindent{\bf Ben Webster}\\
Department\ of Pure Mathematics, University of Waterloo \& \\Perimeter Institute for Theoretical Physics\\ {\tt ben.webster@uwaterloo.ca}\smallskip \\\noindent{\bf Alex Weekes}\\
Perimeter Institute for Theoretical Physics\\ {\tt aweekes@perimeterinstitute.ca}\smallskip \\
{\bf Oded Yacobi}\\
School\ of Mathematics and Statistics, University of Sydney\\ {\tt oded.yacobi@sydney.edu.au}\smallskip \\
\bigskip\\

\renewcommand{\thefootnote}{\fnsymbol{footnote}}

\begin{abstract}
We present a quantization of an isomorphism of Mirkovi\'c and Vybornov which relates the intersection of a Slodowy slice and a nilpotent orbit closure in $\mathfrak{gl}_N$, to a slice between spherical Schubert varieties in the affine Grassmannian of $PGL_n$ (with weights encoded by the Jordan types of the nilpotent orbits).  A quantization of the former variety is provided by a parabolic W-algebra and of the latter by a truncated shifted Yangian.  Building on earlier work of Brundan and Kleshchev, we define an explicit isomorphism between these non-commutative algebras, and show that its classical limit is a variation of the original isomorphism of Mirkovi\'c and Vybornov.  As a corollary, we deduce that the W-algebra is free as a left (or right) module over its Gelfand-Tsetlin subalgebra, as conjectured by Futorny, Molev, and Ovsienko.  
\end{abstract}

\section{Introduction}
In \cite{MV} Mirkovi\'c and Vybornov construct an isomorphism between slices to (spherical) Schubert varieties in the affine Grassmannian of $PGL_n$ on the one hand, and Slodowy slices in $\mathfrak{gl}_N$ intersected with nilpotent orbit closures on the other.  This isomorphism has important applications in geometric representation theory.  To name just a few occurrences, it appears in works on the mathematical definition of the Coulomb branch associated to quiver gauge theories \cite{Nak16}, the analog of the geometric Satake isomorphism for affine Kac-Moody groups \cite{BFII}, and geometric approaches to knot homologies \cite{CK,CKL2}.  

These varieties each have quantizations corresponding to natural Poisson structures on them.  The main aim of this paper is to show that the Mirkovi\'c-Vybornov isomorphism is the classical limit of an isomorphism of these quantizations.  

To be more precise, the Slodowy slice $\cS_e$ through a nilpotent element $e\in \mathfrak{gl}_N$  is quantized by a finite W-algebra.  Finite W-algebras algebras have been extensively studied by Kostant, Lynch, Premet, Gan-Ginzburg, and many others (cf. \cite{GG} and references therein).  The quantization of $\cS_e \cap \overline{\mathbb{O}_{e'}}$, the intersection of $\cS_e$ with the closure of the nilpotent orbit through another nilpotent $e'$, is given by a parabolic W-algebra \cite{L,WebWO}.  Parabolic W-algebras are quotients of finite W-algebras.   

Slices to Schubert varieties in the affine Grassmannian of $PGL_n$ are indexed by pairs $\mu, \lambda$ of dominant coweights of $PGL_n$, such that $\mu \leq \lambda$ in the dominant coroot ordering.  We denote the slice by $\Grlmbar$.  In \cite{KWWY} the present authors, along with Kamnitzer, quantized $\Grlmbar$ using algebras called truncated shifted Yangians. 

The Mirkovi\'c-Vybornov isomorphism is an explicit isomorphism of varieties
\begin{equation}
\label{eq:MVyisom}
\cS_e \cap \overline{\mathbb{O}_{e'}} \cong \Grlmbar,
\end{equation}
where $e,e'$ and are related to $\mu,\lambda$ by a certain combinatorial correspondence (cf. Sections \ref{section:CombData} and \ref{section: MV and Slodowy}).  Naturally one expects that (\ref{eq:MVyisom}) is the classical limit of an isomorphism between the quantizations of these varieties.  That is our main result.

\begin{itheorem}[\mbox{Theorem \ref{thm: main theorem}, part (c)}] \label{thm:main}
Suppose $e,e'$ (respectively $\mu,\lambda$) is a pair of nilpotent elements (respectively dominant coweights) which are related by the Mirkovi\'c-Vybornov isomorphism (\ref{eq:MVyisom}).  Then there is an isomorphism of filtered algebras between 
  the parabolic W-algebra quantizing $\cS_e \cap \overline{\mathbb{O}_{e'}}$ and the truncated shifted Yangian quantizing $\Grlmbar$.
\end{itheorem}

One can immediately conclude from this theorem that (\ref{eq:MVyisom}) is an isomorphism of \textit{Poisson} varieties (Corollary \ref{cor:Poissonisom}).  Moreover, since truncated shifted Yangians are explicitly presented, this theorem provides a presentation of parabolic W-algebras in type A.  This generalizes Brundan and Kleshchev's foundational work on presentations of finite W-algebras \cite{BK}.


Our final corollary of Theorem \ref{thm:main}  uses the recent interpretation of the truncated shifted Yangian in the setting of SUSY gauge theories.  The parabolic W-algebra has a distinguished maximal commutative subalgebra, called the Gelfand-Tsetlin subalgebra.  In the case where $\lambda$ is a multiple of the first fundamental weight, this agrees with the Gelfand-Tsetlin subalgebra as defined by Futorny, Molev, and Ovsienko.  They conjecture that the finite W-algebra is free as a left (or right) module over its Gelfand-Tsetlin subalgebra  \cite[Conjecture 2]{FMO}.  Using Theorem \ref{thm:main} we obtain (a generalization of) this conjecture by connecting it to work of Braverman, Finkelberg, and Nakajima on the mathematical theory of Coulomb branches for $3d$ $\cN=4$ gauge theories.

\begin{icorollary}[\mbox{Corollary \ref{cor: GT free}}]
The parabolic W-algebra is free as a left (or right) module over its Gelfand-Tsetlin subalgebra.
\end{icorollary}

\begin{remark}\hfill
\begin{enumerate}
\item In \cite{MV}, the authors consider a second family of isomorphisms, based on work of Maffei \cite{Maf}, between Slodowy slices and type A quiver varieties.  This isomorphism has already been quantized by Losev \cite[Th. 5.3.3]{L}.
\item When $\lambda$ is a multiple of the first fundamental weight, then $\overline{\mathbb{O}_{e'}}$ is the nilpotent cone of $\mathfrak{gl}_N$.  In this case, the quantization of $\cS_e \cap \overline{\mathbb{O}_{e'}}$ is  a central quotient of the finite W-algebra, and the isomorphism of Theorem \ref{thm:main} is  a variation of Brundan and Kleshchev's theorem (using the Drinfeld presentation of the Yangian instead of the RTT presentation).  Indeed Losev has speculated that Brundan and Kleshchev's presentation  should be understood as a quantization of the Mirkovi\'c-Vybornov isomorphism \cite[Rmk. 5.3.4]{L}, and Theorem \ref{thm:main} makes this precise. 
\end{enumerate}
\end{remark}

In order to prove Theorem \ref{thm:main}, we need results about the highest weight theory of parabolic W-algebras and truncated shifted Yangians. Brundan and Kleshchev describe the highest weights in category $\cO$ of a finite W-algebra in terms of row tableau.  First we describe those highest weights which descend to the parabolic W-algebra using so-called parabolic-singular elements of the Weyl group (Theorem \ref{thm:parabolicWweight}).  These are elements which are simultaneously longest left coset for a parabolic corresponding to $\mu$ and shortest right coset representatives for a parabolic corresponding to $\la$.  
This allows for the following new description of the parabolic W-algebra:

\begin{itheorem}[Theorem \ref{thm:parabolic-annihilator}]
In type A, the parabolic W-algebra is the quotient of the finite W-algebra by the intersection of annihilators of simple modules corresponding to parabolic-singular permutations.
\end{itheorem}

Now to prove Theorem \ref{thm:main} we first prove the desired isomorphism in the case where $\lambda$ is a multiple of the first fundamental coweight (Theorem \ref{mainthm}).  This is an explicit calculation with the Brundan-Kleshchev isomorphism, comparing different subquotients of the Yangian of $\mathfrak{sl}_n$ on the one hand, and the Yangian of $\mathfrak{gl}_n$ on the other.  We then use  results about the highest weight theory of the truncated shifted Yangian given by Kamnitzer, Tingley and the authors in \cite{KTWWY}, and the highest weight theory of the parabolic W-algebra from Section \ref{sec:highestweightsA}, to deduce the general result from the special case.   

In Section \ref{section: MV slices} we introduce general ``MV slices'', and prove an easy but useful result that any two MV slices are Poisson isomorphic (Theorem \ref{theorem: MV slices}).  Recently Cautis and Kamnitzer described a variation on the classical Mirkovi\'c-Vybornov isomorphism, which uses MV slices that are transposes of those used by Mirkovi\'c and Vybornov (cf. Section \ref{sec: the mv isomorphism}).  This isomorphism is much simpler to express in coordinates, and we prove that it is the classical limit of our quantum isomorphism.

\begin{itheorem}[Theorem \ref{thm: main theorem}, part (d)]
The classical limit of the quantum Mirkovi\'c-Vybornov isomorphism in Theorem \ref{thm:main} agrees with Cautis and Kamnitzer's version of the classical Mirkovi\'c-Vybornov isomorphism.
\end{itheorem}

\section*{Acknowledgement}  The authors thank Tomoyuki Arakawa, Joel Kamnitzer, Alex Molev, Stephen Morgan and Dinakar Muthiah for helpful conversations.  We also thank an anonymous referee for very helpful comments. A.W. and B.W. thank the University of Sydney for hosting visits where major parts of this research was conducted, and for providing an International Research Collaboration Award to support B.W.'s visit. O.Y. is supported by the Australian Reseach Council.  A.W. was supported in part by NSERC and the Ontario Ministry of Training, Colleges and Universities.   B.W. was supported in part by the NSF Division of Mathematical Sciences and the Alfred P. Sloan Foundation.  This research was supported in part by Perimeter Institute for Theoretical Physics. Research at Perimeter Institute is supported by the Government of Canada through the Department of Innovation, Science and Economic Development Canada and by the Province of Ontario through the Ministry of Research, Innovation and Science.

\subsection{Notation}
\label{section: notation}

Throughout this paper, we alternate between letting $\fg$ be any simply-laced simple complex Lie algebra, and specializing to the special or general linear Lie algebra.  In the beginning of every section we are careful to note which setting we are in.  

In general, we let $I$ denote the nodes of the Dynkin diagram of $\fg$, and we write $j \sim i$ to mean $j$ and $i$ are connected in the Dynkin diagram. Since Langlands duality often appears in the context of the affine Grassmannian, we will use dual notation, and denote simple coroots by $\{ \alpha_i \}_{i\in I}$ and fundamental coweights by $\{ \varpi_i\}_{i\in I}$, and dually, the simple roots $\{ \alpha_i^\vee\}_{i\in I}$ and fundamental weights by $\{ \varpi_i^\vee\}_{i\in I}$.  We let $\Delta^+$ denote the set of positive roots of $\fg$.  When we specialize to $\fg= \mathfrak{sl}_n$ we set $I= \{1,\ldots, n-1\}$.  

All spaces considered are varieties, schemes, or ind-schemes over $\C$.

\subsection{Combinatorial data}
\label{section:CombData}
Let $\fg= \mathfrak{sl}_n$.
Consider a pair $\lambda, \mu$ of dominant coweights for $\fg$, such that $\lambda \geq \mu$.  Write
\begin{equation}
\label{eq: coweight data}
\lambda = \sum_{i=1}^{n-1} \lambda_i \varpi_{n-i}, \qquad \mu = \sum_{i=1}^{n-1} \mu_i \varpi_{n-i}, \qquad \lambda - \mu = \sum_{i=1}^{n-1} m_i \alpha_{n-i} 
\end{equation}
so that $\lambda \geq \mu$ means precisely that all $m_i \in \mathbb{Z}_{\geq 0}$.  (Our strange indexing conventions above are chosen to match those of \cite{KWWY}.)  Define
\begin{equation}
\label{eq: coweight data 2}
N = \sum_{i=1}^{n-1} i \lambda_{n-i}
\end{equation}
Then $N \varpi_1 \geq \lambda \geq \mu$.  Write $N\varpi_1 - \mu = \sum_i m'_i \alpha_{n-i}$.  

We associate a pair of partitions to the above data as follows: first, the partition $\tau\vdash N$ is defined in exponential notation by \begin{equation}
\label{eq: coweight data 3}
\tau = \big(1^{\lambda_{n-1}}2^{\lambda_{n-2}}\cdots (n-1)^{\lambda_1} \big)^t.
\end{equation}
Second, consider the partition $\pi \vdash N$, 
\begin{equation}
\label{eq: coweight data 4}
\pi=(p_1\leq \cdots \leq p_n),
\end{equation}
defined by 
\begin{equation}
\label{eq: coweight data 5}
p_1=m'_1, p_2=m'_2 -m'_1,...,p_{n-1}=m'_{n-1}-m'_{n-2}, p_n=N-m'_{n-1}.
\end{equation}
Then $\tau \geq \pi$ with respect to the dominance order on partitions.

\begin{remark}
As a matter of convention, we will write partitions as either non-increasing or non-decreasing as appropriate.  
\end{remark}

\section{The affine Grassmannian side}\label{section: Grassmannian}

In this section we recall 
truncated shifted Yangians in type A, and their connection to slices in the affine Grassmannian of $PGL_n$.  Throughout this section  $\fg= \mathfrak{sl}_n$, and we fix a pair $\lambda \geq \mu$ of dominant coweights as in Section \ref{section: notation}.  

\subsection{Slices in the affine Grassmannian}
\label{section: slices in the affine Grassmannian}
Consider (spherical) Schubert cells $\Gr^\mu,\Gr^\lambda$ in the affine Grassmannian $\Gr$ for $PGL_n$.  Our running hypothesis that $\lambda\geq \mu$ implies that $\Gr^\mu \subset \overline{\Gr^\lambda}$, and we let $\Grlmbar$ be the slice to $\Gr^\mu$ in $\overline{\Gr^\lambda}$ at the point $t^{w_0\mu}$.  See \cite[Section 2.2]{KWWY} for more details and precise definitions, as well as Section \ref{section: more on gr slices} below.

$\Grlmbar$ is an irreducible affine variety of dimension $2\langle \rho^\vee, \lambda - \mu\rangle = 2 \sum_i m_i$.  It has a $\C^\times$--action by loop rotation, which contracts it to the unique fixed point $t^{w_0 \mu}$.  $\Grlmbar$ admits a Poisson structure which is homogeneous of degree $-1$ with respect to the loop rotation, as described in \cite[Section 2C]{KWWY}.  

Recall that $\Gr$ admits a description in terms of lattices: every point is given by a $\C[[t]]$--lattice in $\C((t))^n$; this is only well-defined up to multiplication by a power of $t$, but we will consistently choose representatives $\Lambda$ such that $\Lambda \subset \Lambda_0=\C[[t]]^n$.  Denote  $E_\pi = \left\{ t^{p_i-1} e_i, \ldots, t e_i, e_i: \forall i\right\}$ and $E_p = \left\{ t^{p-1} e_i,\ldots, e_i : \forall i \right\}$, where $e_1,...,e_n$ is the standard basis of $\C^n$. Explicitly, we can identify:
\begin{equation}
\label{eq: def of GrNwm}
\Gr^{\overline{N\varpi_1}}_\mu = \left\{ \Lambda \ : \begin{array}{l} (a) \ \ \Lambda \subset \Lambda_0 \text{ a } \C[[t]]\text{--submodule},\\ (b) \ \ \text{image of } E_\pi \text{ gives basis of } \Lambda_0/\Lambda, \\ (c) \ \ \forall i, \ t^{p_i} e_i \in \Lambda + E_{p_i} \end{array}  \right\}
\end{equation}
Since $N \varpi_1\geq\lambda$, we have inclusions of closed subvarieties $\overline{\Gr^\lambda}\subset\overline{\Gr^{N\varpi_1}}$ and $\Grlmbar\subset\Gr^{\overline{N\varpi_1}}_\mu$.  Considering multiplication by $t$ as an endomorphism of $\Lambda_0 / \Lambda$, we can also identify
\begin{equation}
\label{eq: def of Grlm}
\Grlmbar = \left\{ \Lambda \in \Gr^{\overline{N\varpi_1}}_\mu \ : \ t \in \End_\C \left(\Lambda_0/\Lambda\right) \text{ has Jordan type } \leq \tau \right\}
\end{equation}

\subsection{Truncated shifted Yangians}
\label{subsec:aff}
Let $Y=Y(\fg)$ be the Yangian of $\fg$.  This is a filtered $\C$-algebra with generators $E_\alpha^{(r)}, F_\alpha^{(r)}, H_i^{(r)}$ for $\alpha \in \Delta^+, i\in I$, $r\in \Z_{>0}$, and filtration defined by $\deg(X^{(r)})=r$ for any generator $X$.  In fact, $Y$ is generated by the elements $E_i^{(r)} := E_{\alpha_i}^{(r)}, F_i^{(r)} := F_{\alpha_i}^{(r)}$ and $H_i^{(r)}$.  For the defining relations see Theorem 3.5 in \cite{KWWY}.  
 
We will frequently work with the formal generating series
$$
E_i(u) = \sum_{r>0} E_i^{(r)} u^{-r}, \ \ F_i(u) = \sum_{r>0} F_i^{(r)} u^{-r}, \ \ H_i(u) = 1+ \sum_{r>0} H_i^{(r)} u^{-r}$$

\begin{definition}[Definition 3.10, \cite{KWWY}]
\label{def: shifted Yangian}
The \textbf{shifted Yangian} $Y_\mu \subset Y$ is the subalgebra generated by $E_i^{(r)}, H_i^{(r)}$ where $r\geq 1$ and $F_i^{(s)}$ where $s>\mu_i$.
\end{definition}  
    
Introduce formal variables $R_i^{(j)}$ where $i\in I$ and $j=1,...,\la_i$, and consider the tensor product of algebras 
\begin{equation}
\label{eq: YmuR}
Y_\mu[R_i^{(j)}] := Y_\mu \otimes_\C \C[R_i^{(j)}:i\in I, j=1,...,\la_i].
\end{equation}  
Let $R_i(u)=\sum_{j=0}^{\la_i}R_i^{(j)}u^{\la_i-j}$, where we denote $R_i^{(0)}=1$. We define $A_i^{(r)} \in Y_\mu[R_i^{(j)}]$ by
\begin{equation}
\label{eq: H from A}
H_i(u) = r_i(u) \frac{ \prod_{j \sim i} A_j(u - \frac{1}{2})}{A_i(u) A_i(u - 1) },
\end{equation}
where $A_i(u) = 1 + \sum_{r>0} A_i^{(r)} u^{-r}$ and
\begin{equation} \label{eq:rfromc}
r_i(u) = u^{-\lambda_i} R_i(u) \frac{ \prod_{j \sim i} (1- \frac{1}{2} u^{-1})^{m_j}}{(1- u^{-1})^{m_i}}.
\end{equation}
See Sections 4.1 in \cite{KWWY} for details.

\begin{definition}[Section 4.4, \cite{KWWY}]
Let $I_\mu^\la$ be the two-sided ideal of $Y_\mu[R_i^{(j)}]$ generated by $A_i^{(r)}$ for $r>m_i$.  The \textbf{truncated shifted Yangian} is the quotient 
$$Y_\mu^\la:= Y_\mu[R_i^{(j)}]/I_\mu^\la$$  
\end{definition}   
The subalgebra $\Gamma_\mu^\lambda \subset Y_\mu^\lambda$ generated by the images of the elements $A_i^{(r)}, R_i^{(j)}$ is commutative.  In fact, it is freely generated by these elements:
\begin{equation}
\Gamma_\mu^\lambda = \C[ A_i^{(r)}, R_i^{(j)}: i\in I, 1\leq r \leq m_i, 1\leq j \leq \lambda_i],
\end{equation}
as follows e.g.~from Corollary \ref{cor: freeness} below.  We call $\Gamma_\mu^\lambda$  the \textbf{Gelfand-Tsetlin subalgebra} of $Y_\mu^\lambda$.

\begin{remark}
\label{rmk: adjoin formal roots}
  In some situations, it will be more convenient to adjoin formal
  roots $\rt_{i,k}$ for the polynomials $R_i(u)=\prod_{k=1}^{\lambda_i}(u-\frac 12
  \rt_{i,k})$.  We denote the resulting algebra $Y_\mu^\la(\brt)$.  This algebra carries an action of the product of symmetric groups $\Theta=\prod_i S_{\lambda_i}$, and $Y_\mu^\lambda $ is the invariant subalgebra.
\end{remark}

\begin{definition}
\label{def: Br}
A {\bf set of parameters} of weight
$\lambda$ is a tuple $\bR=(\bR_i)_{i\in I}$, where $\bR_i$ is a multiset of $\lambda_i$ complex numbers.
\end{definition}

Given a set of parameters of weight $\lambda$, we can specialize the formal variables $R_i^{(j)}$ via:
\begin{equation} \label{eq:defofc}
R_i(u) = \prod_{c\in \bR_i}(u-\tfrac{1}{2} c)
\end{equation}
We denote by $Y_\mu^\la(\bR)$ the corresponding specialized algebra:
$$
Y_\mu^\la(\bR)=Y_\mu^\la\otimes_{\C[R_i^{(j)}]}\C.
$$
Note that $\bR$ determines the {\it roots} of the specialized polynomial $R_i(u)$, and as a consequence we obtain a specialization of the formal variables $R_i^{(j)} \mapsto \C$.  In terms of elementary symmetric functions, we can make this explicit:
$$
R_i^{(j)} \mapsto (-1)^j e_j(\tfrac{1}{2}\bR_i).
$$
This same algebra arises if we number the elements of
$\bR_i$, and specialize $\gamma_{i,k}$ to the corresponding values.
Thus, no statement about the specializations depends on which version
we use, but certain statements about the families will be cleaner for $Y_\mu^\la(\brt)$.

We will denote $\Gamma_\mu^\lambda(\bR) \subset Y_\mu^\lambda(\bR)$ the image of $\Gamma_\mu^\lambda$, so that
\begin{equation}
\Gamma_\mu^\lambda(\bR) = \C[A_i^{(r)} : i\in I, 1\leq r\leq m_i]
\end{equation}
We call $\Gamma_\mu^\lambda(\bR)$ the Gelfand-Tsetlin subalgebra of $Y_\mu^\lambda(\bR)$.

\subsection{Relationship with functions on slices}
\label{sec: yangians and functions on slices}
The main result of \cite{KMWY} is a proof of \cite[Conjecture 2.20]{KWWY} in the case of $\fg = \mathfrak{sl}_n$.  By \cite[Theorem 4.10]{KWWY}, it follows that:

\begin{theorem}
\label{thm: reducedness in type A}
For any choice of $\bR$, there is an isomorphism
$$ \gr(Y_\mu^\la(\bR)) \cong \C[\Gr_\mu^{\overline{\la}}]$$
of graded Poisson algebras.
\end{theorem}
This isomorphism is given explicitly in terms of generalized minors, see \cite[Section 2A]{KWWY} as well as Section \ref{section: more on gr slices} below.

\begin{remark}
The generalization of this theorem to $\fg$ to more general type is obtained in \cite[Appendix B]{BFN+}. 
\end{remark}


\subsection{Shifted Yangians and Coulomb branches}
\label{sec: Coulomb}

Braverman, Finkelberg, and Nakajima have recently developed a mathematical theory of Coulomb branches for $3d$ $\cN=4$ gauge theories \cite{BFN}, \cite{BFN+}.  For any pair $(\bG, \bN)$ of a reductive group $\bG$ and its representation $\bN$ (both over $\C$), they associate a moduli space $\mathcal{R}_{\bG, \bN}$ carrying an action of $\bG$, and an action of $\C^\times$ by loop rotation.  They then define a commutative ring $\cA(\bG, \bN) := H_\ast^{\bG} (\mathcal{R}_{\bG, \bN} )$ via a convolution product, and its deformation quantization $\cA_{\hbar}(\bG, \bN) := H_\ast^{\bG \times \C^\times}(\mathcal{R}_{\bG, \bN})$.  The Coulomb branch is defined as the affine scheme $\cM_C(\bG, \bN) := \operatorname{Spec} \cA(\bG, \bN)$.  

For us, the most relevant cases of this construction are for certain {\bf quiver gauge theories}, and more precisely the type $A$ cases.  Letting $\fg = \mathfrak{sl}_n$ and fixing coweights $\lambda, \mu$ as in Section \ref{section:CombData}, we define vector spaces $W_i = \C^{\lambda_i}$ and $V_i = \C^{m_i}$, for all $1\leq i \leq n-1$.  We can then define a pair $(\bG, \bN)$ as follows:
\begin{equation}
\bG = \prod_{i=1}^{n-1} \operatorname{GL}(V_i), \qquad 
\bN = \bigoplus_{i=1}^{n-2} \Hom(V_i, V_{i+1}) \oplus \bigoplus_{i=1}^{n-1} \Hom(W_i, V_i)
\end{equation}
We can also incorporate the ``flavour symmetry'' group $\mathbf{F} = \prod_{i=1}^{n-1} \operatorname{GL}(W_i)$, and define $$\cA_{\hbar}(\bG, \bN; \mathbf{F}) := H_\ast^{\bG\times \mathbf{F} \times \C^\times}(\mathcal{R}_{\bG, \bN}).$$
We summarize relevant results from \cite{BFN+}:
\begin{theorem}[\mbox{\cite[Theorem 3.10 and Corollary B.28]{BFN+}}]
Consider data $(\bG, \bN)$ associated to $\lambda, \mu$ as above.
\begin{enumerate}
\item[(a)] There is an isomorphism of graded Poisson algebras 
$$
\C[\Gr_\mu^{\overline{\la}}] \cong \cA(\bG, \bN)
$$
In particular, $\Gr_\mu^{\overline{\la}} \cong \cM_C(\bG, \bN)$ is a Coulomb branch.

\item[(b)] The above isomorphism lifts to an isomorphism of filtered algebras
$$
Y_\mu^\lambda \cong \cA_{\hbar = 1}(\bG, \bN; \mathbf{F}),
$$
which identifies the subalgebras $\Gamma_\mu^\lambda \cong H_{\bG \times \mathbf{F}}^\ast(pt)$.
\end{enumerate}
\end{theorem}

More precisely, in the isomorphism (b) the elements $A_i^{(r)}$ correspond to generators of the equivariant cohomology ring $H_{\operatorname{GL}(V_i)}^\ast(pt)$, and the elements $R_i^{(j)}$ to generators of $H_{\operatorname{GL}(W_i)}^\ast(pt)$. Since $\mathcal{R}_{\bG, \bN}$ is equivariantly formal \cite[Section 2]{BFN}, it follows that $\cA_{\hbar=1}(\bG, \bN; \mathbf{F})$ is free over $H_{\bG \times \mathbf{F}}^\ast(pt)$ as a left module (and also as a right module). Thus we deduce:

\begin{corollary}
\label{cor: freeness}
$Y_\mu^\lambda$ is free as a left (or right) module over $\Gamma_\mu^\lambda$, and $Y_\mu^\lambda(\bR)$ is free as a left (or right) module over $\Gamma_\mu^\lambda(\bR)$.
\end{corollary}
This modest application of the theory of Coulomb branches will allow
us to deduce an analogous freeness result for $W$--algebras, see
Corollary \ref{cor: GT free} below. We will use this connection more
intensively in further work on the representation theory of these
algebras \cite{KTWWY2,WWY2}.

\subsection{Highest weights and product monomial crystals}
\label{sec: monomial crystal}
Consider a module $M$ over the algebra $Y_\mu^\lambda(\bR)$.  We call a vector $\mathbf{1} \in M$ a {\bf highest weight vector} if it generates $M$ and 
$$ H_i^{(r)} \mathbf{1} \in \C \mathbf{1}, \qquad E_i^{(r)} \mathbf{1} = 0, \qquad \forall i\in I, r> 0$$
It follows that the series $H_i(u)$ acts on $\mathbf{1}$ by multiplication by some series 
$$J_i(u) = \sum_{r\geq 0 } J_i^{(r)} u^{-r} \in 1 + u^{-1} \C[[u^{-1}]]$$
We call the tuple $J = (J_i(u))_{i\in I}$ the {\bf highest weight} of $M$.

Conversely, given a tuple $J = (J_i(u))_{i\in I}$ of series as above, there is a universal highest weight module $M(J)$ for $Y_\mu^\lambda(\bR)$ (also called a Verma or standard module).  It is generated by a highest weight vector $\mathbf{1}$ with highest weight $J$, and has a unique simple quotient $L(J)$.  The collection of all tuples $J$ such that $M(J) \neq 0$ (equivalently, $L(J) \neq 0$) is called the {\bf set of highest weights} for $Y_\mu^\lambda(\bR)$.

\subsubsection{The product monomial crystal}
The highest weights of $Y_\mu^\lambda(\bR)$ can be classified in terms of the weight $\mu^\ast=-w_0\mu$ elements of the {\bf product monomial crystal} $\cB(\bR)$, where here $w_0 \in S_n$ is the longest permutation.  In this section we briefly overview $\cB(\bR)$ and its relation to highest weights in general. We then give a combinatorial model of $\cB(\bR)$ in type A, using partitions.

\begin{remark}
In this paper we will not make use of the crystal structure on $\cB(\bR)$.  Rather, we will focus on its underlying set.  We refer the reader to \cite[Section 2]{KTWWY} for further details regarding the crystal $\cB(\bR)$.  
Note that in \cite{KTWWY} the product monomial crystal is denoted $\cB(\lambda,\bR)$.
\end{remark}

$\cB(\bR)$ is a subset of the set Laurent monomials in variables $y_{i, c}$ (the ``Nakajima monomial crystal''), where $i\in I, c\in \C$ (although strictly speaking it is only a $\fg$--crystal when the parameters $\bR$ are ``integral'', see Section \ref{section: pmc in type A}).   To define $\cB(\bR)$, one first defines the fundamental monomial crystals $\cB(y_{i, c})$, corresponding a fundamental weight $\varpi_i$ and parameter $c\in \C$.  It is generated by the monomial $y_{i, c}$ by applying Kashiwara operators. For any $c\in \C$, $\cB(y_{i, c})$ is isomorphic to  the fundamental $\fg$--crystal of highest weight $\varpi_i$.

Next, the general product monomial crystal is defined by multiplying together the elements of various fundamental crystals $\cB(y_{i, c})$:
\begin{equation}
\label{eq: monomial crystal}
\cB(\bR) = \prod_{i\in I, c\in \bR_i} \cB(y_{i, c}) := \Big\{ p = \prod_{i\in I, c\in \bR_i} p_{i, c} : \forall i,c, \  p_{i, c} \in \cB(y_{i, c}) \Big\}.
\end{equation}
Here, the product symbol does not signify Cartesian product, but rather the usual product in $\C[y_{i,c}^\pm]$.  
\begin{remark}
Note that with our conventions (\ref{eq: coweight data}), $\lambda = \varpi_i$ corresponds to $\lambda_{n-i} = 1$ and $\lambda_j = 0$ for $j\neq {n-i}$.  In particular a corresponding set of parameters $\bR$ consists of a singleton, namely $\bR_{n-i} = \{c\}$, and $\cB(\bR)$ is isomorphic to the fundamental $\fg$ crystal of highest weight $\varpi_{n-i}$.

We've chosen to follow the conventions of \cite{KWWY}, which differ from those of \cite{KTWWY} by a diagram automorphism.  We pay for this choice here, since $\cB(\bR) \cong \cB(\lambda^*,\bR)$, where $\cB(\lambda^*,\bR)$ is the product monomial crystal as defined in \cite{KTWWY}.  We'll gain from this choice later on, since the formulation of our main results is cleaner with this convention.
\end{remark}

The weight of a monomial is defined as follows:
$$ \text{wt} \big( \prod_{i, k} y_{i,k}^{a_{i, k}} \big) = \sum_{i, k} a_{i,k} \varpi_i $$
where $i\in I, k\in \C$, and only finitely many of the multiplicities $a_{i, k}\in \Z$ are non-zero.  We denote the elements of weight $\mu$ by $\cB(\bR)_\mu$.

For any $i\in I, k\in \C$, define the monomial
$$ z_{i,k} = \frac{y_{i, k} y_{i, k+2}}{\prod_{j\sim i} y_{j, k+1}} $$
Any element $p\in \cB(\bR)$ can be written in the form
\begin{equation}
\label{eq: factorization of monomial}
p = y_\bR z_{\bS}^{-1} := \prod_{i\in I, c\in \bR_i} y_{i, c} \prod_{i\in I, k \in \bS_i} z_{i, k}^{-1},
\end{equation}
for a unique tuple of multisets $\bS = (\bS_i)_{i\in I}$ (where products are taken with multiplicity).  See Section 2 of \cite{KTWWY} for more details.

\subsubsection{Connection to highest weights}
\label{section: monomials to hw}
As described in \cite[Section 3.6]{KTWWY}, elements of $\cB(\bR)_{\mu^*}$ correspond to highest weights for $Y_{\mu}^{\lambda}(\bR)$.  More precisely, a monomial $p = \prod_{i, k} y_{i, k}^{a_{i,k}} $ corresponds to the series
\begin{equation}
\label{eq: monomial to highest weight}
J_i(u) := u^{-\mu_i} \prod_k (u - \tfrac{1}{2} k)^{a_{i, k}},
\end{equation}
where the rational function on the right-hand side is expanded as an element of $1+ u^{-1}\C[[u^{-1}]]$.  

\begin{theorem}[\mbox{\cite[Theorem 1.3]{KTWWY}}]
\label{conj: KTWWY hw conjecture}
The correspondence (\ref{eq: monomial to highest weight}) defines a bijection between $\cB(\bR)_{\mu^*}$ and the set of highest weights for $Y_\mu^\lambda(\bR)$.
\end{theorem}

\begin{remark}
In \cite{KTWWY2} we  show that this theorem for general Lie type using a
presentation for the Yangian based on its connection to Coulomb
branches, as described in \cite{BFN+}.
\end{remark}

We note that if we write $p = y_\bR z_\bS^{-1}$, the tuple of multisets $\bS = (\bS_i)_{i\in I} $  encodes the action of the elements $A_i^{(r)} \in Y_\mu^\lambda(\bR)$ on a highest weight vector $\mathbf{1}$ of weight $p$:
\begin{equation}
\label{eq: monomial to highest weight 2}
A_i(u) \mathbf{1} = \prod_{k \in \bS_i} (1- \tfrac{1}{2} k u^{-1}) \mathbf{1} 
\end{equation}

\subsubsection{Monomials and partitions}
\label{section: pmc in type A}

There is an alternate description of $\cB(\bR)$ and its combinatorics in terms of tuples of Young diagrams \cite[Section 6.2]{KTWWY} which we'll now explain.  This will be used in Section \ref{section: bijection for general lambda}.

We call a set of parameters $\bR$ \textbf{integral} if for every $i$, $\bR_i$ consists of integers, and moreover, the parity of the elements in $\bR_i$ equals the parity of $i$.  In this case, there is a $\fg$--crystal structure on $\cB(\bR)$.  For arbitrary $\bR$ we can decompose each $\bR_i$ into equivalence classes
$
\bR_i=\bigcup_{\zeta\in\C/2\Z}\bR_i(\zeta),
$   
where $\bR_i(\zeta)=\{c\in\bR_i \hspace{1mm}|\hspace{1mm} c - \zeta \in 2\Z+ i \}$.  We let $\bR=\bigcup_{\zeta}\bR(\zeta)$ be the corresponding decomposition of $\bR$.  

As sets we have that $\cB(\bR) \cong \bigotimes_{\zeta}\cB(\bR(\zeta))$; we can put a $\fg\oplus\cdots\oplus \fg$--crystal structure here, with a copy of $\fg$ acting independently on each equivalence class $\cB(\bR(\zeta))$.  Therefore, to describe $\cB(\bR)$ it suffices to describe each $\cB(\bR(\zeta))$.  Moreover, $\cB(\bR(\zeta)) \cong \cB(\bR(\zeta)-\zeta)$, and hence we can confine ourselves to the case where $\bR$ is integral.  

First let us describe the case of a fundamental crystal $\cB(y_{i, c})$, where $c \equiv i \text{ mod } 2$.  As a set, it is in bijection with the collection of Young diagrams which fit into an $i\times (n -i)$ box.  We picture this by placing the Young diagrams in a skew-grid.  The vertices of the skew-grid are labelled by pairs $(i,\ell)$, where $i\in I$ and $\ell\equiv i \text{ mod } 2$.  The $i\times (n -i)$ box is placed in the grid with its top vertex at the point $(i,c)$.

For example if $n=7, i=3$, and $c=5$ and the Young diagram is $(4,2)$ then we have the following picture.  Here we've circled the vertex $(3,5)$, the $i\times (n-i)$ box is inscribed in blue, and the Young diagram is depicted by placing $1's$ in its boxes:

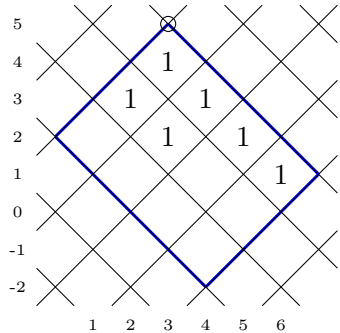
\begin{equation*}
\begin{tikzpicture}[scale=0.5]

{\color{blue}
\draw[line width = 0.04cm] (4,7)--(1,4);
\draw[line width = 0.04cm] (1,4)--(5,0);
\draw[line width = 0.04cm] (5,0)--(8,3);
\draw[line width = 0.04cm] (8,3)--(4,7);
}


\draw (0.5,5.5)--(2.5,7.5);
\draw (0.5,3.5)--(4.5,7.5);
\draw (0.5,1.5)--(6.5,7.5);

\draw (0.5,-0.5)--(8.5,7.5);
\draw (2.5,-0.5)--(8.5,5.5);
\draw (4.5,-0.5)--(8.5,3.5);
\draw (6.5,-0.5)--(8.5,1.5);

\draw (1.5,-0.5)--(0.5,0.5);
\draw (3.5,-0.5)--(0.5,2.5);
\draw (5.5,-0.5)--(0.5,4.5);
\draw (7.5,-0.5)--(0.5,6.5);
\draw (8.5,0.5)--(1.5,7.5);
\draw (8.5,2.5)--(3.5,7.5);
\draw (8.5,4.5)--(5.5,7.5);
\draw (8.5,6.5)--(7.5,7.5);

\draw (4,7) circle(0.2cm);

\draw node at (4,6) {1};
\draw node at (3,5) {1};
\draw node at (4,4) {1};
\draw node at (5,5) {1};
\draw node at (6,4) {1};
\draw node at (7,3) {1};

\draw node at (2,-1) {\tiny 1};
\draw node at (3,-1) {\tiny 2};
\draw node at (4,-1) {\tiny 3};
\draw node at (5,-1) {\tiny 4};
\draw node at (6,-1) {\tiny 5};
\draw node at (7,-1) {\tiny 6};

\draw node at (0,0) {\tiny -2};
\draw node at (0,1) {\tiny -1};
\draw node at (0,2) {\tiny 0};
\draw node at (0,3) {\tiny 1};
\draw node at (0,4) {\tiny 2};
\draw node at (0,5) {\tiny 3};
\draw node at (0,6) {\tiny 4};
\draw node at (0,7) {\tiny 5};
\end{tikzpicture}
\end{equation*} 

To associate a monomial to such a picture, we multiply $y_{i,c}$ by $z_{j,\ell}^{-1}$, as $(j,\ell)$ ranges over the coordinates of the bottom vertices of all the boxes in the partition.  For example, the diagram above corresponds to the monomial
$$ y_{3, 6} z_{3,4}^{-1} z_{2,3}^{-1} z_{4, 3}^{-1} z_{3, 2 }^{-1} z_{5,2}^{-1} z_{6, 1}^{-1}\in \cB(y_{3,6})$$  The rest of the elements of $\cB(y_{3,6})$ correspond to the other partitions fitting into the blue box.
 
 In general suppose $\bR$ is any integral set of parameters. Then elements of $\cB(\bR)$ are identified with diagrams consisting of circled vertices and numbered boxes.  The circled vertices correspond to the elements of $\bR$: for every $c\in \bR_i$ we circle the vertex at $(i,c)$.  If $c\in \bR_i$ occurs with multiplicity then the vertex is circled multiple times.  
 
 Such a diagram corresponds to an element of $\cB(\bR)$ if and only if it can be decomposed into a tuple of overlayed partitions.  More precisely, we must be able to place partitions at each circled vertex on the grid, in such a way that the number in a given box counts the times that box appears in a partition.  Note that a choice of such partitions may not be unique.      

For example, consider the case where $\fg=\mathfrak{sl}_9$ and we take $\bR_3=\{3,5,5\},\bR_5=\{5\}, \bR_6=\{2,4\}$, and $\bR_7=\{5\}$.  The left picture below depicts a candidate element of $\cB(\bR)$.  To check that it is  an element of $\cB(\bR)$ we must be able to place partitions at the circled vertices so that the number in each box counts the number of partitions that contain it.
The right picture depicts such a choice of partitions, verifying that this diagram is indeed in $\cB(\bR)$ .  Note  that since $5$ occurs twice in $\bR_3$ we are able to place two partitions at $(3,5)$.  
 
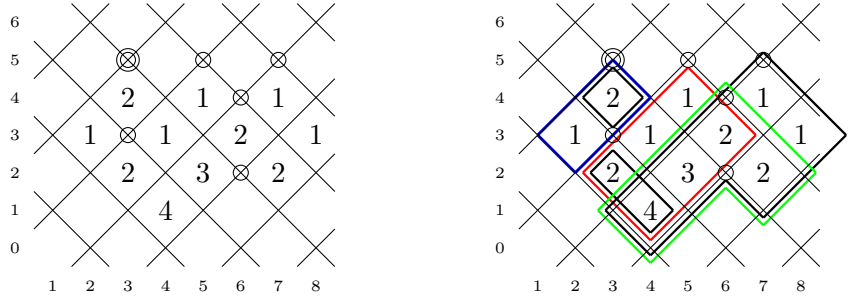
\begin{equation*}
\begin{tikzpicture}[scale=0.5]

%
%

%

%

\draw (0.5,5.5)--(2.5,7.5);
\draw (0.5,3.5)--(4.5,7.5);
\draw (0.5,1.5)--(6.5,7.5);

\draw (1.5,0.5)--(8.5,7.5);
\draw (3.5,0.5)--(8.5,5.5);
\draw (5.5,0.5)--(8.5,3.5);
\draw (7.5,0.5)--(8.5,1.5);

\draw (2.5,0.5)--(0.5,2.5);
\draw (4.5,0.5)--(0.5,4.5);
\draw (6.5,0.5)--(0.5,6.5);
\draw (8.5,0.5)--(1.5,7.5);
\draw (8.5,2.5)--(3.5,7.5);
\draw (8.5,4.5)--(5.5,7.5);
\draw (8.5,6.5)--(7.5,7.5);

\draw (3,6) circle(0.2cm);
\draw (3,6) circle(0.3cm);
\draw (3,4) circle(0.2cm);
\draw (5,6) circle(0.2cm);
\draw (7,6) circle(0.2cm);
\draw (6,5) circle(0.2cm);
\draw (6,3) circle(0.2cm);

\draw node at (2,4) {1};
\draw node at (3,3) {2};
\draw node at (4,2) {4};

\draw node at (3,5) {2};
\draw node at (4,4) {1};
\draw node at (5,3) {3};

\draw node at (5,5) {1};
\draw node at (6,4) {2};

\draw node at (7,5) {1};

\draw node at (8,4) {1};
\draw node at (7,3) {2};

\draw node at (1,0) {\tiny 1};
\draw node at (2,0) {\tiny 2};
\draw node at (3,0) {\tiny 3};
\draw node at (4,0) {\tiny 4};
\draw node at (5,0) {\tiny 5};
\draw node at (6,0) {\tiny 6};
\draw node at (7,0) {\tiny 7};
\draw node at (8,0) {\tiny 8};

\draw node at (0,1) {\tiny 0};
\draw node at (0,2) {\tiny 1};
\draw node at (0,3) {\tiny 2};
\draw node at (0,4) {\tiny 3};
\draw node at (0,5) {\tiny 4};
\draw node at (0,6) {\tiny 5};
\draw node at (0,7) {\tiny 6};

\end{tikzpicture} 
\hspace{2cm}
\begin{tikzpicture}[scale=0.5]

\draw[line width = 0.03cm] (3,5.8)--(3.8,5);
\draw[line width = 0.03cm] (3,5.8)--(2.2,5);
\draw[line width = 0.03cm] (3,4.2)--(3.8,5);
\draw[line width = 0.03cm] (3,4.2)--(2.2,5);

{\color{blue}
\draw[line width = 0.04cm] (3,6)--(4,5);
\draw[line width = 0.04cm] (4,5)--(2,3);
\draw[line width = 0.04cm] (2,3)--(1,4);
\draw[line width = 0.04cm] (1,4)--(3,6);
}

{\color{red}
\draw[line width = 0.03cm] (5,5.8)--(6.8,4);
\draw[line width = 0.03cm] (6.8,4)--(4,1.2);
\draw[line width = 0.03cm] (4,1.2)--(2.2,3);
\draw[line width = 0.03cm] (2.2,3)--(5,5.8);
}

\draw[line width = 0.03cm] (7,6.2)--(9.2,4);
\draw[line width = 0.03cm] (9.2,4)--(7,1.8);
\draw[line width = 0.03cm] (7,1.8)--(6,2.8);
\draw[line width = 0.03cm] (6,2.8)--(4,0.8);
\draw[line width = 0.03cm] (4,0.8)--(2.8,2);
\draw[line width = 0.03cm] (2.8,2)--(7,6.2);

\draw[line width = 0.03cm] (3,3.6)--(4.6,2);
\draw[line width = 0.03cm] (4.6,2)--(4,1.4);
\draw[line width = 0.03cm] (4,1.4)--(2.4,3);
\draw[line width = 0.03cm] (2.4,3)--(3,3.6);

{\color{green}
\draw[line width = 0.03cm] (6,5.4)--(8.4,3);
\draw[line width = 0.03cm] (8.4,3)--(7,1.6);
\draw[line width = 0.03cm] (7,1.6)--(6,2.6);
\draw[line width = 0.03cm] (6,2.6)--(4,0.6);
\draw[line width = 0.03cm] (4,0.6)--(2.6,2);
\draw[line width = 0.03cm] (2.6,2)--(6,5.4);

}

\draw (0.5,5.5)--(2.5,7.5);
\draw (0.5,3.5)--(4.5,7.5);
\draw (0.5,1.5)--(6.5,7.5);

\draw (1.5,0.5)--(8.5,7.5);
\draw (3.5,0.5)--(8.5,5.5);
\draw (5.5,0.5)--(8.5,3.5);
\draw (7.5,0.5)--(8.5,1.5);

\draw (2.5,0.5)--(0.5,2.5);
\draw (4.5,0.5)--(0.5,4.5);
\draw (6.5,0.5)--(0.5,6.5);
\draw (8.5,0.5)--(1.5,7.5);
\draw (8.5,2.5)--(3.5,7.5);
\draw (8.5,4.5)--(5.5,7.5);
\draw (8.5,6.5)--(7.5,7.5);

\draw (3,6) circle(0.2cm);
\draw (3,6) circle(0.3cm);
\draw (3,4) circle(0.2cm);
\draw (5,6) circle(0.2cm);
\draw (7,6) circle(0.2cm);
\draw (6,5) circle(0.2cm);
\draw (6,3) circle(0.2cm);

\draw node at (2,4) {1};
\draw node at (3,3) {2};
\draw node at (4,2) {4};

\draw node at (3,5) {2};
\draw node at (4,4) {1};
\draw node at (5,3) {3};

\draw node at (5,5) {1};
\draw node at (6,4) {2};

\draw node at (7,5) {1};

\draw node at (8,4) {1};
\draw node at (7,3) {2};

\draw node at (1,0) {\tiny 1};
\draw node at (2,0) {\tiny 2};
\draw node at (3,0) {\tiny 3};
\draw node at (4,0) {\tiny 4};
\draw node at (5,0) {\tiny 5};
\draw node at (6,0) {\tiny 6};
\draw node at (7,0) {\tiny 7};
\draw node at (8,0) {\tiny 8};

\draw node at (0,1) {\tiny 0};
\draw node at (0,2) {\tiny 1};
\draw node at (0,3) {\tiny 2};
\draw node at (0,4) {\tiny 3};
\draw node at (0,5) {\tiny 4};
\draw node at (0,6) {\tiny 5};
\draw node at (0,7) {\tiny 6};

\end{tikzpicture} 
\end{equation*}

To associate a monomial to such a diagram we multiply $y_\bR$ by $z_{j,l}^{-k}$, where $(j,\ell)$ ranges over the bottom vertices of the numbered boxes, and $k$ is the number of the box.  In the example above, the diagram corresponds to the monomial
$$
y_\bR z_{3,3}^{-2}z_{5,3}^{-1}z_{7,3}^{-1}z_{2,2}^{-1}z_{4,2}^{-1}z_{6,2}^{-2}z_{8,2}^{-1}z_{3,1}^{-2}z_{5,1}^{-3}z_{7,1}^{-2}z_{4,0}^{-4} \in \cB(\bR)
$$  

We reiterate that the assumption that $\bR$ is an integral set of parameters is made only for the sake of convenience.  We could set up the same combinatorics for general $\bR$, where we depict elements of $\cB(\bR)$ by tuples of such diagrams, one for each $\zeta \in \C/2\Z$ such that $\bR(\zeta)$ is nonempty.  

\subsection{Maps between truncated shifted Yangians}
\label{subsec: maps between yangians}

Given $\lambda\geq \mu$, recall that we define $N = \sum_i i \lambda_{n-i}$.  Consider a set of parameters $\bR = (\bR_i)_{i\in I}$ of weight $\lambda$, and a set of parameters $\tbR$ of weight $N\varpi_1$.  Note that the latter is prescribed by the single multiset $\tbR_{n-1}$ of size $N$. For this reason, we will abuse of notation and simply identify $\tbR = \tbR_{n-1}$.  

Our goal is to establish the following commutative diagram: 
\begin{equation}
\label{diagram: maps of yangians}
\xymatrix{
Y_\mu \ar[r]^{\phi'} \ar[dr]_{\phi} & Y_\mu^{N \varpi_1} (\tbR) \ar@{-->}[d]^{\phi''} \\
& Y_\mu^\lambda(\bR)   }
\end{equation}
where $\phi, \phi'$ are the (defining) quotient maps.

\begin{theorem}
\label{thm: maps between yangians}
A map $\phi''$ making the above diagram commute exists iff
\begin{equation}
\label{eq: R tilde}
\tbR = \bigcup_{i=1}^{n-1} \big( \bR_i + (n-i-1) \big) \cup \big( \bR_i + (n-i-3)\big) \cup \cdots \cup \big( \bR_i - (n-i-1)\big)
\end{equation}
as a union of multisets.  In this case, $\phi''$ quantizes the inclusion $\Gr^{\overline{\lambda}}_\mu \subset \Gr^{\overline{N\varpi_1}}_\mu$ as a closed Poisson subvariety.
\end{theorem}

The last claim in the theorem simply follows from the form of the identification $\gr Y_\mu^\lambda(\bR) \cong \C[ \Gr^{\overline{\lambda}}_\mu]$.  Indeed as in \cite{KWWY}, for any $\lambda\geq \mu$ the surjection $Y_\mu \rightarrow Y_\mu^\lambda(\bR)$ corresponds to the inclusion $\Gr^{\overline{\lambda}}_\mu \subset \Gr_\mu$ into the opposite cell $\Gr_\mu$.  So (\ref{diagram: maps of yangians}) expresses the inclusions 
\begin{equation}
\xymatrix{
\Gr_\mu &\; \Gr^{\overline{N\varpi_1}}_\mu \ar@{_{(}->}[l] \\
&\; \Gr^{\overline{\lambda}}_\mu \ar@{_{(}->}[u] \ar@{_{(}->}[lu]  }
\end{equation}
We will prove Theorem \ref{thm: maps between yangians} in Section \ref{sec:proof of them maps between yangians} below.  First we record some consequences.

%

When the map $\phi''$ exists, every highest weight module for $Y_\mu^\lambda(\bR)$ pulls-back to a highest weight module for $Y_\mu^{N\varpi_1}(\tbR)$. Recall from Section \ref{section: monomials to hw} that an element of the monomial crystal, expressed in the variables $y_{i,k}$, explicitly encodes the action of the series $H_i(u)$ on a highest weight vector.  Since $H_i(u) \mapsto H_i(u)$ under $Y_\mu^{N\varpi_1}(\tbR)\rightarrow Y_\mu^\lambda(\bR)$, the pull-back of highest weights corresponds to an inclusion of sets $\cB(\bR)_{\mu^\ast} \subset \cB(\tbR)_{\mu^\ast}$.  Slightly more generally, we have:
\begin{lemma}
\label{lemma:crystalinclusion}
Let $\bR, \tbR$ satisfy (\ref{eq: R tilde}).  Then there is an inclusion of sets
$$\cB(\bR) \subset \cB(\tbR) $$
If $\bR$ is integral, then this is an inclusion of crystals.
\end{lemma}
\begin{proof}
The case where $\lambda = \varpi_i$ is analogous to \cite[Lemma 5.31]{KTWWY}, and the general case follows by taking products.  
\end{proof}

\begin{remark}
The above results are analogs of the embedding of $\mathfrak{sl}_n$ representations
$$  (\C^n)^{\otimes \lambda_1}\otimes (\wedge^2 \C^n)^{\otimes \lambda_2} \otimes \cdots \otimes (\wedge^{n-1} \C^n)^{\otimes\lambda_{n-1}}  \subset (\wedge^{n-1} \C^n)^{\otimes N}, $$
and in fact when $\bR$ is sufficiently generic Lemma \ref{lemma:crystalinclusion} can be interpretted as a crystal version of this embedding.
\end{remark}

\begin{corollary}
\label{thm: maps between yangians 2}
When $\phi'': Y_\mu^{N\varpi_1}(\tbR) \rightarrow Y_\mu^\lambda(\bR)$ as above exists, we have a containment
$$
\ker \phi'' \subset \bigcap_{p} \operatorname{Ann} L_p, 
$$
the intersection being over the simple $Y_\mu^{N\varpi_1}(\tbR)$--modules $L_p$ with highest weights $p \in \cB(\bR)_{\mu^\ast} \subset \cB(\tbR)_{\mu^\ast}$.  
\end{corollary}

Defining this map in the case where we consider $R_i(u)$ as a formal
polynomial, rather than specializing to numerical values, is slightly
more complicated.  Of course, Theorem  \ref{thm: maps between
  yangians} shows that we have a homomorphism $Y^{N\omega_1}_\mu\to
Y^\la_\mu$ sending
\begin{equation}
\tilde{R}_{n-1}(u)\mapsto \prod_{i=1}^{n-1}\prod_{k=1}^{n-i}
R_{i}(u-\tfrac{n-i-1}{2}+k-1).\label{eq:R-match}
\end{equation}

Unfortunately, this map is not necessarily surjective;  it is more
convenient to consider the enlarged version where we have a surjective
map 
$$
Y^{N\omega_1}_\mu(\tilde{\brt})\to Y^\la_\mu(\brt)
$$ 
of the algebras from Remark \ref{rmk: adjoin formal roots}. This map is defined by sending the roots of the LHS of \eqref{eq:R-match}
to the roots of RHS (by an arbitrary bijection).

\subsubsection{Proof of Theorem \ref{thm: maps between yangians}}
\label{sec:proof of them maps between yangians}
Recall that we set
\begin{equation}
\label{eq: coweight data 6}
\lambda - \mu =\sum_i m_i \alpha_{n-i}, \ \ N\varpi_1 - \mu = \sum_i m_{i}' \alpha_{n-i}
\end{equation}
In addition denote $N\varpi_1 -\lambda = \sum_i m_i'' \alpha_{n-i}$.  In particular $m_i = m_i' - m_i''$. We note the following:

\begin{lemma}
\label{lemma: no final m}
$$ N \varpi_1 - \lambda = \sum_{i=2}^{n-1} \lambda_{n-i} \Big( (i-1)\alpha_1 + (i-2)\alpha_2 + \ldots + \alpha_{i-1}\Big). $$
Thus, we have that the coefficient $m_1'' = 0$.
\end{lemma}
\begin{proof}
We have $N\varpi_1 - \lambda = \sum_i \lambda_{n-i} ( i \varpi_1 - \varpi_i) $.  Now observe that 
\begin{equation*}
 i \varpi_1 - \varpi_i = (i-1) \alpha_1 + (i-2)\alpha_2 + \ldots + \alpha_{i-1}.\qedhere 
\end{equation*}
\end{proof}

Recall from Section \ref{subsec:aff} that  $Y_\mu^\lambda(\bR) = Y_\mu / \langle A_i^{(r)} : i\in I, r>m_i \rangle $, where $A_i^{(r)} \in Y_\mu$ are defined by
$$
H_i(u) = r_i(u) \frac{A_{i-1}(u - \frac{1}{2})A_{i+1}(u - \frac{1}{2})}{A_i(u) A_i(u - 1) }
$$
with
$$ 
r_i(u) = \frac{R_i(u)}{u^{\lambda_i}}\frac{  (1- \frac{1}{2} u^{-1})^{m_{i-1}+m_{i+1}}}{(1- u^{-1})^{m_i}}
$$

Similarly $Y_\mu^{N\varpi_1}(\tbR) = Y_\mu / \langle \widetilde{A}_i^{(r)}: i\in I, r>m_i'\rangle $, where $\widetilde{A}_i^{(r)} \in Y_\mu$ are defined by
$$
H_i(u) = \widetilde{r}_i(u) \frac{\widetilde{A}_{i-1}(u - \frac{1}{2})\widetilde{A}_{i+1}(u - \frac{1}{2})}{\widetilde{A}_i(u) \widetilde{A}_i(u - 1) }
$$
where 
$$
\widetilde{r}_i(u) = \left( \frac{ \widetilde{R}(u)}{u^N}\right)^{\delta_{i,n-1}} \frac{  (1- \frac{1}{2} u^{-1})^{m_{i-1}'+m_{i+1}'}}{(1- u^{-1})^{m_i'}}
$$

From the definitions, for all $i$ we therefore have an equality in $Y_\mu$:
$$
 r_i(u) \frac{A_{i-1}(u - \frac{1}{2})A_{i+1}(u - \frac{1}{2})}{A_i(u) A_i(u - 1) } = \widetilde{r}_i(u) \frac{\widetilde{A}_{i-1}(u - \frac{1}{2})\widetilde{A}_{i+1}(u - \frac{1}{2})}{\widetilde{A}_i(u) \widetilde{A}_i(u - 1) }
$$
Using the definition of $r_i(u)$ and $\widetilde{r}_i(u)$, for $i=n-1$ we can rewrite this as
\begin{equation}
\label{eq: yangian maps compatibility 1}
\frac{\widetilde{A}_{n-2}(u-\tfrac{1}{2})}{\widetilde{A}_{n-1}(u)\widetilde{A}_{n-1}(u-1)} = \frac{R_{n-1}(u)}{\widetilde{R}(u)} \frac{u^{m_{n-1}''} (u-1)^{m_{n-1}''}}{(u-\tfrac{1}{2})^{m_{n-2}''}} \frac{A_{n-2}(u - \frac{1}{2})}{A_{n-1}(u) A_{n-1}(u - 1) }
\end{equation}
and for $i=1,\ldots, n-2$ as
\begin{equation}
\label{eq: yangian maps compatibility 2}
\frac{\widetilde{A}_{i-1}(u - \frac{1}{2})\widetilde{A}_{i+1}(u - \frac{1}{2})}{\widetilde{A}_i(u) \widetilde{A}_i(u - 1) } = R_i(u) \frac{u^{m_i''} (u-1)^{m_i''}}{(u-\tfrac{1}{2})^{m_{i-1}'' + m_{i+1}''}}  \frac{A_{i-1}(u - \frac{1}{2})A_{i+1}(u - \frac{1}{2})}{A_i(u) A_i(u - 1) }
\end{equation}

\begin{corollary}
\label{cor: existence of f}
There are unique series $f_i(u) \in u^{m_i''}(1+u^{-1} \C[[u^{-1}]])$ such that
$$ \widetilde{A}_i(u) = \frac{f_i(u)}{u^{m_i''}} A_i(u) $$
These satisfy
\begin{equation}
\label{equation: defining f}
R_{n-1}(u) = \frac{\widetilde{R}(u) f_{n-2}(u-\tfrac{1}{2})}{f_{n-1}(u) f_{n-1}(u-1)}, \ \ R_i(u) = \frac{f_{i-1}(u-\tfrac{1}{2}) f_{i+1}(u-\tfrac{1}{2})}{f_i(u) f_i(u-1)}
\end{equation}
for $i = 1,\ldots, n-2$.
\end{corollary}
\begin{proof}
By \cite[Lemma 2.1]{GKLO}, $A_i(u)$ and $\widetilde{A}_i(u)$ must differ by multiplication by an element of $1+u^{-1}\C[[u^{-1}]]$.  The precise form above follows by rearranging (\ref{eq: yangian maps compatibility 1}) and (\ref{eq: yangian maps compatibility 2}).
\end{proof}

\begin{lemma}
\label{lemma: f is polynomial}
$\ker \phi' \subset \ker \phi$ if and only if $f_i(u) \in \C[u]$.
\end{lemma}
\begin{proof}
Assume that 
$$\widetilde{A}_i^{(s)}\in \ker \phi = \langle A_i^{(r)}: i\in I, r>m_i\rangle $$
for all $s > m_i'$. Equating coefficients in $ u^{m_i''} \widetilde{A}_i(u) = f_i(u) A_i(u) $, we see that $f_i(u)$ cannot contain any negative powers of $u$. Indeed, if it did then a non-trivial linear combination of elements $\{A_i^{(1)},\ldots, A_i^{(m_i)}\}$ would be zero in $Y_\mu^\lambda(\bR)$.  But these elements are algebraically independent in $\Gamma_\mu^\lambda(\bR)$.

Conversely, if $f_i(u)$ is a polynomial then $\widetilde{A}_i^{(s)}$ is a linear combination of elements from $\ker \phi$.
\end{proof}

Theorem \ref{thm: maps between yangians} follows from the next result:
\begin{proposition}
\label{prop: parameters}
The map $\phi''$ exists iff the following identities hold:
\begin{align*}
\widetilde{R}(u) &= \prod_{i=1}^{n-1} R_i(u+\tfrac{n-i-1}{2}) R_i(u+\tfrac{n-i-3}{2})\cdots R_i(u-\tfrac{n-i-1}{2}), \\
f_{k}(u-\tfrac{1}{2}) &= \prod_{i=1}^{k-1} R_i(u+\tfrac{k-i-1}{2}) R_i(u+\tfrac{k-i-3}{2})\cdots R_i(u-\tfrac{k-i-1}{2}) 
\end{align*}
for $k=1,\ldots, n-1$.
\end{proposition}
\begin{proof}
Note that $\phi''$ exists if and only if $\ker\phi'\subset \ker \phi$.  Hence if $\phi''$ exists then $f_i(u) $ is a polynomial by Lemma \ref{lemma: f is polynomial}, and it is monic of degree $m_i''$ by Corollary \ref{cor: existence of f}.  Since $m_1''=0$ by Lemma \ref{lemma: no final m}, we know that $f_1(u) =1$. Applying (\ref{equation: defining f}) with $i=1$, we then obtain
$$ R_1(u) = f_2(u-\tfrac{1}{2}) $$
Proceeding by induction on $i$ using (\ref{equation: defining f}), we get the claimed form of $\widetilde{R}(u)$ and $f_i(u)$.

Conversely, if we define $\widetilde{R}(u)$ and $f_i(u)$ by the claimed form above, then (\ref{equation: defining f}) holds, and the $f_i(u)$ are monic polynomials of the correct degree.  By the previous lemma, it follows that $\ker \phi' \subset \ker \phi$.
\end{proof}

\section{Around W-algebras}
\label{section:aroundWalg}

\subsection{Finite W-algebras}
\label{subsection:finiteWalg}

Let $\fg$ be a complex semisimple Lie algebra, and $e\in\fg$ a nilpotent element.   Complete this to an $\mathfrak{sl}_2$-triple $\{f,h,e\}$.  The \textbf{Slodowy slice} is the affine space $\cS=e+\fg^f$, where $\fg^f=\{x\in \fg \mid [x,f]=0\}$.  It naturally inherits a Poisson structure from $\fg\cong\fg^*$ \cite{GG}.  Recall that the symplectic leaves of $\fg$ are the nilpotent orbits $\cO$, and $\cS$ intersects the symplectic leaves transversally.  


We recall now a construction of  finite W-algebras which quantize the
Slodowy slices.  Recall that an $\mathbb{Z}$-grading of $\fg$
\[
\mathfrak{g}=\bigoplus_{i\in\mathbb{Z}}\mathfrak{g}_i.
\]
is called {\bf good} for a nilpotent $e$ if
\begin{enumerate}
\item The operator $\operatorname{ad}(e)$ has degree 2.
\item We have
  $\mathfrak{g}_i\cap \ker \operatorname{ad}(e)=0$ for $i\leq -1$.
\item We have $\mathfrak{g}_i\subset  \operatorname{image}
  \operatorname{ad}(e)$ if $i\geq 1$.  
\end{enumerate}
Note that by a simple
application of $\mathfrak{sl}_2$ representation theory, every nilpotent $e$ has a good grading induced by
considering the weights of $h$.  

For any good grading, the space $\mathfrak{g}_{-1}$ is symplectic with the form \[\langle x,y \rangle =(e,[x,y])=([e,x],y)=(x,[y,e]),\] where $(\cdot,\cdot)$ is the usual Killing form.  This follows from the fact that 
$\operatorname{ad}(e)\colon \mathfrak{g}_{-1}\to \mathfrak{g}_{1}$ is
an isomorphism.   Choose a Lagrangian subspace $\mathfrak{l}\subset \mathfrak{g}_{-1}$ and set 
\begin{equation}
\label{eq: fm from grading}
\mathfrak{m}=\mathfrak{l}\oplus \bigoplus_{i< -1}\mathfrak{g}_i.
\end{equation}
Note that if the grading in question is even (i.e. $\mathfrak{g}_i\neq
0$ implies $i\in 2\Z$) then $\mathfrak{m}=\bigoplus_{i<
  -1}\mathfrak{g}_i$ and we can avoid the choice.
Then $\chi=(e,\cdot):\mathfrak{m}\to \C$ is a character.  Finally, let
$\mathfrak{m}_\chi := \operatorname{span}\{ a-\chi(a):a\in\mathfrak{m} \}$.

Define the \textbf{finite W-algebra} $W(e)=(U(\fg)/U(\fg)\mathfrak{m}_\chi)^{\mathfrak{m}}$.  By the following theorem, this algebra is a quantization of $\cS$.

\begin{theorem}[Theorem 4.1, \cite{GG}]
There is a filtration on $W(e)$ (the Kazhdan filtration) such that $gr(W(e)) \cong \C[\cS]$.  
\end{theorem}

We will be interested in quotients of $W(e)$, called parabolic W-algebras, which quantize the intersection $\cS\cap \overline{\cO}$.  

\subsubsection{Conventions}
\label{section: conventions on pyramids}
We closely follow the conventions of \cite[Section 3]{BK3}, \cite[Section 7]{BK2}, although we do not follow their grading conventions: Brundan and Kleshchev divide their even gradings by two, while we will not.  We will also number the boxes of our pyramid differently.  Let us briefly outline our conventions here.

For $\pi = (p_1 \leq p_2\leq\ldots \leq p_n)$ a partition of $N$, we will consider $\pi$ as a right-justified pyramid with boxes numbered from right to left, top to bottom. For example, $\pi = (2,3, 4)$ will be correspond to
\begin{equation}
\young( ::21,:543,9876)
\end{equation}
We number the columns of $\pi$ from left to right, and rows from top to bottom.

Corresponding to the pyramid $\pi$, we consider the nilpotent element 
$$ e_\pi = \sum_{k,\ell} e_{k\ell},$$
summing over pairs \young(k\ell) of adjacent boxes in $\pi$.  The
grading on $\fg$ is defined by $\deg(e_{k\ell}) = 2(\col(\ell) -
\col(k))$, where $\col(\ell)$ denotes the number of the column
containing \young(\ell).  Finally, the Kazhdan filtration on $U(\fg)$ corresponding to $\pi$ is defined by declaring that 
\begin{equation}
\label{eq: kazhdan filtration factor of 2}
\deg( e_{k\ell}) = 2(\col(\ell) - \col(k) + 1)
\end{equation}
\begin{remark}
In \cite{BK3}, the authors use the convention of \eqref{eq: kazhdan filtration factor of 2} in the introduction, but divide by a factor of 2 in \cite[Section 8]{BK3}, to match the usual filtration on Yangians.
\end{remark}

\subsection{Brundan and Kleshchev's presentation}
\label{section:nilpotent-side}
\subsubsection{Shifted Yangians} In the case where $\fg=\mathfrak{gl}_N$ Brundan and Kleshchev gave a presentation of the W-algebra.
To describe this result we first recall their definition of the shifted Yangians  \cite{BK3}.  Here we work with the $\mathfrak{gl}_n$-Yangian $Y_n$, which is a $\C$-algebra with generators
$E_i^{(r)}, F_i^{(r)}$ for $1\leq i < n$, and $D_i^{(r)}$ for $1\leq
i\leq n$ and $r\geq1$.  

To describe the defining relations of $Y_n$ we follow \cite[Theorem 5.2]{BK2} and introduce generating series 
$
D_i(u)=1+\sum_{r\geq1}D_i^{(r)}u^{-r},
$   
and define $\tilde{D}_i^{(r)}$ via
$$
\sum_{r\geq0}\tilde{D}_i^{(r)}u^{-r}=-D_i(u)^{-1}.
$$
The defining relations of $Y_n$ are as follows:
\begin{align*}
[D_i^{(r)},D_j^{(s)}]&=0, \\
[E_i^{(r)},F_j^{(s)}]&=\delta_{i,j}\sum_{t=0}^{r+s-1}\tilde{D}_i^{(t)}D_{i+1}^{(r+s-1-t)},\\
[D_i^{(r)},E_j^{(s)}]&=(\delta_{i,j}-\delta_{i,j+1})\sum_{t=0}^{r-1}D_i^{(t)}E_{j}^{(r+s-1-t)},\\
[D_i^{(r)},F_j^{(s)}]&=(\delta_{i,j+1}-\delta_{i,j})\sum_{t=0}^{r-1}F_j^{(r+s-1-t)}D_{i}^{(t)},\\
[E_i^{(r)},E_i^{(s+1)}]-[E_i^{(r+1)},E_i^{(s)}]&=E_i^{(r)}E_i^{(s)}+E_i^{(s)}E_i^{(r)},\\
[F_i^{(r+1)},F_i^{(s)}]-[F_i^{(r)},F_i^{(s+1)}]&=F_i^{(r)}F_i^{(s)}+F_i^{(s)}F_i^{(r)},\\
[E_i^{(r)},E_{i+1}^{(s+1)}]-[E_i^{(r+1)},E_{i+1}^{(s)}]&=-E_i^{(r)}E_{i+1}^{(s)},\\
[F_i^{(r+1)},F_{i+1}^{(s)}]-[F_i^{(r)},F_{i+1}^{(s+1)}]&=-F_{i+1}^{(s)}F_i^{(r)},\\
[E_i^{(r)},E_j^{(s)}]&=0 \quad \text{ if } |i-j|>1,\\
[F_i^{(r)},F_j^{(s)}]&=0 \quad \text{ if } |i-j|>1,\\
[E_i^{(r)},[E_i^{(s)},E_j^{(t)}]]+[E_i^{(s)},[E_i^{(r)},E_j^{(t)}]]&=0 \quad \text{ if } |i-j|=1,\\
[F_i^{(r)},[F_i^{(s)},F_j^{(t)}]]+[F_i^{(s)},[F_i^{(r)},F_j^{(t)}]]&=0 \quad \text{ if } |i-j|=1.
\end{align*}

$Y_n$ has a filtration defined as follows \cite[Section 5]{BK3}: inductively define elements $E_{i,j}^{(r)}$, for $1\leq i<j \leq n$ and $r>0$, by $E_{i,i+1}^{(r)} = E_i^{(r)}$ and $E_{i,j}^{(r)} = [E_{i,j-1}^{(r)}, E_{j-1}^{(1)}]$, and similarly $E_{i+1,i}^{(r)} = F_i^{(r)}$ and $E_{j,i}^{(r)}=[F_{j-1}^{(1)}, E_{j-1,i}^{(r)}]$.  Also denote $E_{i,i}^{(r)} = D_i^{(r)}$.  Then the filtration is defined by declaring the elements $E_{i,j}^{(r)}$ to have degree $r$; note that $Y_n$ satisfies a PBW theorem in these elements.

Let $\sigma = (s_{i,j})_{1\leq i,j\leq n}$ be a shift matrix of non-negative integers, meaning that
$$ s_{i,j} + s_{j,k} = s_{i,k} $$
whenever $|i-j| + |j-k| = |i-k|$.

\begin{definition}[Section 2, \cite{BK3}]
The {\bf shifted $\mathfrak{gl_n}$-Yangian} $Y_n(\sigma) \subset Y_n$ is the subalgebra generated by $D_i^{(r)}$ for $r>0$, $E_i^{(r)}$ for $r> s_{i,i+1}$, and $F_i^{(r)}$ for $r>s_{i+1,i}$, with the induced filtration from $Y_n$.
\end{definition}

There is another family of generators for $Y_n(\sigma)$, denoted
$T_{i,j}^{(r)}$ for $1\leq i,j \leq n$ and $r>s_{i,j}$.  See
\cite{BK3} for the definition of these generators as well as their relation to the presentation given above.  For $1\leq i \leq n$ we define the principal quantum minor:
\begin{equation}
\label{eq: quantum minor}
Q_i(u) = \sum_{w\in S_i}(-1)^w T_{w(1),1}(u)\cdots T_{w(i),i}(u-i+1)
\end{equation}
where $T_{i,j}(u) = \delta_{i,j} + \sum_{r>s_{ij}} T_{i,j}^{(r)} u^{-r}$. For our present purposes, the most important relation involving these new generators is the following equation (cf. \cite[Theorem 8.7(i)]{BK2}):
\begin{equation}
\label{eq: D and Q}
D_i(u) = \frac{Q_{i}(u+i-1)}{Q_{i-1}(u+i-1)}
\end{equation}

\begin{remark}There is a subtle point here: the identity
  $D_i(u)=\frac{Q_i(u+i-1)}{Q_{i-1}(u+i-1)}$ is true in the Yangian
  with no shift.  In the shifted Yangian the $T_{ij}^{(r)}$ generators are
  defined using a Gauss decomposition with shifted generators \cite[Section 2.2]{BK}.  Hence
  the $T_{ij}^{(r)}$ in the shifted Yangian are not the same as the
  generators with the same name in the full Yangian.  However, Brown
  and Brundan prove that the quantum minors are in fact the same, so
  the identity is true with the shifted $T_{ij}^{(r)}$ as well
  \cite{BB}.  More precisely, they prove that $Q_n(u)=Q_n^0(u)$, where
  $Q_n^0(u)$ is the quantum determinant corresponding to the Yangian
  with $\sigma=0$, i.e. the full $\mathfrak{gl}_n$-Yangian.  This
  implies that $Q_i(u)=Q_i^0(u)$ for any $i=1,..,n$, using the
  embeddings
  $Y(\mathfrak{gl}_n) \supset Y(\mathfrak{gl}_{n-1}) \supset \cdots$.
\end{remark}

We'll need also the decomposition
\begin{equation}
\label{eq: tensor product SY times center is Y}
Y_n(\sigma)\cong SY_n(\sigma) \otimes Z(Y_n(\sigma)),
\end{equation}
where $Z(Y_n(\sigma))$ is the center and $SY_n(\sigma)$ is the subalgebra of $Y_n(\sigma)$ generated by $H_i^{(r)}$ for $r>0$, $E_i^{(r)}$ for $r>s_{i,i+1}$, and $F_i^{(r)}$ for $r>s_{i+1,i}$.  Here $H_i^{(r)}$ are coefficients of 
$\frac{D_{i+1}(u)}{D_{i}(u)}$ \cite[Section 2.6]{BK}).  The center $Z(Y_n(\sigma))$ is free generated by the coefficients of the series $Q_n(u)$ \cite[Theorem 2.6]{BK}.

\subsubsection{Brundan and Kleshchev's Theorem} 
\label{section: BK theorem}
Let $\pi = (p_1 \leq p_2\leq\ldots \leq p_n)$ be a partition of $N$, and consider the lower-triangular shift matrix $\sigma$ where $s_{i,j} = p_j - p_i$ for $i\leq j$.   Let $W(\pi)$ be the quotient of $Y_n(\sigma)$ by the two-sided ideal generated by the elements $D_1^{(r)}$ for $r> p_1$,
\begin{equation}
W(\pi) = Y_n(\sigma) / \langle D_1^{(r)} : r > p_1 \rangle
\end{equation}
The algebra $W(\pi)$ inherits a filtration from $Y_n(\sigma)$. 

\begin{theorem}[Theorem 10.1, \cite{BK3}]
\label{thm: bk3}
There is an isomorphism of algebras $W(\pi) \cong W(e_\pi)$.  This isomorphism doubles filtered degrees, i.e. $ F_{\leq r} W(\pi) \cong F_{\leq 2r} W(e_\pi)$.
\end{theorem}

We will follow the conventions of \cite[Sections 3.3--3.4]{BK} for the above isomorphism, which differ from \cite{BK3} by a certain automorphism $\eta$. This distinction will only be relevant in Section \ref{sec: completing proof of classical limit}.

\begin{remark}
\label{rmk: degree doubling}
Note that the above degree doubling is harmless: the filtration (\ref{eq: kazhdan filtration factor of 2}) on $W(e_\pi)$ is {\em even}, and so we may safely rescale it removing a factor of two.  This is the approach followed by Brundan and Kleshchev, so in their work no such doubling appears.  We have elected to maintain the factor of two to match standard conventions on the Kazhdan filtration (e.g. \cite[Section 4]{GG}), while also following usual conventions for filtrations of Yangians.
\end{remark}

The commutative subalgebra $\Gamma(\pi) \subset W(\pi)$ generated by the images of the elements $D_i^{(r)}$ is called the \textbf{Gelfand-Tsetlin subalgebra} of $W(\pi)$, following the terminology of \cite{FMO}.  By \cite[Corollary 6.3]{BK3}, this is a polynomial ring
$$
\Gamma(\pi) = \C[ D_i^{(r)} : 1\leq i \leq n, 1\leq r\leq p_i ]
$$
We may also think of $\Gamma(\pi) \subset W(\pi)$ as the subalgebra generated by the centers of the subalgebras in a chain of inclusions $W(\pi_1)\subset \cdots \subset W(\pi_n) = W(\pi)$, see \cite{FMO}.

\begin{remark}
When $\pi = (1,\ldots,1)$, we have $W(\pi) =  U(\mathfrak{gl}_n)$ and $\Gamma(\pi) \subset U(\mathfrak{gl}_n)$ is the usual Gelfand-Tsetlin subalgebra.
\end{remark}

Consider a module $M$ over the algebra $W(\pi)$.  We call a vector $\mathbf{1} \in M$ a \textbf{highest weight vector} if it generates $M$ and  
\begin{align*}
 &D_i^{(r)} \mathbf{1} \in \C \mathbf{1} \text{ for } i=1,...,n, r\geq1, \\
 &E_{i,j}^{(r)} \mathbf{1} = 0 \text{ for } 1 \leq i<j\leq n, r\geq 1.
 \end{align*}
As in Section \ref{sec: monomial crystal}, the \textbf{highest weight} of $M$ is a collection of series whose coefficients record the action of the $D_i^{(r)}$ on $\mathbf{1}$.

Let $Row(\pi)$ be the set of row symmetrized $\pi$-tableaux, i.e. tableau of shape $\pi$ with complex entries viewed up to row equivalence.  A row tableau $T \in Row(\pi)$ encodes a highest weight of $W(\pi)$ via 
\begin{equation}
\label{eq:Whighwts}
(u-i+1)^{p_i}D_i(u-i+1) \mapsto \prod_{a\in T_i}(u+\tfrac{1}{2}a-\tfrac{n}{2}),
\end{equation}
where $T_i$ denotes the $i$-th row of $T$.  Brundan and Kleshchev prove that this describes a bijection between highest weights of $W(\pi)$ and $Row(\pi)$ (\cite[Section 6]{BK}).
Given a multiset $\bR$ of $N$ complex numbers we let $\RowR$ be the set of row tableaux with entries from
$\bR$ (with the same multiplicities).    
%
%

\subsection{Parabolic W-algebras}

We will
require some facts about parabolic W-algebras which may be of some
independent interest.  

In type A parabolic W-algebras quantize the intersection of a Slodowy slice with the closure of a nilpotent orbit.  They arise from Hamiltonian reduction of the
primitive quotients of the universal enveloping algebra. 
These quotients were studied by the first author \cite[\S 2]{WebWO} and  by Losev \cite[\S 5.2]{L}.

\subsubsection{Differential operators on partial flag varieties}
\label{sec:diff-oper-part}

Let $G$ be a reductive complex algebraic group. Given a parabolic $P$, we consider the homogenous space $X=G/P$, and
the universal differential operators on it as a quotient of $U(\fg)$.
Let $\fg\cong \fu_-\oplus \fl\oplus \fu$ be the decomposition of $\fg=Lie(G)$
into a Levi subalgebra, and two complementary radicals, with
$\fp=\fl\oplus \fu$.  

We'll be interested in sheaves of twisted differential operators on
$X$.  See \cite[\S 1-2]{BBJant} for a general discussion of these
rings. Since we wish to consider TDOs over more general rings, let us
give a complete definition.  Fix a commutative $\C$-algebra $S$. 
\begin{definition}
  A filtered  sheaf of algebras
  $\mathscr{D}$,
   \[
  \{0\}=\mathscr{D}_{\leq -1} \subset  \mathscr{D}_{\leq 0} \subset \cdots \subset \mathscr{D},\quad \bigcup_{n\geq 0} \mathscr{D}_{\leq n}=\mathscr{D}
\]
is a {\bf TDO with coefficients in $S$} if there is an isomorphism
of
  graded Poisson algebras $$\operatorname{gr} \mathscr{D}\to
  \operatorname{Sym}^\bullet(\mathcal{T})\otimes S$$ where
  $\mathcal{T}=\mathcal{T}(X)$ is the tangent sheaf of $X$.  The Poisson
  bracket on $\operatorname{Sym}^\bullet(\mathcal{T})\otimes S$ is the
  unique $S$-linear Poisson bracket such that $\{X,Y\}$ is the Lie derivative
  $\mathcal{L}_XY$ for $X$ a vector field and $Y$ an arbitrary tensor.

A {\bf homogeneous TDO} is a TDO equipped with a $G$-equivariant
structure, and a Lie algebra map $\fg\to  \Gamma(X;\mathscr{D}_{\leq 1})$
lifting the action map $\fg\to \Gamma(X;\mathcal{T})$.
\end{definition}


As in \cite{BeBe}, we consider the
sheaf of $\fg$ valued functions $\fg^0 = \fg\otimes \cO_{X}$.  Note that $\fg^0$ is the sheaf of sections of the trivial bundle $X \times \fg$, and we have a short exact sequence of vector bundles
$$
0 \to G \times_P \fp  \to X \times \fg \to G\times_P \fg/\fp \to 0
$$ 
We let $\fp^0$ be the local sections of $G\times_P\fp$, and so we have an exact sequence of sheaves
$$
0 \to \fp^0 \to \fg^0 \to \mathcal{T} \to 0
$$
We consider also the algebra sheaf $U^0=U(\fg^0
)\otimes S=U(\fg)\otimes \cO_X \otimes S$.

Given a character $\gamma\colon \fp\to S$, we consider  the ideal in $\cI^\gamma \subset U^0$
generated by the kernel of the map $U(\fp^0) \otimes S\to \cO_{X}\otimes S$
induced by the character
$\gamma-\rho+\rho_P\colon \fp^0 \otimes S\to  \cO_{X}\otimes S$.  Here $\rho$ is the usual half-sum of positive roots of $G$, and $\rho_P$ is the half-sum of the positive roots of the Levi subgroup $L$. 
In other words $\cI^\gamma$ is generated by $\xi - (\gamma-\rho+\rho_P)(\xi)$, where $\xi \in \fp^0 \otimes S$.  Define $\mathscr{D}_\gamma=U^0/\cI^\gamma$.

We can define a TDO $\mathscr{D}_\gamma$ on $X$ by considering the quotient of $U^0$ by
this ideal, with the obvious homogeneous structure.
\begin{proposition}[\mbox{\cite[Theorem 2.4]{Mibook}}]\label{prop:TDO-characters}
This construction defines a bijection between homogeneous TDOs on $X$ and characters  $\gamma:\fp\to S$.
\end{proposition}


If we choose
$S=\Sym(\fp/[\fp,\fp])$, we can take the
universal character $\iota \colon \fp\to \fp/[\fp,\fp] \subset S$.  We
can consider the section algebra
$A(\fp)=\Gamma(X;\mathscr{D}_\iota)$.  When there is no risk of
confusion, we will simply write $A$.

We'll also consider two other cases: when $S=\C$ and
$\gamma\colon \fp\to \C$ is an honest character, and when
$S=\hSym(\fp/[\fp,\fp])$, the completion of $\Sym(\fp/[\fp,\fp])$ at $0$, and we have
$\gamma+\iota:\fp \to \hSym(\fp/[\fp,\fp])$.  We have the resulting
algebras $A_\gamma(\fp)=\Gamma(X;\mathscr{D}_\gamma)$ and $A_{\gamma+\iota}(\fp)=\Gamma(X;\mathscr{D}_{\gamma+\iota})$, and as above, when there is no risk of confusion we'll write simply $A_\gamma, A_{\gamma+\iota}$.

We always have
that $\gr A_\gamma\cong \C[T^*X]$ with the grading induced by cotangent
scaling.  Note that this shows that 
the algebra $A$ is flat over $\Sym(\fp/[\fp,\fp])$,
since its fibers have constant character for the $\C^\times$-action
(so actually every piece of the order filtration is flat).   This
shows that $A_{\gamma+\iota}$ is flat as well.

Thus, the algebra $A_{\gamma+\iota}$ provides a family
over a regular ring which interpolates between the
generic behavior around $\gamma$, and the specialized behavior at
$\gamma$.  In this case, we let $K$ be the fraction field of $S = \hSym(\fp / [\fp,\fp])$ and
let $\defD_{\gamma}:=\mathscr{D}_{\gamma+\iota}\otimes_SK$ denote the TDO over $S$
associated to $\gamma+\iota$, base changed to $K$.  We let 
$\defA_\gamma= A_{\gamma+\iota}\otimes_SK$.

This last algebra is interesting because it satisfies the appropriate
analogue of the Beilinson-Bernstein theorem for {\it all} $\gamma$,
without any dominance hypothesis.  This should be expected, because
$\gamma$ is always ``generic'' but the sense in which localization
holds generically is subtle, since it is not a Zariski open property.
However, it is easy to check that the original proof of
Beilinson-Bernstein \cite{BeBe} and its extension to the parabolic
case by \cite[2.9]{Kitchen} work over any characteristic 0 field, in
particular over $K$. Here we must interpret ``dominant'' as in \cite[\S
2.6]{Kitchen}: a weight over $K$ is dominant if  for all $i$, its inner product with
$\al_i^\vee$ is not a negative  integer.   The weight
$\gamma+\iota$ is obviously dominant in this sense since this inner
product is never an integer\footnote{The
  papers  \cite{BeBe} and \cite{Kitchen} use opposite sign
  conventions; luckily, this is irrelevant for us since $\gamma+\iota$
is dominant and anti-dominant in this sense, so even if one mixes up the sign
conventions, one will arrive at the correct result.}.  Thus we have that:

\begin{theorem}\label{th:affinity}
    The functor \[\Gamma(X;-)\colon \defD_\gamma\mmod\to
  \defA_\gamma\mmod\] is an equivalence.
\end{theorem}
\excise{\begin{proof}

  Let $S=\Sym(\fp/[\fp,\fp])$, and 
  consider the trivial family $G/P\times \Spec(S)\to \Spec(S)$, and
  $\mathscr{D}_\iota$ as a sheaf of algebras on this family.  
First, we note that localization is an open property on $\Spec(S)$;  the failure of derived equivalence is measured by finitely
many homological degrees of the algebra sheaf 
sheaf $\mathcal{K}$ by \cite[3.3]{KalDEQ}, and of exactness by
flatness of the stalks of the sheaf over the section algebra.  

The statement of the theorem is that localization holds at the generic
point of $\Spec \hSym(\fp/[\fp,\fp])$; it is equivalent to prove it at
the generic point of $\Spec(S) $ and thus at any closed point in
this spectrum, that is at $\gamma+\eta$ for $\eta$ a  character
of $\fp$. For simplicity, assume that $\al_i^\vee(\eta)\in
(0,\epsilon)$ for $\epsilon$ a small positive real number.  
Let $R_P$ be the set of positive roots corresponding to $P$.
Note that 
$\alpha_i^\vee(\gamma+\eta-\rho_P)=
  -1$ 
if $\alpha_i\in R_P$.  If $\alpha_i\notin R_P$, then $
  \alpha_i^\vee( \gamma+\eta-\rho_P)$ is a fixed complex number
  $\alpha_i^\vee( \gamma-\rho_P)$ plus $\alpha_i^\vee(\eta) \in
(0,\epsilon)$. If we choose $\epsilon$ sufficiently small, this will
never be
  a non-negative integer.  Thus, this weight is anti-dominant and regular in the
  sense of \cite[\S 2.6]{Kitchen}. Thus, 
by \cite[2.9]{Kitchen}, we have the desired equivalence.
\end{proof}}

The algebra $A(\fp)$ is not quite an analogue of the universal enveloping
algebra since even in the case of a Borel $\fp=\fb$, we will not
obtain $U(\fg)$, but instead the finite extension 
$A(\fb)=U(\fg)\otimes_{Z(\fg)}U(\fh)$ quantizing the Grothendieck-Springer
resolution.  When we ultimately compare parabolic W-algebras to
Yangians, this algebra matches the larger algebra $Y^\la_\mu(\brt)$
where formal roots of $R_i$ are adjoined, see Remark \ref{rmk: adjoin formal roots}.  

We can identify $Z(\fg)$ as a subalgebra of $U(\mathfrak{h})$ in two
different ways: there is the usual Harish-Chandra homomorphism, which
sends a central element to the Cartan term in its PBW expansion, and
the $\rho$-shifted version of this homomorphism, which identifies
$Z(\fg)$ with $U(\mathfrak{h})^W$, so
the maximal ideal for the orbit of a weight $\la$ is the ideal of
central elements vanishing on the Verma module of highest weight
$\la-\rho$.  We'll usually want to use the latter, but it will be
useful to sometimes have the former; note that either map will give
$A(\fb)=U(\fg)\otimes_{Z(\fg)}U(\fh)$; the question is just one of the
coordinates on $\fh$.
\begin{remark}
\label{rem:matchingcons}
Note that in the case of $\mathfrak{gl}_N$, this matches
the convention of \cite[\S 3.8]{BK}: the elements $Z_N^{(r)}$ are sent
to the degree $r$ elementary symmetric function in the diagonal
elements $e_{i,i}$.   If we identify a dominant weight of
$\mathfrak{gl}_N$ with a partition $\nu_1\geq \cdots \geq \nu_N$ as
usual, then this shift sends it to
$(\nu_1+\frac{N-1}{2},\nu_2+\frac{N-3}{2},\dots,
\nu_N+\frac{1-N}{2})$.
\end{remark}

We have an induced $W$-action on $A(\fb)=U(\fg)\otimes_{Z(\fg)}U(\fh)$ 
trivially on the first tensor factor and is the usual action on
the second (if we use the shifted Harish-Chandra homomorphism).  Thus, we can recover $U(\fg)$ as the invariants of this
action.  

Consider the group $\Norm=N_G(\mathfrak{l})/L$, the normalizer of the
Levi $\mathfrak{l}$ of $\mathfrak{p}$ in $G$ modulo the Levi subgroup
integrating it; since Cartan subalgebras in $\mathfrak{l}$ are unique
up to conjugacy in $L$, we have that $\Norm$ is also the simultaneous
normalizer of $L$ and $H$ modulo $H$.  That is, it is the subgroup of
$W$ normalizing $L$.  

For general $\fp$, we can write
$\fp/[\fp,\fp]=\mathfrak{z}(\mathfrak{l})$ as a quotient of $\fh$, and
thus 
write $A(\fp)$ as a quotient of 
$A(\fb)=U(\fg)\otimes_{Z(\fg)}U(\fh)$, where here we have to be sure
to use the unshifted Harish-Chandra homomorphism (and thus act by the
dot action on $\fh$).  The elements of $\Theta\subset W$
descend to automorphisms of $A(\fp)$ under this map.  

\begin{definition}
  Let $W(0,\fp) = A(\fp)^\Norm \subset A(\fp)$ be the invariant subalgebra.
\end{definition}


\begin{remark}
\label{rem:Theta}
In the type A context of primary interest to us, the
Levi $\mathfrak{l}$ will be the block diagonal matrices with block
sizes given by some composition; the group $\Theta$ will be a product
of symmetric groups permuting the blocks with the same size.  Under
our ultimate match of conventions, the scalars $\la_i$ will be the
number of blocks of size $i$, so $\Theta=\prod S_{\la_i}$.  Note that
this matches the use of $\Theta$ in Remark \ref{rmk: adjoin formal roots}.  
\end{remark}

Note that we always have a surjective map of $S=Sym(\fp/[\fp,\fp])$-algebras $U(\fg)\otimes_\C S\to A(\fp)$ as
proven by Borho and Brylinski \cite[3.8]{BoBrI}.  Since this is a
surjective map, it sends the center $Z(\fg) \otimes_\C S$ to the
center $S$ of $A(\fp)$.  

The map $Z(\fg)\cong U(\mathfrak{h})^W\to S$ is induced by the translation by $\rho_P$, followed by
the obvious projection $\mathfrak{h}\to \fp/[\fp,\fp]$.  That is, the
induced map on spectra sends a character $\gamma$ on $\fp$ to the
$W$-orbit of the 
restriction of $\gamma+\rho_P$ to $\fh$. 

Since the $\Norm$-action is constructed by pushing down the action of
$W$ in $U(\fg)\otimes_{Z(\fg)}U(\fh)$, the image of the natural map
$U(\fg)\to A(\fp)$ is $\Theta$-invariant.   
Though the map $W(0,\fp) \hookrightarrow A(\fp)$ is not surjective, it
becomes so after base change to $\C$:
\begin{lemma}
\label{lem:basechangeiso}
The algebra $A_\gamma(\fp)$ is
naturally isomorphic to the quotient $W(0,\fp)_\gamma$ of $W(0,\fp)$ by the maximal ideal in
$Z(\fg) $ which corresponds to the weight $\gamma+\rho_P$ under the
Harish-Chandra homomorphism.  
\end{lemma}
\begin{proof}
 We have a surjective map  $U(\fg)\otimes_\C S\to A_\gamma$, sending
 every element of $S$ to a scalar by \cite[3.8]{BoBrI}, so $U(\fg)\to A_\gamma$ must be
 surjective, and of course, this factors through the map $W(0,\fp)\to
 A_\gamma$.  Our calculation above of the map $Z(\fg)\to S$ shows that
 the maximal ideal for the weight $\gamma+\rho_P$ is indeed killed by
 this map.  That this gives all elements of the ideal is easily
 checked by considering the associated graded.  
\end{proof}

\subsubsection{Specializing to type A}
\label{sec:typeAdiffops}
For $\mathfrak{gl}_N$, we can take the parabolic subalgebra $\fp$ to consist of  block upper triangular
matrices for some composition $\tau$ of $N$.  In particular, if $\tau=(\tau_1,...,\tau_\ell)$ then the Levi subalgebra $\fl\subset\fp$ is block diagonal matrices where the $j^{th}$ block consists of $\tau_j\times \tau_j$ matrices.   A character $\gamma:\fp \to \C$ is simply an
assignment of a scalar $r_j$ to the $j^{th}$ block for $j=1,...,\ell$.  Given $\gamma$ we define a multiset $\bR_i$ to
be the set of (twice) the values we assign to a block of length $i$:
$$
\bR_i=\{ 2r_j \;|\; \tau_j=i, j=1,...,\ell \}.
$$
Combining these together we obtain a set of parameters $\bR=(\bR_i)_{i\in I}$ (cf. Definition \ref{def: Br}).  (The factor of 2 in the definition of $\bR_i$ is inserted to match the
conventions of Section \ref{section: Grassmannian}.)  

The vector $\rho_P$ is given by 
\[\frac 12(\tau_1-1,\tau_1-3,\dots, -\tau_1+1,\;\dots\;,\tau_\ell-1,\tau_\ell-3,\dots, -\tau_\ell+1)\]
so the weight $\gamma+\rho_P$ is a concatenation of vectors of the
form $\frac 12 (r+i-1,r+i-3,\dots, r-i+1)$ for the different $r\in
R_i$.  The normalizer $\Theta$ acts by permuting these blocks if they have
the same size (cf. Remark \ref{rem:Theta}), so after taking the $\Theta$-invariants of $A(\fp)$, we need only
remember $\bR$.  In other words, given $\bR$ which is compatible with $\fp$ (that is, $|\bR_i|$ equals the number of $i \times i$ blocks in $\fl$), we can choose a $\gamma$ so that $\gamma+\rho_P$ recovers $\bR$ as above.  The corresponding two-sided ideal of $W(0,\fp)$ generated by the maximal ideal of $Z(\mathfrak{gl}_N)$ is independent of the choice of $\gamma$. Thus, we will use $W(0,\fp)_{\bR}$ to denote this quotient of $W(0,\fp)$.  By Lemma \ref{lem:basechangeiso} that natural map $W(0,\fp) \to A(\fp)$ induces an isomorphism $W(0,\fp)_{\bR} \cong A_\gamma(\fp)$.

\begin{remark}
If we replace $GL_N$ with $SL_N$, we simply kill the kernel of
the surjective map $U(\mathfrak{gl}_N)\to U(\mathfrak{sl}_N)$, which
means that $\bR$ would only be well-defined up to
simultaneous translation.  Alternatively, we can think about this in
terms of the unique automorphism of $U(\mathfrak{gl}_N)$ which fixes
$U(\mathfrak{sl}_N)$ and sends $Z_N^{(1)}\mapsto Z_N^{(1)}+k$.  Thus, we have $W(0,\fp)_{\bR}\cong
W(0,\fp)_{\bR+k}$ for any $k\in \C$.
\end{remark}

Note that if $\fp=\fb$ is a Borel, then all blocks are of size $1$ so we
only have $R_1$. We let $U(\fg)_\bR=W(0,\fb)_\bR$.  As discussed
above (cf. Remark \ref{rem:matchingcons}), the quotient  $U(\fg)_\bR$ can be defined by sending
$Z_{N}^{(s)}$ to the scalar $e_s(R_1)$, that is, by sending the formal
polynomial $Z_N(u)\mapsto \prod_{r\in R_1} (u+\nicefrac{r}{2})$.    
Our Harish-Chandra homomorphism calculation shows
that:
\begin{lemma} 
The surjective map $U(\fg) \to W(0,\fp)_{\bR}$ factors  through $U(\fg)_{\tbR}$ where $\tbR$ satisfies the condition of \eqref{eq: R tilde}.
\end{lemma}

\begin{remark}
In this formalism, we can think of the deformation $A_{\gamma+\iota}$
as corresponding to a similar set, where we replace each complex
number $r\in R_i$ with a ``point'' in a formal neighborhood of this point.  
\end{remark}




\subsubsection{Definition of parabolic W-algebras}
\label{sec:parabolic-w-algebras}

Now we consider W-algebra analogues of the algebras defined in the previous section, which will be defined by non-commutative Hamiltonian reduction.  Following the notation from Section \ref{subsection:finiteWalg}, for any module $N$ of a quotient of $U(\fg)\otimes_\C S$, we have an induced $\fm$-action where 
\begin{equation}
\label{eq: m action with character}
m\cdot n=mn-\chi(m)n \text{ for all }m\in \fm, n\in N
\end{equation}
where on the RHS, the action is the module structure.  Let $Q_\pi=W(0,\fp)/W(0,\fp)
  \mathfrak{m}_\chi$ and 
consider the non-commutative Hamiltonian reductions 
\begin{align}\label{def:Ae}
A(e,\fp) &:=\Hom_{A(\fp)}\big(A(\fp)/A(\fp) \mathfrak{m}_\chi,A(\fp)/A(\fp)
\mathfrak{m}_\chi\big)=\big(A(\fp)/ A(\fp)\mathfrak{m}_\chi\big)^{\fm}. \\
  W(e,\fp)&:=\Hom_{W(0,\fp)}(Q_\pi, Q_\pi)=Q_\pi^{\fm}. \nonumber
\end{align}	
The algebra $W(e,\fp)$ is the {\bf parabolic W-algebra}.

We can also obtain
$W(e,\fp)_\gamma,A(e,\fp)_{\gamma+\iota}$ and
$\defA(e,\fp)_{\gamma}$ over
$\C,\operatorname{\widehat{Sym}}(\fp/[\fp,\fp])$ and $K$ respectively by
tensoring $A(e)$ with the appropriate base ring or by Hamiltonian
reduction of the corresponding algebras when $e=0$.  The equivalence
of these descriptions follows from the flatness of $A(e,\fp)$ over
$\Sym(\fp/[\fp,\fp])$, and the fact that $H^i(\fm; Q_\pi)=0$ for
$i>0$. This latter vanishing is proven exactly as in
\cite[Prop. 5.2]{GG}; the argument there only uses that the group $M$
integrating $\fm$ acts freely on the coadjoint orbit through $\chi$, and thus
applies to any algebra with an inner action of $\fm$.
Note that  $W(e,\fb)\cong W(e)$, as defined in Section \ref{subsection:finiteWalg}.  In type A, we can use the
notation $W(e)_{\bR},W(e,\fp)_{\bR}$ as in Section
\ref{sec:diff-oper-part}; as discussed there, these algebras only
depend on $\bR$ up to simultaneous translation.

The algebra $W(e,\fp)_{\gamma}$
is the sections of a quantum structure sheaf on the S3-variety
$\mathfrak{X}^{e}_{\fp}$, as defined in \cite[\S
9.2]{BLPWgco}.    As proven in \cite[Proposition 10]{WebWO}
and \cite[Lemma 5.2.1]{L}, the associated graded of this algebra is
isomorphic to the algebra of global functions on
$\mathfrak{X}^{e}_{\fp}$.  

We can also write $W(e,\fp)$ as a quotient of the finite
$W$-algebra $W(e)\to W(e,\fp)$ by an ideal
$J_\fp$.  This ideal is constructed by considering the kernel $I_\fp$ of the
map $U(\fg)\to W(0,\fp)$ and then applying Losev's lower dagger
operation $J_{\fp}:=(I_\fp)_\dagger$ \cite{LosW}.
Note that this ideal must be prime, since
$W(e,\fp)$ is a domain.   Our aim is to ultimately
understand this ideal, using the geometry of
$X$. 


\begin{remark}
We can make a slightly cleaner statement about the classical limit of
$W(e,\fp)_\gamma$
if the natural map
$T^*X\to \fg^*$ is generically injective.   This is always the case
in type A, but for some parabolics in other Lie algebras it fails; for the
classical groups, a criterion for this property is given by Hesselink
\cite[Theorem 7.1]{Hes}.   In this case, the obvious map induces an isomorphism $\C[\mathfrak{X}^{e}_{\fp}]\cong \C[\mathcal{S}_e\cap
(G\cdot \mathfrak{p}^\perp)]$. 

In the case when the map is injective,  we can therefore think of  $W(e,\fp)_{\gamma}$
as a quantization of $\mathcal{S}_e\cap
(G\cdot \mathfrak{p}^\perp)$.  In particular, in type A, if $\fp$ corresponds to $\tau$ as in Section \ref{sec:typeAdiffops}, then $W(e,\fp)_\gamma$ quantizes the intersection $\mathcal{S}_e\cap \overline{\mathbb{O}_\tau}$, where $\mathbb{O}_\tau\subset  \mathfrak{gl}_N$ is the nilpotent orbit  of type $\tau$.

 As discussed in \cite[Remark 5.2.2]{L}, the issue about the map
$T^*X\to \fg^*$ also manifests in the
  natural map $W(e)_{\gamma}\to W(e,\fp)_{\gamma}$
  failing to be surjective on
  the associated graded for the most obvious filtrations on these
  algebras. 
\end{remark}

\newcommand{\WGamma}{W\Gamma}
Let $
(\defD_\gamma,\mathfrak{m}_\chi)\mmod$ denote the category of sheaves
of $\defD_\gamma$--modules on which the module action of
$\mathfrak{m}_\chi$ integrates to a group action of the unipotent
group  $M$.  Let  
$\mathscr{Q}_{\pi}=\mathscr{D}_{\gamma}/\mathscr{D}_{\gamma}\mathfrak{m}_\chi$.
By  Theorem \ref{th:affinity} we have that
$\mathscr{D}_{\gamma}\otimes_{A_\gamma}-$ is left and right adjoint to the sections
functor $\Gamma$.  Thus, we have that 
\[\WGamma(\mathcal{M})=\Hom(\mathscr{Q}_{\pi}, \mathcal{M})=\Hom (Q_\pi, \Gamma(M) ). \]
Combining this with \cite[Prop. 10]{WebWO}, we
obtain the result:

\begin{corollary}
\label{cor:equiv}
    The functor $\WGamma=\Hom(\mathscr{Q}_{\pi},-)\colon (\defD_\gamma,\mathfrak{m}_\chi)\mmod\to
  \defA(e,\fp)_{\gamma}\mmod$ is an equivalence of categories.
\end{corollary}


\subsubsection{Aside: $B$--algebras}

Given a $\Z$-graded algebra $A$, the \textbf{$B$-algebra} with respect to this grading is the quotient
$B(A)=A^0/\sum_{k\in \Z_{>0}} A^{-k}A^k$ \cite[Section 5.1]{BLPWgco}.  As we'll now recall, this algebra controls aspects of the highest weight theory for $A$. In the context of symplectic duality it also has geometric significance, for example as a cohomology ring by Hikita's conjecture \cite{H} and its extension by Nakajima, see \cite[Sections 1.6 and 8]{KTWWY}.

In particular, if $M$ is an $A$--module and $m \in M$ a `highest-weight' element (i.e. $A^0 m \subseteq \C m$ and $A^k m = 0$ for $k>0$), then there is an induced action of $B(A)$ on the line $\C m$. Conversely, for any homomorphism $B(A)\rightarrow \C$ we get a composed homomorphism 
$$A^{\geq 0} = \bigoplus_{k\geq 0} A^k \twoheadrightarrow B(A) \rightarrow \C,$$
and so an induced $A$--module $A \otimes_{A^\geq 0} \C$ callled a \textbf{standard module}. The element $m\in 1\otimes 1$ is highest weight; this construction is left adjoint to that described above.  

\begin{remark}
\label{rmk: base change of b algebra}
Suppose $S\subset A^0$ is a subalgebra which is central.  Then for any
commutative $S$--algebra $S'$, we can extend the grading to
$A\otimes_S S'$.
There is then a natural isomorphism $B( A \otimes_S S') \cong B(A)\otimes_S S' $.
\end{remark}

If $A$ is a commutative ring, then its $\Z$--grading corresponds to a $\mathbb{G}_m$--action on $\operatorname{Spec}A$.  In this case, $B(A)$ is canonically isomorphic to the coordinate ring of the scheme-theoretic fixed-point locus $(\operatorname{Spec} A)^{\mathbb{G}_m}$.  We can leverage this in the non-commutative situation: when $A$ is the global section ring of a quantized conical symplectic resolution $\mathfrak{M}$, there is an inequality 
\begin{equation}
\label{eq: dimension bound B algebras}
\dim_\C \C[ \mathfrak{M}^{\mathbb{G}_m} ] = \dim_\C B( \C[\mathfrak{M}] ) \geq \dim_\C B(A),
\end{equation}
by \cite[Proposition 5.1]{BLPWgco}.

\subsubsection{Highest weights in type \texorpdfstring{$A$}{A}}
\label{sec:highestweightsA}
We now return to the type A setting, and set $\fg=\mathfrak{gl}_N$.
Fix partitions $\pi,\tau \vdash N$.  We let $e_\pi\in \fg$ be the nilpotent defined in Section \ref{section: conventions on pyramids}, and $\fp \subset \fg$ (resp. $P \subset GL_N$) be the parabolic subalgebra (resp. subgroup) corresponding to $\tau$. We fix also a
character $\gamma\colon \fp\to \C$ and the corresponding set of parameters $\bR$ as
defined in Section \ref{sec:typeAdiffops}.  We set
$W(\pi,\fp)=W(e_\pi,\fp), A(\pi,\fp)=A(e_\pi,\fp),
W(\pi,\fp)_{\bR}=W(e_\pi,\fp)_{\bR}$, etc.  

Assume $w\in W=S_N$ is simultaneously a longest left coset representative
of $W_{\pi}$ and a shortest right coset representative for $W_\tau$.  We
call such a permutation {\bf parabolic-singular}, and let $\PS(\pi,\fp)$
be the set of such permutations.  This set is in bijection with the
set of row-strict tableaux of
shape $\tau$ and type $\pi$, sending a tableau to the unique longest
permutation that sends its row reading word to a weakly ordered one.
Then the transposed version of \cite[Cor. 2.6]{BrO} shows that $\PS(\pi,\fp)$ is
non-empty if and only if
$\tau\leq \pi^t$ in the dominance order.

\excise{\ocom{\begin{lemma}
$\PS(\pi,\fp)$ is
non-empty if and only if
$\tau\leq \pi^t$ in the dominance order. 
\end{lemma}

\begin{proof}
Recall that a weak composition $\kappa=(\kappa_1,...,\kappa_r)$ is a sequence of nonnegative integers.
The length $\ell(\kappa)$ is the number of positive parts of $\kappa$. Given a weak composition $\kappa$ such that $\ell(\kappa)\geq i$,  we define $\kappa-1^i$ to be the weak composition obtained by subtracting one from the first $i$ positive parts of $\kappa$.

Now write $\tau=(\tau_1\geq \tau_2 \geq\cdots)$ and $\pi=(\pi_1\geq\pi_2\geq\cdots)$.  
By definition, any $w \in \PS(\pi,\fp)$ satisfies:
\begin{align*}
&w(1)<\cdots<w(\tau_1), \;w(\tau_1+1)<\cdots<w(\tau_2),\;... \\
&w^{-1}(1)>\cdots>w^{-1}(\pi_1), \;w^{-1}(\pi_1+1)>\cdots> w^{-1}(\pi_2),\;...
\end{align*}
These conditions are satisfiable if and only if 
\begin{align*}
\tau_1 \leq& \ell(\pi) \\
\tau_2 \leq& \ell(\pi-1^{\tau_1}) \\
\tau_3 \leq& \ell(\pi-1^{\tau_1}-1^{\tau_2}) \\
&\vdots
\end{align*}
We leave it as an exercise to check that this is equivalent to  $\tau\leq \pi^t$.
\end{proof}
Should I include more details here?  Admittedly this is still a
  bit opaque, but I don't know how much space we want to waste on this
  lemma.}
\bcom{It's silly to do it ourselves.}  }

Let $W(\pi,\fp)$ have the grading induced by eigenvalues of
$\rho^\vee-\rho_{P}^\vee$, and consider the associated $B$--algebra
$B( W(\pi, \fp))$.
By \cite[5.1]{BLPWgco}, this algebra is finite over $S=Sym(\fp/[\fp,\fp])$.   


We'll be particularly interested in the related algebras
$B(W(\pi,\fp)_{\bR})$ and $B(\defA(\pi,\fp)_{\gamma})$; these are the
base change of $B(A(\pi,\fp))$ to the closed point $\gamma$ and to the
generic point of its formal neighborhood respectively, as in Remark
\ref{rmk: base change of b algebra}.  Thus, given a point $\nu \in \Spec B(\defA(\pi,\fp)_{\gamma})$, we can takes its Zariski closure in $\Spec B(A(\pi,\fp))$ and intersect that with $\Spec B({W}(\pi,\fp)_{\bR})$.  Since $\Spec B(A(\pi,\fp))$ is finite and thus proper over $S$, this intersection will be a single point, which we call its {\bf specialization}.   For a general finite map, we could have points in $\Spec B({W}(\pi,\fp)_{\bR})$ which are not the specialization of a more generic point, but this will not happen if $B(W(\pi,\fp))$  is free as a module over $S$ (or equivalently, flat over $S$).  

\begin{lemma}\label{lem:specialize}
The $B$--algebra $B(W(\pi,\fp))$ is free of rank $\PS(\pi,\fp)$ as an $S$-module.
Thus, 
 a weight for $W(\pi,\fp)_{\bR}$ is the highest weight of a module over
  $W(\pi,\fp)_{\bR}$ if and only if it is the specialization of the
  highest weight of a module over $\defA(\pi,\fp)_{\gamma}$.  
\end{lemma}

\begin{proof}
In order to show that an $S$-algebra is free of a given rank, it
suffices to check that it has this rank generically, and that there is
no closed point where the rank of the base change is larger.  

On the one hand, the base change of $B(A(\pi,\fp))$ to the generic point $B\otimes_S K$ has dimension equal to $\# \PS(\pi,\fp)$ by \cite[5.3]{BLPWgco}, since the S3--variety $\fX^{e_\pi}_\fp$ has a torus action with fixed points in bijection with $\PS(\pi,\fp)$.  

On the other hand, by (\ref{eq: dimension bound B algebras}) the dimension of
$B(W(\pi,\fp)_{\bR})$ is bounded above by the ``commutative B-algebra'': the
quotient of $\C[\fX]$ by the ideal generated by functions of non-zero weight. By \cite[A.1 \& 2]{H}, this has dimension equal the
Euler characteristic of another S3 variety, taking the Slodowy slice
to a regular element in the Levi $\fl$ of $\fp$ (which thus has Jordan
type $\tau$ corresponding to the diagonal blocks of $\fl$), and $G/Q$,
where $Q$ is a parabolic with $e$ regular in its Levi (so of type
$\pi$).  Thus $\fp$ and $\pi$ essentially switch roles.  We can obtain
a bijection between $\PS(\pi,\fp)$ and $\PS(\tau,\mathfrak{q})$ by
taking inverse and multiplying by $w_0$.  Note that this requires
reversing the order on blocks of $\tau$ but this order is immaterial,
so this presents no issue.  Note the appearance of the same reversal
in \cite[10.4-5]{BLPWgco}.  Thus $B(W(\pi,\fp)_{\bR})\leq \#\PS(\pi,\fp)$

Since the dimension of the fiber is a lower semi-continuous function,
this shows that this dimension must be constant, and by the usual
argument, $B(A(\pi,\fp))$ must be a free $S$-algebra. 
\end{proof}

Thus, we can use localization to find the highest weights of modules over $\defA(\pi,\fp)_{\gamma}$, and thus over $W(\pi,\fp)_{\bR}$.  Note that the parabolic-singular permutations are precisely the shortest right coset representatives such that $M$ acts freely on the
Schubert cell $NwP/P$, where $N$ is the unipotent of the Borel subgroup $B\subset G$ whose Lie algebra is $\fb$.  These are precisely the Schubert cells that
carry a $(M,\chi)$-equivariant local system $\mathscr{L}_w$
\cite[Definition 12]{WebWO}. That is, $\mathscr{L}_w$ is a coherent
$D$-module, such that the action of $\mathfrak{m}$ on $\mathscr{L}_w$
by $m\cdot x = \alpha_m x -\chi(m)x$ integrates to an $M$-action
(where $\alpha_m$ is the vector field on $X$ given by the
infinitesimal action of $m$), c.f.~equation (\ref{eq: m action with
  character}).  For such $w$, we consider the $D$-modules on $X$
given by $\delta_w=i_!\mathscr{L}_w,\nabla_w=i_*\mathscr{L}_w$, where $i:NwP/P \to G/P$ is the natural inclusion.


Let $\mathfrak{c}=\{h\in \mathfrak{h} \mid [h,e_\pi]=0\}$ be the
centralizer of $e_\pi$ in the Cartan subalgebra $\mathfrak{h}$.  If we
complete $e_\pi$ to an $\mathfrak{sl}_2$ triple, this will also
centralize $f_\pi$.  In fact, the simultaneous centralizer of $e_\pi$
and $f_\pi$ is reductive with $\mathfrak{c}$ a Cartan.

We thus have a natural map $U(\mathfrak{c})\to W(\pi)$, which is a
quantum moment map for the induced action of $C$ on
the Slodowy slice $\mathcal{S}_\pi$.
\begin{lemma}\label{lem:highest-weight}
  The highest weight of $\WGamma(\nabla_w)$ over the torus
  $\mathfrak{c}$ 
  is given by \[w(\gamma+\rho_P)+\rho=ww_0(\gamma-\rho_P)+\rho.\]
\end{lemma}
\begin{proof}
  This is analogous to \cite[12.3.1]{HTT}.  Let $w \in \PS(\pi,\fp)$ be a
  parabolic-singular permutation. The module $\nabla_w$ is a
  pushforward $\iota\colon ww_0Nw_0P/P\hookrightarrow X$, and thus,
  we can just compute the pushforward on this open subvariety.  We can
  identify $ wU_-P/P \cong \operatorname{Ad}_{w}(\mathfrak{u}_-)$,
  with the subvariety $NwP/P$ sent to
  $ \mathfrak{n} \cap \operatorname{Ad}_{w}(\mathfrak{u}_-)=
  \mathfrak{n} \cap \operatorname{Ad}_{w}(\mathfrak{n}_-)$
  since $w$ is a shortest right coset representative.  Now we
  enumerate the roots in $\operatorname{Ad}_{w}(\mathfrak{u}_-)$ by
  $\beta_1,\dots, \beta_N$, with the first $k$ roots
  $\{\beta_1,\dots, \beta_k\} $ being those that are positive. Let
  $x_i$ denote the corresponding coordinates on
  $ \operatorname{Ad}_{w}(\mathfrak{u}_-)$ and $y_i$ the dual basis of
  $\operatorname{Ad}_{w}(\mathfrak{u}_-)$.  We can also assume that
  $\{\beta_1,\dots, \beta_p\}$ for some $p\leq k$ are the (necessarily
  simple) weight spaces on which $\chi$ is non-zero.  The fact that
  $w$ is a longest left coset representative guarantees that these are
  any such weight space lies in
  $\operatorname{Ad}_{w}(\mathfrak{n}_-)$, so the parabolic-singular
  property shows that these lie in
  $\operatorname{Ad}_{w}(\mathfrak{u}_-)$.

Thus, we can identify the pushforward of
  $\mathscr{L}_w$ to this affine space with the module over the Weyl
  algebra $\mathbb{W}=\C [\mathfrak{u}_-]\otimes \Sym
  (\mathfrak{u}_-)$ of $\mathfrak{u}_-$ which is generated by a single element $e^{\chi}$
  with the relations $\frac{\partial}{\partial  x_i}e^{\chi}=\chi(y_i)
  e^{\chi}$ for $i=1,\dots, k$, and $x_i\cdot e^\chi=0$ for
  $i=k+1,\dots, n$.  The function $e^{\chi}$ generates the Whittaker
  functions under multiplication by functions which are constant on
  $M$-orbits and multiplication by constant vector fields.  
In these coordinates, we
have that $\mathfrak{c}$ acts by the Euler operator
\begin{equation}\label{eq:c-action}
h\mapsto w(\gamma+\rho-\rho_P)(h) +\sum_{i=1}^N
\beta_i(h)x_i\frac{\partial}{\partial  x_i}.
\end{equation}

Note that since $\mathfrak{c}$ commutes with $\fm$, we have that
$\beta_i(h)=0$ if $i\leq p$.  
On the function $e^\chi$, we have that
\[x_i\frac{\partial}{\partial  x_i}e^\chi=
\begin{cases}
  \chi(y_i) x_i e^\chi & i\leq p\\
  0 & p<i\leq k\\
  -\beta_i(h) &k<i 
\end{cases}\]
Thus, equation \eqref{eq:c-action} becomes 
\begin{equation}\label{eq:cc-action}
h\cdot e^\chi\mapsto (w(\gamma+\rho-\rho_P)(h) +\sum_{i=k+1}^n
\beta_i(h))e^\chi.
\end{equation}
Note here that $\beta_i$ ranges over the roots in
$\operatorname{Ad}_{w}(\mathfrak{u}_-)\cap
\mathfrak{n}_-=\operatorname{Ad}_{ww_0^P}(\mathfrak{n}_-)\cap
\mathfrak{n}_-$, so the sum is $\rho-ww_0^P\rho=\rho-w\rho+2w\rho_P$.  Thus, we have that
the weight of $e^\chi$ is $w(\gamma+\rho_P)+\rho=ww_0(\gamma-\rho_P)+\rho$.
\end{proof}

We call $\gamma\in \mathfrak{h}^*\otimes_\C S$ {\bf row-sum-distinct} if the
restrictions $w\cdot (\gamma +\rho_P)
  |_{\mathfrak{c}}$  are distinct for  $w\in  \PS(\pi,\fp)$.  Note, this is stronger than
  having stabilizer $W_P$, as the case of $N=4,\pi=(2,2),\fp=\fb$ and
  $\gamma=(4,3,2,1)$ shows: the permutations $(2,4,1,3)$ and $(2,4,3,1)$ are both parabolic-singular, but $w\cdot (\gamma +\rho_P)$ have the same restriction to $\mathfrak{c}=\{\mathrm{diag}(a,a,b,b):a,b\in \C \}$.  For a fixed $P$, this is an open condition determined by
  finitely many hyperplanes.

 In particular, the weight $\gamma +\iota$ for $\gamma\in
\fh^*$ is always row-sum-distinct, since if $w\cdot (\gamma +\iota+\rho_P)
  |_{\mathfrak{c}}=w'\cdot (\gamma +\iota+\rho_P)
  |_{\mathfrak{c}}$ for $w\neq w'\in \PS(\pi,\fp)$, then we must
  have $w\cdot (\gamma +\rho_P)
  |_{\mathfrak{c}}=w'\cdot (\gamma +\rho_P)
  |_{\mathfrak{c}}$ and thus $w\cdot \iota |_{\mathfrak{c}}=w'\cdot
  \iota |_{\mathfrak{c}}$. In this case, $w$ and $w'$ are in the same
  double coset, and thus must be equal (since each double coset
  contains at most one parabolic-singular permutation). 

  This actually shows something
  stronger: the difference $w\cdot (\gamma +\rho_P)
  |_{\mathfrak{c}}-w'\cdot (\gamma +\rho_P)
  |_{\mathfrak{c}}$ is never an integral weight (since it is never a
  complex-valued weight).

Recall that we have fixed a
character $\gamma\colon \fp\to \C$ and the corresponding set of parameters $\bR$ as
defined in Section \ref{sec:typeAdiffops}.
For a given $w\in \PS(\pi,\fp)$, we consider the weight $w\cdot
(\gamma+\rho_P)$;  we let 
\begin{equation}
\label{eq: W to tableaux}
\mathsf{T}_w\in \RowRt
\end{equation}
be the row-symmetrized
tableau of shape $\pi$ which has $  w\cdot
(\gamma+\rho_P)$ as a row reading (this corresponds to a filling in the alphabet $\tbR$).  
We let $L_w$ be the simple module
attached to this tableau by Brundan and Kleshchev (cf. Section \ref{section: BK theorem}).  We let
$\widetilde{\mathsf{T}}_w$ and $\widetilde{L}_w,\widetilde{\nabla}_w$ be corresponding objects for
$\gamma+\iota$, base changed to $K$.
  \begin{lemma}\label{lem:generic-sections}
    We have an isomorphism $\widetilde{L}_w\cong \WGamma(\widetilde{\nabla}_w)$.
  \end{lemma}
  \begin{proof}
    By row-sum-distinctness, each of the simples $\widetilde{L}_w$ have distinct
    highest weights for $\mathfrak{c}$, as do
    $\WGamma(\tilde{\nabla}_w)$.  In fact, since the weights of
    different $\WGamma(\tilde{\nabla}_w)$'s 
    are never congruent modulo integral weights, there are no
    $\mathfrak{c}$-equivariant maps between them, and thus no
    $W(\pi,\fp)$-equivariant ones.   Thus,
    $\WGamma(\tilde{\nabla}_w)$ will be simple, and isomorphic to whichever of the
    modules $\tilde{L}_{w'}$ has the same highest weight for $\mathfrak{c}$.   By
    construction, this is $\tilde{L}_w$. 
  \end{proof}

\begin{theorem}
\label{thm:parabolicWweight}
The highest weights of modules in category
$\cO$ over $W(\pi,\fp)_{\bR}$ are given by the tableaux ${\mathsf{T}}_w$ for
$w\in \PS(\pi,\fp)$.  
\end{theorem}
\begin{proof}
 First we check this for
  $\gamma+\iota$ after base change to $K$.  In this case, the simple
  modules are given by $\tilde{L}_w$.  By Lemma \ref{lem:specialize},
  the simples at $\gamma$ have highest weights obtained by
  specialization, that is, they are $L_w$.  
\end{proof}

\begin{lemma}
  The action of $W(\pi,\fp)_{\bR}$ on $\WGamma(\nabla_w)$ is faithful.
\end{lemma}
\begin{proof}
  The module $\nabla_w$ is a naive pushforward from the open subset
  $ww_0Nw_0P/P\subset X$, so we can show faithfulness on this open subset.  On this open subset, $\nabla_w$ is the
  pushforward of the Whittaker functions on an affine subspace, which is
  faithful.
\end{proof}

A standard argument shows that a faithful module of finite length over a domain must
have a faithful composition factor.  Equivalently, we have that
$J_\fp$ is the intersection of the annihilators of $W(\pi)_{\tbR}$
acting on the composition factors of $\WGamma(\nabla_w)$.  Thus, we
have:
\begin{theorem}\label{thm:parabolic-annihilator}
  The algebra $W(\pi,\fp)_{\bR}$
acts faithfully on at least one $L_w$ for $w\in \PS(\pi,\fp)$; that is \[W(\pi,\fp)_{\bR}\cong
W(\pi)_{\tbR}\big/\bigcap_{w\in \PS(\pi,\fp)} \Ann(L_w).\]  
\end{theorem}


\section{The quantized Mirkovi\'c-Vybornov isomorphism}
\label{section: main theorem}

Throughout this section we let $\fg=\mathfrak{sl}_n$.

\subsection{The main theorem}
\label{section: MV and Slodowy}

Recall our notation from \ref{section:CombData}: $\lambda \geq \mu$ are dominant coweights of $\fg$, and $\tau \geq \pi$ are partitions of $N$.  Let $\fp \mathfrak{gl}_N$ be the parabolic subalgebra corresponding to $\tau$.

Recall from (\ref{eq: def of Grlm}) that for $\Lambda \in  \Gr_\mu^{\overline{\lambda}}$, we have $\Lambda \subset \Lambda_0$ with $\Lambda_0/\Lambda \cong \C^N$ via the basis $E_\pi$.  Furthermore the operator on $\Lambda_0/\Lambda$ induced by multiplication by $t$ is nilpotent of type $\leq \tau$.  Therefore we have a map $\Gr_\mu^{\overline{\lambda}} \to  \overline{\mathbb{O}_{\tau}} \subset \mathfrak{gl}_N$.

\begin{theorem}[\cite{MV08}]
\label{theorem:MV}
The map sending $\Lambda\in \Gr_\mu^{\overline{\lambda}}$ to the action of $t$ on $\Lambda_0/\Lambda\cong \C^N$ defines an isomorphism of varieties 
$$ \Grlmbar \stackrel{\sim}{\longrightarrow} \mathcal{T}_\pi \cap \overline{\mathbb{O}_{\tau}} $$
for a suitable transverse slice $\cT_\pi$ to $\mathbb{O}_\pi \subset \overline{\mathbb{O}_\tau}$.  Moreover, the following diagram commutes:
$$
\xymatrix{
\Gr_\mu^{\overline{N\varpi_1}} \ar[r]^\sim & T_\pi \cap \mathcal{N}_{\mathfrak{gl}_N} \\
\Grlmbar \ar@{^{(}->}[u] \ar[r]^\sim & T_\pi \cap \overline{\mathbb{O}_\tau} \ar@{^{(}->}[u]
}
$$
where the vertical arrows are the inclusions of closed subvarieties.
\end{theorem}

We review this theorem in more detail in Section \ref{sec: the mv isomorphism} below.

\begin{remark}
In this paper, we use a formulation of the above result above due to Cautis-Kamnitzer \cite[Section 3.3]{CK}.  It is somewhat simpler than the original construction of Mirkovi\'c-Vybornov.  It is possible to modify the results of this paper to precisely match the original Mirkovi\'c-Vybornov isomorphism, however this comes at the cost of less pleasant maps of algebras and associated combinatorics.  

\end{remark}


We note that although in general $\mathcal{T}_\pi$ differs from the better-known Slodowy slice, these are isomorphic as Poisson varieties (cf. Section \ref{section: MV slices}).  Therefore the above theorem implies that  $\mathcal{S}_\pi \cap \overline{\mathbb{O}_{\tau}}\cong \Gr_\mu^{\overline{\lambda}}$, where $\cS_\pi$ is the Slodowy slice.  

Now, on the one hand $\Gr_\mu^{\overline{\lambda}}$ is quantized by $Y^{\lambda}_\mu(\bR)$ for any set of parameters $\bR$.  On the other hand, $\mathcal{S}_\pi \cap \overline{\mathbb{O}_{\tau}}$ is quantized by $W(\pi,\fp)_{\bR}$.  Our main result shows that we can lift the isomorphism of Mirkovi\'c and Vybornov to the level of quantizations.

\begin{theorem}\label{thm: main theorem}\hfill
\begin{enumerate}
\item[(a)]
There is an isomorphism of filtered algebras
$$ \Phi: Y_\mu^{N\varpi_1} \stackrel{\sim}{\longrightarrow} W(\pi) $$
compatible with specialization of parameters on both sides.  It identifies the Gelfand-Tsetlin subalgebras $\Gamma_\mu^{N\varpi_1} \stackrel{\sim}{\longrightarrow} \Gamma(\pi)$.

\item[(b)] $\Phi$ induces a bijection between highest weight modules over $Y_\mu^{N \varpi_1}$ and $W(\pi)$, such that
\[\tikz[->,thick]{
\matrix[row sep=10mm,column sep=20mm,ampersand replacement=\&]{
\node (a) {$\cB(\tbR)$}; \& \node (b) {$\RowR$};  \\
\node (d) {$\cB(\bR)$}; \& \node (e) {$\RowRtc$};  \\};
\draw (a) -- node[above,midway]{$\sim$} (b) ; 
\draw (d) edge[draw=none] node [sloped, auto=false,allow upside down]{$\subseteq$} (a) ; 
\draw (e) edge[draw=none] node [sloped, auto=false,allow upside down]{$\subseteq$} (b) ; 
\draw (d) -- node[above,midway]{$\sim$}(e) ; 
}\]
when $\tbR$ and $\bR$ are related as in (\ref{eq: R tilde}).  (Here $\RowRtc$ is a set of row symmetrized $\pi$--tableaux parametrizing highest weights for $W(\pi, \fp)_\bR$, see Section \ref{section: bijection for general lambda})
\item[(c)] $\Phi$ induces an isomorphism of filtered algebras 
$$Y_\mu^\lambda \stackrel{\sim}{\longrightarrow} W(\pi,\fp)$$
It is compatible with specializations of parameters on both sides, yielding isomorphisms $Y_\mu^\lambda(\bR) \stackrel{\sim}{\longrightarrow} W(\pi, \fp)_\bR$ for any set of parameters $\bR$, via the commutative diagram
\[\tikz[->,thick]{
\matrix[row sep=10mm,column sep=20mm,ampersand replacement=\&]{
\node (a) {$Y_\mu^{N\varpi_1}(\tbR) $}; \& \node (b) {$W(\pi)_{\tbR} $};  \\
\node (d) {$Y_\mu^\lambda(\bR) $}; \& \node (e) {$W(\pi, \fp)_\bR$};  \\};
\draw (a) -- node[above,midway]{$\sim$} (b) ; 
\draw[->>] (a) -- (d) ; 
\draw[->>] (b) -- (e) ; 
\draw (d) -- node[above,midway]{$\sim$}(e) ; 
}\]
These isomorphisms identify the Gelfand-Tsetlin subalgebras $\Gamma_\mu^\lambda \stackrel{\sim}{\longrightarrow} \Gamma(\pi, \fp)$ and $\Gamma_\mu^\lambda(\bR) \stackrel{\sim}{\longrightarrow} \Gamma(\pi, \fp)_\bR$.

\item[(d)] The classical limit agrees with the MV isomorphism.
\end{enumerate}
\end{theorem}

We will split the proof of this theorem into parts, which occupy the remainder of the paper.  In Sections \ref{mainsec}, \ref{subsection:Preliminaries} and \ref{sec: classical limit}, we will prove parts (a), (b) and (d) of this theorem, respectively.  Part (c) follows from parts (a), (b) by a simple argument, as we will show presently.

This linkage uses the quotient maps:
\begin{equation}
Y^{N\varpi_1}_\mu\to Y^{\lambda}_\mu\qquad
W(\pi)\to W(\pi,\fp),\label{eq:quotients}
\end{equation}
introduced in Sections \ref{subsec: maps between yangians} and
\ref{sec:parabolic-w-algebras} respectively.  Note that these maps are
not surjective; instead it is better to work with the surjective maps
\begin{equation}
Y^{N\varpi_1}_\mu(\tilde{\brt})\to Y^{\lambda}_\mu(\brt)\qquad A(\pi,\fb)\to A(\pi,\fp),\label{eq:quotients2}
\end{equation}
The subalgebras $Y_\mu^\lambda \subset Y_\mu^\lambda(\brt)$ generates
$Y_\mu^\lambda(\brt)$ over its center, and similarly for $W(\pi, \fb)
\subset  A(\pi, \fb)$. 

From part (a), we can construct an $S_N$-equivariant isomorphism
$Y^{N\varpi_1}_\mu(\tilde{\brt})\cong A(\pi,\fb)$ by base change.
If we show that we have an induced isomorphism
$Y^{\lambda}_\mu(\brt)\cong A(\pi,\fp)$, this will necessarily be
equivariant for the actions of $\Theta$ on $Y^{\lambda}_\mu(\brt)$ by permuting
formal roots as in Remark \ref{rmk: adjoin formal roots} and on
$A(\pi,\fp)$ by the action discussed in Section
\ref{sec:diff-oper-part}.  
We thus obtain the desired
map in Theorem \ref{thm: main theorem}(c) by taking invariants of these $\Theta$-actions.

\begin{proof}[Proof of Theorem \ref{thm: main theorem}(c) (assuming parts (a) and (b))]
  In order to complete the proof, we need only
  show that the kernels in \eqref{eq:quotients} match under $\Phi$
  after base change at these maximal ideals.
  By Theorem \ref{thm:parabolic-annihilator}, the kernel of the latter
  map is the intersection of the annihilators of the simple modules
  over $Y^{\lambda}_\mu(\brt)$.  By Theorem \ref{thm: main theorem}(b) and Corollary \ref{thm: maps between yangians 2}, we thus
  obtain an induced surjective map $Y^{\lambda}_\mu(\brt) \to A(\pi,\fp)$.
  
   Since each piece of the filtration of $Y^{\lambda}_\mu(\brt)$ and
$A(\pi,\fp)$ is finite dimensional, if the map $\Phi$ is not a
filtered isomorphism, it will fail to be an isomorphism after
specialization at a maximal ideal in the center.  Thus, we can
consider the quotient $Y^{\lambda}_\mu(\bR)$ with $\bR$ giving a
maximal ideal of the center, and the corresponding quotient of
$W(\pi,\fp)_{\bR}$. 

When we take associated graded of both sides, we obtain the functions on $\Gr^{\bar \lambda}_\mu$  (Theorem \ref{thm: reducedness in type A}) and $\mathcal{S}_\pi \cap \overline{\mathbb{O}_{\tau}}$, respectively.  
  Both are irreducible varieties of the same dimension, thus a surjective ring map from one to the other must be an isomorphism.
\end{proof}

We next note an immediate corollary about the original Mirkovi\'c-Vybornov isomorphism:

\begin{corollary}
\label{cor:Poissonisom}
$\Gr_\mu^{\overline{\lambda}}$ and $\mathcal{T}_\pi \cap \overline{\mathbb{O}_{\tau}}$ are  isomorphic as Poisson varieties.  This isomorphism intertwines the $\C^\times$--action by loop rotation on $\Gr^{\overline{\lambda}}_\mu$ with the square root of the Kazhdan action on $\cT_\pi \cap \overline{\mathbb{O}_\tau}$ (see Remark  \ref{rmk: degree doubling}).
\end{corollary}

\begin{remark}
In the classical limit, the Gelfand-Tsetlin subalgebras $\Gamma_\mu^\lambda(\bR) \subset Y_\mu^\lambda(\bR)$ and $\Gamma(\pi, \fp)_\bR \subset W(\pi, \fp)_\bR$ become Poisson commutative subalgebras of the coordinate rings of $\Grlmbar$ and $\mathcal{T}_\pi \cap \overline{\mathbb{O}_{\tau}}$, resp.  In other words, they define integrable systems. It follows that the Mirkovi\'c-Vybornov isomorphism also intertwines these integrable systems.
\end{remark}

We note a second immediate corollary, which establishes a generalization of \cite[Conjecture 2]{FMO} for all parabolic W-algebras. By combining part (c) of the main theorem, together with Corollary \ref{cor: freeness}, we deduce:
\begin{corollary}
\label{cor: GT free}
$W(\pi, \fp)$ is free as a left (or right) module over its Gelfand-Tsetlin subalgebra $\Gamma(\pi, \fp)$.
\end{corollary}

\subsection{Proof of Theorem \ref{thm: main theorem}(a): The case of \texorpdfstring{$\lambda = N \varpi_1$}{lambda= N pi1}}
\label{mainsec}
In this section, we will consider the case where $\lambda = N \varpi_1$ and  $\mu$ is a dominant weight such that $\mu \leq \la$.  From this data, we have a partition
$\pi \vdash N$ as in Section \ref{section:CombData}.
We'll describe (Theorem \ref{mainthm}) an isomorphism between $Y_\mu^{N\varpi_1}$ and $W(\pi)$. 

To state the theorem precisely, first we need to define a map
\begin{equation}
\label{eq: main yangian map}
\phi:Y_\mu \to SY_n(\sigma)
\end{equation}
 by    
\begin{eqnarray*}
H_i(u) &\mapsto& \frac{(u+\tfrac{i-1}{2})^{\mu_i}}{u^{\mu_i}}\frac{D_{i+1}(-u-\tfrac{i-1}{2})}{D_{i}(-u-\tfrac{i-1}{2})} \\
E_i(u) &\mapsto& E_{i}(-u-\tfrac{i-1}{2}) \\
F_{\mu,i}(u) &\mapsto& (-1)^{\mu_i} F_{\mu,i}(-u-\tfrac{i-1}{2}) \\
\end{eqnarray*}
for $i \in I$.  Here $F_{\mu, i} (u) = \sum_{r>0} F_i^{(\mu_i+r)} u^{-r}$, on each side.

\begin{proposition}
\label{prop: iso of shifted yangians}
The map $\phi$ is an isomorphism of filtered algebras.
\end{proposition}

For the proof, we will make use of the following lemma regarding ``non-standard'' embeddings of the shifted Yangian $Y_\mu \hookrightarrow Y$:
\begin{lemma}
Fix a monic polynomial 
$$Q_i(u) = u^{\mu_i} + Q_i^{(1)} u^{\mu_i-1}+\ldots + Q_i^{(\mu_i)} \in \C[u]$$
for each $i=1,\ldots, n-1$.  There is a corresponding embedding $Y_\mu \hookrightarrow Y$, defined on the generators by
\begin{align*}
E_i^{(r)} & \mapsto E_i^{(r)}, \\
H_i^{(r)} & \mapsto H_i^{(r)} + Q_i^{(1)} H_i^{(r-1)} +\ldots + Q_i^{(\mu_i)} H_i^{(r-\mu_i)}, \\
F_i^{(s)} & \mapsto F_i^{(r)} + Q_i^{(1)} F_i^{(s-1)} +\ldots + Q_i^{(\mu_i)} F_i^{(s-\mu_i)}
\end{align*}
for all $r>0$ and $s>\mu_i$, and where we interpret $H_i^{(0)} = 1$ and $H_i^{(r)} = 0$ for $r<0$.
\end{lemma}
\begin{proof}
Assuming that this map defines a homomorphism, it is easy to see that it is an embedding: its associated graded agrees with that of the defining embedding $Y_\mu \subset Y$.

To prove that it is a homomorphism, one can verify the relations directly; we give a different argument.  By \cite[Lemma 3.7]{KTWWY}, $Y_\mu$ is a left coideal of $Y$ with respect to its defining embedding $Y_\mu \subset Y$ (see Definition \ref{def: shifted Yangian}).  By \cite[Proposition 3.8]{KTWWY}, there is a 1-dimensional module $\C \mathbf{1}_Q$ for $Y_\mu$ determined by the polynomials $Q_i(u)$. We can then consider 
$$ Y_\mu \stackrel{\Delta}{\longrightarrow} Y \otimes Y_\mu \longrightarrow Y \otimes \End(\C \mathbf{1}_Q) \cong Y $$
The composition is precisely the claimed homomorphism.
\end{proof}

\begin{proof}[Proof of Proposition \ref{prop: iso of shifted yangians}]
When $\mu =0$ the fact that this map defines an isomorphism $Y \stackrel{\sim}{\rightarrow} SY_n$ follows from \cite[Remark 5.12]{BK2}, after a minor modification: here we are following Drinfeld's conventions as opposed to the ``opposite'' presentation of \cite{BK2}.  

When $\mu\neq0$, consider the composition
$$ Y_\mu \hookrightarrow Y \stackrel{\sim}{\longrightarrow} SY_n $$
where the second arrow is the above $\mu =0$ isomorphism, while the first arrow is the embedding from the previous lemma for the polynomials $Q_i(u) = (u+\tfrac{i-1}{2})^{\mu_i}$.  This map $Y_\mu \hookrightarrow SY_n$ agrees with $\phi$ on the generators of $Y_\mu$, and its image is precisely $SY_n(\sigma)$.

\end{proof}     

Recall the algebra $Y_\mu[R^{(j)}]$ from (\ref{eq: YmuR}); note that as $\lambda = N\varpi_1$ we only adjoin variables $R^{(j)} := R_{n-1}^{(j)}$ for $j = 1,\ldots, N$.  We extend $\phi$ to an isomorphism $\phi:Y_\mu[R^{(j)}] \to SY_n(\sigma)\otimes \C[Z^{(1)},...,Z^{(N)}]$, where the $Z^{(j)}$ are formal variables.  On the central generators $\phi$ is defined by the equation
\begin{equation}
\label{eq:center}
(-1)^N R(-u+\tfrac{n}{2}) \mapsto u^N+Z^{(1)}u^{N-1}+\cdots+Z^{(N)}=:Z_N(u)
\end{equation}
We now consider the following diagram:
$$
\xymatrix{
& Y_n(\sigma)\ar[d]^\psi \ar@/^6pc/[dd]^{\kappa} \\
Y_\mu[R^{(j)}]\ar[r]^>>>>{\phi} \ar[d]_{\tau} & SY_n(\sigma)\otimes\C[Z^{(1)},...,Z^{(N)}] \ar[d]^\xi \\
Y_\mu^{N\varpi_1} \ar@{-->}[r]^\Phi & W(\pi)
}
$$ 
Here $\tau:Y_\mu[R^{(j)}] \to Y_\mu^{N\varpi_1}$ is the defining quotient map, while $\kappa:Y_n(\sigma) \to W(\pi)$ is Brundan and Kleshchev's surjection from Theorem \ref{thm: bk3}.  The map $\psi$ is the identity on $SY_n(\sigma)$ and on the center is defined by the equation
\begin{equation}
\label{eq:Z}
Z_N(u)=u^{p_1}(u-1)^{p_2}\cdots (u-n+1)^{p_n}\psi\big(Q_n(u)\big).
\end{equation}
The map $\xi$ is equal to $\kappa$ on $SY_n(\sigma)$ and on the center is defined by the equation
$$
\xi(Z_N(u))=u^{p_1}(u-1)^{p_2}\cdots (u-n+1)^{p_n}\kappa\big(Q_n(u)\big).
 $$   
 Note that by \cite[Lemma 3.7]{BK} the right hand side of the above equation is a polynomial in $u$ of degree $N$, and hence $\xi$ is a well-defined surjection.
 By construction we have that $\kappa=\xi\circ\psi$.

\begin{theorem}
\label{mainthm}
 We have that $\phi(I_\mu^\lambda) = \ker (\xi)$ and therefore $\phi$
 descends to an isomorphism $\Phi:Y_\mu^{N\varpi_1} \to W(\pi)$ of
 filtered algebras.
The map $\Phi$ induces an isomorphism $Y_\mu^{N\varpi_1} (\bR)\cong W(\pi)_{\bR}.$
 \end{theorem}

%
     
The proof of Theorem \ref{mainthm} will be given in the next section.   Note that the final claim, relating central quotients, is clear from (\ref{eq:center}) and the above discussion.

\subsubsection{Proof of Theorem \ref{mainthm}}
Our first order of business is to determine the image of $A_i^{(\ell)}$ under $\phi$.  From the identity $D_i(u)=\frac{Q_i(u+i-1)}{Q_{i-1}(u+i-1)}$ of (\ref{eq: D and Q}) we obtain that 
$$
\frac{D_{i+1}(-u-\frac{i-1}{2})}{D_{i}(-u-\frac{i-1}{2})}=\frac{Q_{i-1}(-u+\frac{i-1}{2})}{Q_{i}(-u+\frac{i-1}{2})}\frac{Q_{i+1}(-u+\frac{i+1}{2})}{Q_{i}(-u+\frac{i+1}{2})}.
$$
and hence the image 
\begin{equation}
\label{eq: image of H}
\phi\big( H_i(u) \big) = \frac{(u+\tfrac{i-1}{2})^{\mu_i} }{u^{\mu_i}} \frac{\psi(Q_{i-1}(-u+\frac{i-1}{2}))}{\psi(Q_{i}(-u+\frac{i-1}{2}))}\frac{\psi(Q_{i+1}(-u+\frac{i+1}{2}))}{\psi(Q_{i}(-u+\frac{i+1}{2}))}
\end{equation}
The next result is analogous to Corollary \ref{cor: existence of f}:
\begin{lemma}
\label{lemma: defining s 1}
There exist unique series $s_i(u) \in \C[Z^{(1)},\ldots, Z^{(N)}][[u^{-1}]]$ with constant term $1$ such that 
$$ \phi( A_i(u) ) = s_i(u) \psi\big( Q_i(-u+\tfrac{i-1}{2})\big) $$
for $i=1,\ldots, n-1$.  These satisfy the equations
\begin{equation}
\label{eq: defining s 1}
\phi\big( r_i(u) \big) \frac{s_{i-1}(u-\tfrac{1}{2}) s_{i+1}(u-\tfrac{1}{2}) }{s_i(u) s_i(u-1)} = \frac{(u+\tfrac{i-1}{2})^{\mu_i}}{u^{\mu_i}}
\end{equation}
for $i=1,\ldots, n-2$, and
\begin{equation}
\label{eq: defining s 2}
\phi\big(r_{n-1}(u)\big) \frac{s_{n-2}(u-\tfrac{1}{2})}{s_{n-1}(u) s_{n-1}(u-1)} = \frac{(u+\tfrac{n-2}{2})^{\mu_{n-1}}}{u^{\mu_{n-1}}} \psi\big( Q_n(-u+\tfrac{n}{2} ) \big)
\end{equation}
Moreover, these equations determine the $s_i(u)$ uniquely.
\end{lemma}
\begin{proof}
For each $i$, we have two factorizations for $\phi\big( H_i(u) \big)$: one in terms of $\phi\big(r_i(u) \big)$ and the $\phi\big( A_j(u) \big)$ by (\ref{eq: H from A}), and one in terms of the $\psi\big( Q_j (-u) \big)$ by (\ref{eq: image of H}) (with appropriate shifts in $u$ in both cases).  The claim now follows by applying the uniqueness of such factorizations \cite[Lemma 2.1]{GKLO}.
\end{proof}
Note that this result implies the desired match of Gelfand-Tsetlin subalgebras.

\begin{lemma}
\label{lemma: defining s 2}
For $i=1,\ldots, n-1$, 
\begin{equation}
\label{eq: s explicitly}
s_i(u) =\frac{ (u-\tfrac{i-1}{2})^{p_1} (u-\tfrac{i-3}{2})^{p_2} \cdots (u+\tfrac{i-1}{2})^{p_i} }{u^{m_i}}
\end{equation}
\end{lemma}
\begin{proof}
Denote the right-hand side of \ref{eq: s explicitly} by $x_i(u)$.  By the previous lemma, it suffices to show that the $x_i(u)$ satisfy the equations (\ref{eq: defining s 1}), (\ref{eq: defining s 2}).

For the case of equation (\ref{eq: defining s 2}), the left-hand side is 
$$ 
\phi\big(r_{n-1}(u) \big) \cdot \frac{x_{n-2}(u-\tfrac{1}{2})}{x_{n-1}(u) x_{n-1}(u-1)} 
$$
$$
= u^{-N} R(u)\frac{ (1-\tfrac{1}{2}u^{-1})^{m_{n-2}} }{(1-u^{-1})^{m_{n-1}} } \cdot \frac{u^{m_{n-1}} (u-1)^{m_{n-1}} }{(u-\tfrac{1}{2})^{m_{n-2}} (u+\tfrac{n-2}{2})^{p_{n-1}} \prod_{j=1}^{n-1} (u - \tfrac{n}{2}+j-1)^{p_j}}  
$$
after cancelling common factors between $x_{n-2}(u-\tfrac{1}{2})$ and $x_{n-1}(u)$.  This reduces to
$$ \frac{R(u)}{u^{p_n - p_{n-1}}} \frac{1}{(u+\tfrac{n-2}{2})^{p_{n-1}} \prod_{j=1}^{n-1}(u-\tfrac{n}{2} + j-1)^{p_j}} $$

Now consider the right hand side of (\ref{eq: defining s 2}).  Applying (\ref{eq:Z}) and (\ref{eq:center}), we get
\begin{align*}
& \ \ \ \ \ \ \ \frac{(u+\tfrac{n-2}{2})^{\mu_{n-1}}}{u^{\mu_{n-1}}} \phi\big( Q_n(-u+\tfrac{n}{2}) \big) \\
&= \frac{(u+\tfrac{n-2}{2})^{p_n - p_{n-1}} }{u^{p_n - p_{n-1}}} \frac{ Z_N(-u+\tfrac{n}{2}) }{(-u+\tfrac{n}{2})^{p_1} (-u + \tfrac{n-2}{2})^{p_2} \cdots (-u-\tfrac{n-2}{2})^{p_n} } \\
&= \frac{(u+\tfrac{n-2}{2})^{p_n - p_{n-1}} }{u^{p_n - p_{n-1}}} \frac{ (-1)^N \phi\big( R(u) \big) }{(-1)^N (u-\tfrac{n}{2})^{p_1} (u - \tfrac{n-2}{2})^{p_2} \cdots (u+\tfrac{n-2}{2})^{p_n} }
\end{align*}
and we see that the right and left sides agree.  

Verifying that the $x_i(u)$ satisfy equation (\ref{eq: defining s 1}) for $1\leq i < n-1$ is analogous, and is left as an exercise to the reader.
\end{proof}
\begin{lemma}
$\phi(I_\mu^\lambda) \subset \ker(\xi) $
\end{lemma}
\begin{proof}
Combining the two lemmas, 
$$
\phi\big(A_i(u)\big)= \frac{ (u-\tfrac{i-1}{2})^{p_1} (u-\tfrac{i-3}{2})^{p_2} \cdots (u+\tfrac{i-1}{2})^{p_i} }{u^{m_i}}\psi\big(Q_{i}(-u+\tfrac{i-1}{2})\big).
$$
By Theorem 3.5 in \cite{BK} we have that $\kappa\big(T_{\ell k}^{(r)}\big)=0$ for $r>p_k$.  Therefore for $k=1,...,n$
$$
\frac{(u-\frac{i-1}{2}+k-1)^{p_{k}}}{u^{p_k}}  \kappa\big(T_{\ell k}(-u+\tfrac{i-1}{2}-k+1)\big)
$$
is a polynomial in $u^{-1}$ of degree $p_k$.  Observe by (\ref{eq: quantum minor}) that 
\begin{multline*}
  \xi\circ \psi\big(Q_{i}(-u+\tfrac{i-1}{2})\big)\\ = \sum_{w \in S_{n-i}}
  (-1)^w \kappa\big(T_{w(1), 1}(-u+\tfrac{i-1}{2})\big) \cdots
  \kappa\big(T_{w(i), i}(-u+\tfrac{i-1}{2}-i+1)\big)
\end{multline*}
Since $p_1+\cdots+p_{n-i}=m_i$, it follows that $\xi\circ \phi\big(A_i(u)\big)$ is a polynomial in $u^{-1}$ of degree $m_i$. This proves the claim.
\end{proof}

\begin{lemma}
$\phi(I_\mu^\lambda) \supset \ker(\xi) $
\end{lemma}
\begin{proof}
By Lemmas \ref{lemma: defining s 1} and \ref{lemma: defining s 2}, we have
$$ 
\phi(A_1(u)) = s_1(u) \psi\big( Q_1(-u)\big) = \psi\big( Q_1(-u) \big)
$$
Noting that $D_1(u) = Q_1(u)$, it follows that $\psi\big(D_1^{(p)}\big) = (-1)^r \phi\big(A_1^{(r)}\big)$. 

By definition $\ker(\kappa) = \langle D_1^{(r)}: r> p_1\rangle$, so 
$$\ker(\xi) = \psi(\ker(\kappa)) = \langle \psi\big( D_1^{(r)} \big) : r> p_1\rangle$$
Since $p_1 = m_1$, the elements $\phi\big(A_1^{(r)}\big) \in \phi\big(I_\mu^\lambda\big)$ for $r>p_1$.  So $\ker(\xi) \subset \phi(I_\mu^\lambda)$.
\end{proof}

This completes the proof of Theorem \ref{mainthm}(a).
\excise{
\subsubsection{The Gelfand-Tsetlin subalgebra}
\label{sec:gelf-tsetl-subalg}

The last thing we need to check is the the Gelfand-Tsetlin suablgebras
of $Y^{N\omega_1}_\mu$ and $W(\pi)$ match in the expected way.  This
follows immediately from the compatibility of maps for smaller $\pi$.
Let $m\leq n$, and consider $\pi'=(p_1\leq \cdots \leq p_m)$, and
$M=\sum_{i=1}^mp_i$.  The nilpotent $e_\pi$ lies in the Levi
$\mathfrak{gl}_M\times \mathfrak{gl}_{N-M}$, so the inclusion of
$W(\pi')$ into the $W$-algebra of $\mathfrak{gl}_M\times
\mathfrak{gl}_{N-M}$ and thence into $W(\pi)$ induces an inclusion of
these algebras.

This map is compatible \bcom{Presumably from here on in, it's just
  writing some formulas. Maybe Alex knows them better than I do.  }
 
\acom{How exactly does this map work?  There is definitely a commutative diagram}
$$
\xymatrix{
 Y_m(\sigma') \ar[r] \ar[d] & Y_n(\sigma) \ar[d] \\
 W(\pi') \ar[r] & W(\pi)
}
$$
\acom{
where the top arrow is the obvious Yangian map sending generators to the same-named guys.  So the bottom arrow would have to send BK's generators to the same-named guys also.  BK use this map $W(\pi')\rightarrow W(\pi)$ e.g.~in \cite[Section 6.5]{BK}. But this map is a bit of a mess; I'm not sure how to phrase it directly in W-algebra terms.  

That said, via this system of maps $W(\pi') \rightarrow W(\pi)$ we can define a GT subalgebra, and this clearly gets identified with the GZ subalgebra of $Y_\mu^\lambda(\bR)$. Also, on the classical limit, this map is not bad at all: I believe it is the map $\cT_\pi \rightarrow \cT_{\pi'}$ which projects onto the smaller block matrices.  So these seem like the right maps to consider.
 }

}

\subsection{Proof of Theorem \ref{thm: main theorem}(b): The product monomial crystal and row tableau}
\label{subsection:Preliminaries}

 Let $\bR$ be a
set of parameters of weight $\lambda$ and define $\tbR$ to be the corresponding set of parameters of weight
$N\varpi_1$, as in Theorem \ref{thm: maps between yangians}.  We let $\gamma$
be a $W_P$-invariant weight such that the values of the weight on
blocks of size $i$ are given by the elements of $R_i$ with
multiplicity; while this is not unique, its orbit under the Weyl group
is.  Note that the elements of $\tbR$ are just the entries
of $\gamma+\rho_P$.

Note that the isomorphism $\Phi$ preserves the notion of highest weight vector and highest weight module: it sends $E$'s to $E$'s and $H$'s to $D$'s.  In this section we describe how the highest weights  of $Y_\mu^\lambda(\bR)$ and $Y_\mu^{N\varpi_1}(\tbR)$ (as described in Section \ref{sec: monomial crystal}) match up respectively with the highest weights of $W(\pi)_{\tbR}$ and $W(\pi,\fp)_\bR$ (as described in \cite{BK} and Section \ref{section:nilpotent-side}).  

That is, we will describe the commutative diagram 
\begin{equation}
\label{eq: hw sets diagram}
\tikz[->,thick]{
\matrix[row sep=10mm,column sep=20mm,ampersand replacement=\&]{
\node (a) {$\cB(\tbR)$}; \& \node (b) {$\RowR$};  \\
\node (d) {$\cB(\bR)$}; \& \node (e) {$\RowRtc$};  \\};
\draw (a) -- node[above,midway]{$\sim$} (b) ; 
\draw (d) edge[draw=none] node [sloped, auto=false,allow upside down]{$\subseteq$} (a) ; 
\draw (e) edge[draw=none] node [sloped, auto=false,allow upside down]{$\subseteq$} (b) ; 
\draw (d) -- node[above,midway]{$\sim$}(e) ; 
}
\end{equation}
as prescribed by Theorem \ref{thm: main theorem}(b). Both vertical maps are natural inclusions of subsets, and the horizontal maps  are bijections induced by $\Phi$.







\subsubsection{A bijection for \texorpdfstring{$\lambda=N\varpi_1$}{lambda=N pi1}}

Consider the isomorphism $\Phi: Y_\mu^{N\varpi_1} \rightarrow W(\pi)$ from Theorem \ref{mainthm}.  
By Equation (\ref{eq:center}), it follows that $\Phi$ descends to an isomorphism $$Y_\mu^{N\varpi_1}(\tbR) \cong W(\pi)_{\tbR}.$$

On the one hand, the set of highest weights $\cB(\tbR)_\mu$ of $Y_\mu^{N\varpi_1}(\tbR)$ is in bijection with the set
\begin{equation}
\label{eq: hw as flags }
H_\mu^{N\varpi_1}(\tbR)=\left\{ (\bS_i)_{i\in I}:\begin{array}{l} |\bS_i|=m_i \text{ and } \\ \bS_1+ n \subset \bS_2 + (n-1) \subset \cdots \subset \bS_{n-1}+1 \subset \tbR \end{array} \right\}
\end{equation}
As in (\ref{eq: monomial to highest weight 2}), the highest weight corresponding to $(\bS_i)\in H_\mu^{N \varpi_1}(\tbR)$ is given by 
$$A_i(u)\mapsto  \prod_{s\in \bS_i}(1-\tfrac{1}{2} s u^{-1}) = u^{-m_i} \prod_{s\in \bS_i} (u - \tfrac{1}{2} s)$$

On the other hand, recall from Section \ref{section: BK theorem} that the set of highest weights for $W(\pi)_{\tbR}$ is  $\RowRt$, the set of row symmetrized $\pi$--tableaux $T$ on the alphabet $\tbR$, and that $T\in \RowRt$ encodes a highest weight according to
$$ (u-i+1)^{p_i} D_i(u-i+1) \mapsto \prod_{a\in T_i} (u+\tfrac{1}{2}a - \tfrac{n}{2}) $$

\begin{proposition}
\label{prop2}
Let $\tbR$ be a multiset of size $N$.
The isomorphism $\Phi:Y_\mu^{N\varpi_1}(\tbR) \to W(\pi)_{\tbR}$ induces a bijection $$\RowRt \to \cB(\tbR)_\mu$$ given by 
$T \mapsto \bS=(\bS_i)$,
where $$\bS_i= \big(T_1 \cup \cdots \cup T_{i}\big) -(n-i+1),$$
and $T_i$ denotes the $i$-th row of $T$.    
\end{proposition}
Equivalently, the $i$th row 
$$ T_i = \big( \bS_i +(n-i)\big) \setminus \big( \bS_{i-1} + (n-i+1)\big) $$
is the difference between parts of the ``flag'' of multisets (\ref{eq: hw as flags }).

\begin{proof}
We begin with the equation 
$$
(u-i+1)^{p_i}D_i(u-i+1) = (u-i+1)^{p_i} \frac{Q_i(u)}{Q_{i-1}(u) } = \frac{(u-\tfrac{i-1}{2})^{m_i} \Phi \big(A_i(-u+\tfrac{i-1}{2})\big) }{(u-\tfrac{i-1}{2})^{m_{i-1}} \Phi \big( A_{i-1}(-u+\tfrac{i-2}{2}) \big)}
$$
The first equality is Equation (\ref{eq: D and Q}), while the second equality follows from Lemmas \ref{lemma: defining s 1}, \ref{lemma: defining s 2} after cancelling common factors.

For a highest weight $\bS = (\bS_i)$ for $Y_\mu^{N\varpi_1}(\tbR)$, the right-hand side maps to
$$
\frac{\prod_{s\in \bS_i} (u+ \tfrac{s-i+1}{2})}{\prod_{s\in \bS_{i-1}} ( u+\tfrac{s-i+2}{2})}
$$
To find the corresponding tableaux $T\in \RowRt$, we must write the above as 
$$ \prod_{a\in T_i} ( u +\tfrac{a}{2} - \tfrac{n}{2}), $$
which leads $T_i = \big( \bS_i + (n-i+1)\big) \setminus \big( \bS_{i-1} + (n-i+2)\big) $.  This proves the proposition.
 
\end{proof}

\subsubsection{A bijection for general \texorpdfstring{$\lambda$}{lambda}}
\label{section: bijection for general lambda}

Next we'll prove that the bijection of Proposition \ref{prop2} induces
a bijection between the highest weights of $Y_\mu^\lambda(\bR)$ and
the highest weights of $W(\pi,\fp)_\bR$.  We'll do this by first
identifying the tableau in $\RowRt$ which descend to highest
weights of $W(\pi,\fp)_\bR$; we term these ``overshadowing tableau''.
Once this is done, we need only check that these satisfy the same conditions as the subcrystal
$\cB(\bR) \subset \cB(\tbR)$ (cf. Lemma
\ref{lemma:crystalinclusion}).  

%

Let $\RowRtc$ denote the set of highest weights of $W(\pi,\fp)_\bR$.  By Theorem \ref{thm:parabolic-annihilator} there is an inclusion $\RowRtc \subset \RowRt$.  
Now suppose $c\in \bR_i$.  Then in $\tbR$ the element $c$ has $n-i$ ``descendants'', namely the elements
$$
\{ c+n-i-1,c+n-i-3,...,c-n+i+1\}.
$$ 
We'll call this set the $c$-\textbf{block} in $\tbR$.

Given a row tableau $T \in \RowRt$, we can divide the boxes
of the tableau into $c$-blocks.  Note that this decomposition will not
be unique if $\tbR$ contains any element with multiplicity greater
than 1.  We say that the tableau $T$ is \textbf{overshadowing} if this
division into $c$-blocks can be chosen so that for every $c\in \bR$ the elements of the $c$-block occur in strictly decreasing order down the tableau.

Put another way, given $T\in \RowRt$, an $\bR$-coloring of $T$ is a coloring of the contents of $T$ using $|\bR|$ colors, such that for every $c\in\bR$ the elements colored $c$  form a $c$-block, and they are in strictly decreasing order down the rows.  Clearly $T$ is overshadowing if and only there exists and $\bR$-coloring of $T$.

\begin{lemma}
$\RowRtc$ is precisely the subset of overshadowing row tableau in $\RowRt$.  
\end{lemma}

\begin{proof}
By Theorem 4.13, the set $\RowRtc$ is the set of tableaux where the row
reading word is of the form  $  w\cdot
(\gamma+\rho_{\fp})$, for $w\in \PS(\pi,\fp)$ and $\gamma$ is a
$W_{\fp}$-invariant weight where
each element of $\bR_i$ appears $n-i$ times.  Thus, the coordinates of
$\gamma+\rho_{\fp}$ are the concatenations of the $c$-blocks for the
different $c\in \bR_i$ for all $i$, ordered by the value of $c$.  The
longest left coset property says that every pair of elements of the same
$c$-block must be reversed in order.    That is, they must be in
decreasing order in rows (that is, they must satisfy the overshadowing condition) or in the same
row.  On the other hand, if they are in the same row, the shortest
right 
coset condition assures that they must have remained in the same
order, contradicting the longest left coset property. Thus, this
tableau must be overshadowing.  

Conversely, if a tableau is overshadowing, then the division into
$c$-blocks fixes a unique parabolic-singular permuation which sends $\gamma+\rho_{\fp}$
to a row reading of this tableau which matches the $c$-blocks  of the tableaux
$c$-blocks of $ \gamma+\rho_{\fp}$, while ordering each row by the
order on $c$-blocks in $\gamma+\rho_{\fp}$.  This makes the shortest
right coset property clear, and the longest left coset property
follows because overshadowing shows that every $c$-block is completely reversed.
\end{proof}

Let $\BB(\la)$ be the crystal associated to an irreducible representation of $\fg$ of highest weight $\la$.  By   \cite[Prop. 2.9]{KTWWY}, the crystal $\cB(\tbR)$ is isomorphic to 
$\BB(t_1\varpi_1)\otimes\cdots\otimes\BB(t_q\varpi_1)$, where $\tbR=\{c_1^{t_1},...,c_q^{t_q} \}$.  

Now, we shall describe the inclusion $\cB(\bR) \subset \cB(\tbR)$.  First, consider the case when $\lambda$ is fundamental.  The elements of $\cB(y_{i,c})$ are in bijection with
partitions fitting inside an $i\times n-i$ box, that is, with no more
than $i$ parts and $\xi_p\leq n-i$ (cf. Section \ref{section: pmc in type A}).  We identify a partition with its
diagram $\{(a,b)\in \Z_{>0}\times \Z_{>0}\mid 1\leq a\leq \xi_b\}$, and to the partition
$\xi$ we associate the monomial \[y_{\xi,c}=y_{i,c}\cdot \prod_{(a,b)\in \xi}
z_{i-a+b,c-a-b}^{-1}.\]

Thus, we wish to factor these into terms corresponding to
$\cB(y_{1,c+j})$ for $j=-i+1,\dots, i-1$.  This is easily done
using the formula \[y_{i,c}=y_{1,-i+1}y_{1,-i+3}\cdots y_{1,i-1}\prod_{k=1}^{i-1} \prod_{j=0}^{k-1}
z_{i-k,2j-k}^{-1}.\]
Thus, we have that 
\[y_{\xi,c}=\prod_{p=1}^i y_{1,c+i-2p+1}\prod_{q=1}^{\xi_p+i-p}
z_{q,c-q+i-2p}^{-1}=\prod_{p=1}^i y_{(\xi_p+i-p),c+i+1-2p}\] where we
consider $(\xi_p+i-p)$ as a partition with one row.  This gives us an
element in $\cB(y_{1,c+i+1-2p})$, resulting in the inclusion $$\cB(y_{i,c}) \subset \prod_{j=-i+1,...,i-1}\cB(y_{1,c+j}).$$

For general $\lambda$ we take the product over all such inclusions.  More precisely, for $p\in \cB(\bR)$ we write $p=\prod_{i\in I,c\in \bR_i}y_{\xi^{n-i,c},c}$.  Then by the above argument we can view $y_{\xi^{i,c},c} \in \prod_{j=-n+i+1,...,n-i-1}\cB(y_{1,c+j})$, and hence $$p\in \prod_{i\in I, c\in \bR_i} \prod_{j=-n+i+1,...,n-i-1}\cB(y_{1,c+j})=\cB(\tbR).$$

This procedure has a nice description in terms of diagrams.  Consider a monomial $p \in \B(\bR)$.  Recall that $p$ can be represented diagrammatically as in Section \ref{section: pmc in type A}, where here we assume that $\bR$ is an integral set of parameters

To define the image of $p$ in $\B(\tbR)$, the idea is to ``project'' the circled vertices onto the line corresponding to the $n-1$ node of the Dynkin diagram, and fill the squares along this projection with $1$s.  

For instance, if we work in type $A_6$, with $\bR_3=\{4\}$, all other $\bR_i$ empty, and  we attach the partition $(2,1)$ to this vertex then we have the  picture on the left; 
after projecting we obtain the picture on the right:

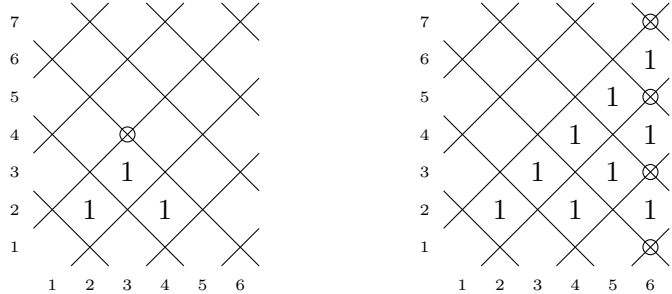
\begin{equation*}
\begin{tikzpicture}[scale=0.5]

\draw (0.5,5.5)--(2.5,7.5);
\draw (0.5,3.5)--(4.5,7.5);
\draw (0.5,1.5)--(6.5,7.5);

\draw (1.5,0.5)--(6.5,5.5);
\draw (3.5,0.5)--(6.5,3.5);
\draw (5.5,0.5)--(6.5,1.5);

\draw (2.5,0.5)--(0.5,2.5);
\draw (4.5,0.5)--(0.5,4.5);
\draw (6.5,0.5)--(0.5,6.5);
\draw (6.5,2.5)--(1.5,7.5);
\draw (6.5,4.5)--(3.5,7.5);
\draw (6.5,6.5)--(5.5,7.5);

\draw (3,4) circle(0.2cm);
\draw node at (2,2) {1};
\draw node at (3,3) {1};
\draw node at (4,2) {1};

\draw node at (1,0) {\tiny 1};
\draw node at (2,0) {\tiny 2};
\draw node at (3,0) {\tiny 3};
\draw node at (4,0) {\tiny 4};
\draw node at (5,0) {\tiny 5};
\draw node at (6,0) {\tiny 6};

\draw node at (0,1) {\tiny 1};
\draw node at (0,2) {\tiny 2};
\draw node at (0,3) {\tiny 3};
\draw node at (0,4) {\tiny 4};
\draw node at (0,5) {\tiny 5};
\draw node at (0,6) {\tiny 6};
\draw node at (0,7) {\tiny 7};
\end{tikzpicture}
\hspace{2cm}
\begin{tikzpicture}[scale=0.5]

\draw (0.5,5.5)--(2.5,7.5);
\draw (0.5,3.5)--(4.5,7.5);
\draw (0.5,1.5)--(6.5,7.5);

\draw (1.5,0.5)--(6.5,5.5);
\draw (3.5,0.5)--(6.5,3.5);
\draw (5.5,0.5)--(6.5,1.5);

\draw (2.5,0.5)--(0.5,2.5);
\draw (4.5,0.5)--(0.5,4.5);
\draw (6.5,0.5)--(0.5,6.5);
\draw (6.5,2.5)--(1.5,7.5);
\draw (6.5,4.5)--(3.5,7.5);
\draw (6.5,6.5)--(5.5,7.5);

\draw (6,1) circle(0.2cm);
\draw (6,3) circle(0.2cm);
\draw (6,5) circle(0.2cm);
\draw (6,7) circle(0.2cm);

\draw node at (2,2) {1};
\draw node at (3,3) {1};
\draw node at (4,2) {1};
\draw node at (4,4) {1};
\draw node at (5,3) {1};
\draw node at (5,5) {1};
\draw node at (6,2) {1};
\draw node at (6,4) {1};
\draw node at (6,6) {1};

\draw node at (1,0) {\tiny 1};
\draw node at (2,0) {\tiny 2};
\draw node at (3,0) {\tiny 3};
\draw node at (4,0) {\tiny 4};
\draw node at (5,0) {\tiny 5};
\draw node at (6,0) {\tiny 6};

\draw node at (0,1) {\tiny 1};
\draw node at (0,2) {\tiny 2};
\draw node at (0,3) {\tiny 3};
\draw node at (0,4) {\tiny 4};
\draw node at (0,5) {\tiny 5};
\draw node at (0,6) {\tiny 6};
\draw node at (0,7) {\tiny 7};
\end{tikzpicture}
\end{equation*}
In general,  the inclusion $\B(\bR) \to \B(\tbR)$ is defined by applying this projection to every vertex. 
For instance, consider the monomial data $p_0\in\B(\bR)$ on the left below, where $\bR_3=\bR_5=\{4\}$ and all other $\bR_i$ are empty. 
The corresponding monomial data in $\B(\tbR)$ is on the right: 
\begin{equation*}
\begin{tikzpicture}[scale=0.5]

\draw (0.5,7.5)--(1.5,8.5);
\draw (0.5,5.5)--(3.5,8.5);
\draw (0.5,3.5)--(5.5,8.5);
\draw (0.5,1.5)--(6.5,7.5);

\draw (0.5,-0.5)--(6.5,5.5);
\draw (2.5,-0.5)--(6.5,3.5);
\draw (4.5,-0.5)--(6.5,1.5);
\draw (6.5,-0.5)--(6.5,-0.5);

\draw (1.5,-0.5)--(0.5,0.5);
\draw (3.5,-0.5)--(0.5,2.5);
\draw (5.5,-0.5)--(0.5,4.5);
\draw (6.5,0.5)--(0.5,6.5);
\draw (6.5,2.5)--(0.5,8.5);
\draw (6.5,4.5)--(2.5,8.5);
\draw (6.5,6.5)--(4.5,8.5);

\draw (3,4) circle(0.2cm);
\draw (5,4) circle(0.2cm);

\draw node at (2,2) {1};
\draw node at (3,3) {1};
\draw node at (4,2) {2};
\draw node at (5,3) {1};
\draw node at (5,1) {1};
\draw node at (6,2) {1};

\draw node at (1,-0.8) {\tiny 1};
\draw node at (2,-0.8) {\tiny 2};
\draw node at (3,-0.8) {\tiny 3};
\draw node at (4,-0.8) {\tiny 4};
\draw node at (5,-0.8) {\tiny 5};
\draw node at (6,-0.8) {\tiny 6};

\draw node at (0,0) {\tiny 0};
\draw node at (0,1) {\tiny 1};
\draw node at (0,2) {\tiny 2};
\draw node at (0,3) {\tiny 3};
\draw node at (0,4) {\tiny 4};
\draw node at (0,5) {\tiny 5};
\draw node at (0,6) {\tiny 6};
\draw node at (0,7) {\tiny 7};
\draw node at (0,8) {\tiny 8};
\end{tikzpicture}
\hspace{2cm}
\begin{tikzpicture}[scale=0.5]
\draw (0.5,7.5)--(1.5,8.5);
\draw (0.5,5.5)--(3.5,8.5);
\draw (0.5,3.5)--(5.5,8.5);
\draw (0.5,1.5)--(6.5,7.5);

\draw (0.5,-0.5)--(6.5,5.5);
\draw (2.5,-0.5)--(6.5,3.5);
\draw (4.5,-0.5)--(6.5,1.5);
\draw (6.5,-0.5)--(6.5,-0.5);

\draw (1.5,-0.5)--(0.5,0.5);
\draw (3.5,-0.5)--(0.5,2.5);
\draw (5.5,-0.5)--(0.5,4.5);
\draw (6.5,0.5)--(0.5,6.5);
\draw (6.5,2.5)--(0.5,8.5);
\draw (6.5,4.5)--(2.5,8.5);
\draw (6.5,6.5)--(4.5,8.5);

\draw (6,1) circle(0.2cm);
\draw (6,3) circle(0.2cm);
\draw (6,3) circle(0.4cm);
\draw (6,5) circle(0.2cm);
\draw (6,5) circle(0.4cm);
\draw (6,7) circle(0.2cm);

\draw node at (2,2) {1};
\draw node at (3,3) {1};
\draw node at (4,2) {2};
\draw node at (4,4) {1};
\draw node at (5,3) {2};
\draw node at (5,1) {1};
\draw node at (5,5) {1};
\draw node at (6,2) {2};
\draw node at (6,4) {2};
\draw node at (6,6) {1};

\draw node at (1,-0.8) {\tiny 1};
\draw node at (2,-0.8) {\tiny 2};
\draw node at (3,-0.8) {\tiny 3};
\draw node at (4,-0.8) {\tiny 4};
\draw node at (5,-0.8) {\tiny 5};
\draw node at (6,-0.8) {\tiny 6};

\draw node at (0,0) {\tiny 0};
\draw node at (0,1) {\tiny 1};
\draw node at (0,2) {\tiny 2};
\draw node at (0,3) {\tiny 3};
\draw node at (0,4) {\tiny 4};
\draw node at (0,5) {\tiny 5};
\draw node at (0,6) {\tiny 6};
\draw node at (0,7) {\tiny 7};
\draw node at (0,8) {\tiny 8};
\end{tikzpicture}
\end{equation*}

Finally we are ready to prove:
\begin{proposition}\label{lem:simples match2}
Under the bijection of Proposition \ref{prop2}, $\RowRtc$ is identified with $\cB(\bR)$.
\end{proposition}
This completes the proof of part (b) of Theorem \ref{thm: main theorem}.
\begin{proof}
By the above discussion, we view $\cB(\bR) \subset \cB(\tbR)$.  
Let $\bS=(\bS_i)\in\cB(\bR)$, and suppose it corresponds to $T\in \RowRt$ under the bijection of Proposition \ref{prop2}.    
Denoting the rows of $T$ by $T_i$, we have that 
$T_n=\tbR\setminus(\bS_{n-1}+1), T_1=\bS_{1}+n$ and for $i=1,...,n-2$,  
$$T_{i}=(\bS_i+(n-i))\setminus(\bS_{i-1}+(n-i+1)).$$

We'll show that  $T \in \RowRtc$, i.e. $T$ has an $\bR$-coloring.  We'll first show that it suffices to prove this in the case when $p$ consists of only one vertex (i.e. $|\bR|=1$).  Without loss of generality assume that $\bR$ is integral.  Now color each partition in $p$.  For instance we could have the example on the left below.
When we view $p$ as a monomial datum in $\cB(N\varpi_1,\tbR)$ we remember the color of the partitions.  In the example we obtain the diagram on the right.
 \begin{equation*}
\begin{tikzpicture}[scale=0.5]
\draw (0.5,7.5)--(1.5,8.5);
\draw (0.5,5.5)--(3.5,8.5);
\draw (0.5,3.5)--(5.5,8.5);
\draw (0.5,1.5)--(6.5,7.5);

\draw (0.5,-0.5)--(6.5,5.5);
\draw (2.5,-0.5)--(6.5,3.5);
\draw (4.5,-0.5)--(6.5,1.5);
\draw (6.5,-0.5)--(6.5,-0.5);

\draw (1.5,-0.5)--(0.5,0.5);
\draw (3.5,-0.5)--(0.5,2.5);
\draw (5.5,-0.5)--(0.5,4.5);
\draw (6.5,0.5)--(0.5,6.5);
\draw (6.5,2.5)--(0.5,8.5);
\draw (6.5,4.5)--(2.5,8.5);
\draw (6.5,6.5)--(4.5,8.5);

\draw (3,4) circle(0.2cm);
\draw (5,4) circle(0.2cm);

\draw node at (2,2) {\color{blue} 1};
\draw node at (3,3) {\color{blue} 1};
\draw node at (4,2) {\color{blue}1,\color{red} 1};
\draw node at (5,3) {\color{red} 1};
\draw node at (5,1) {\color{red} 1};
\draw node at (6,2) {\color{red} 1};

\draw node at (1,-0.8) {\tiny 1};
\draw node at (2,-0.8) {\tiny 2};
\draw node at (3,-0.8) {\tiny 3};
\draw node at (4,-0.8) {\tiny 4};
\draw node at (5,-0.8) {\tiny 5};
\draw node at (6,-0.8) {\tiny 6};

\draw node at (0,0) {\tiny 0};
\draw node at (0,1) {\tiny 1};
\draw node at (0,2) {\tiny 2};
\draw node at (0,3) {\tiny 3};
\draw node at (0,4) {\tiny 4};
\draw node at (0,5) {\tiny 5};
\draw node at (0,6) {\tiny 6};
\draw node at (0,7) {\tiny 7};
\draw node at (0,8) {\tiny 8};
\end{tikzpicture}
\hspace{2cm}
\begin{tikzpicture}[scale=0.5]
\draw (0.5,7.5)--(1.5,8.5);
\draw (0.5,5.5)--(3.5,8.5);
\draw (0.5,3.5)--(5.5,8.5);
\draw (0.5,1.5)--(6.5,7.5);

\draw (0.5,-0.5)--(6.5,5.5);
\draw (2.5,-0.5)--(6.5,3.5);
\draw (4.5,-0.5)--(6.5,1.5);
\draw (6.5,-0.5)--(6.5,-0.5);

\draw (1.5,-0.5)--(0.5,0.5);
\draw (3.5,-0.5)--(0.5,2.5);
\draw (5.5,-0.5)--(0.5,4.5);
\draw (6.5,0.5)--(0.5,6.5);
\draw (6.5,2.5)--(0.5,8.5);
\draw (6.5,4.5)--(2.5,8.5);
\draw (6.5,6.5)--(4.5,8.5);

\draw (6,1) circle(0.2cm);
\draw (6,3) circle(0.2cm);
\draw (6,3) circle(0.4cm);
\draw (6,5) circle(0.2cm);
\draw (6,5) circle(0.4cm);
\draw (6,7) circle(0.2cm);

\draw node at (2,2) {\color{blue} 1};
\draw node at (3,3) {\color{blue} 1};
\draw node at (4,2) {\color{blue} 1,\color{red} 1};
\draw node at (4,4) {\color{blue} 1};
\draw node at (5,3) {\color{blue} 1,\color{red} 1};
\draw node at (5,1) {\color{red} 1};
\draw node at (5,5) {\color{blue} 1};
\draw node at (6,2) {\color{blue} 1,\color{red} 1};
\draw node at (6,4) {\color{blue} 1,\color{red} 1};
\draw node at (6,6) {\color{blue} 1};

\draw node at (1,-0.8) {\tiny 1};
\draw node at (2,-0.8) {\tiny 2};
\draw node at (3,-0.8) {\tiny 3};
\draw node at (4,-0.8) {\tiny 4};
\draw node at (5,-0.8) {\tiny 5};
\draw node at (6,-0.8) {\tiny 6};

\draw node at (0,0) {\tiny 0};
\draw node at (0,1) {\tiny 1};
\draw node at (0,2) {\tiny 2};
\draw node at (0,3) {\tiny 3};
\draw node at (0,4) {\tiny 4};
\draw node at (0,5) {\tiny 5};
\draw node at (0,6) {\tiny 6};
\draw node at (0,7) {\tiny 7};
\draw node at (0,8) {\tiny 8};
\end{tikzpicture}
\end{equation*}
Now, when we apply the bijection, we naturally obtain a  row tableau whose entries are colored (we don't know a priori that this is an $\bR$-colored tableau - this is what we want to show).  Indeed when we look at $(\bS_i+(n-i))\setminus(\bS_{i-1}+(n-i+1)$, we preserve the color of the elements that haven't been cancelled (for $c\in\bS_i$, the element $c+n-i\in\bS_i+(n-i)$ is understood to have the same color as $c$).  Moreover, the last row is given by $\tbR\setminus(\bS_{n-1}+1)$, and the elements of $\tbR$ are colored the same color as the node which ``overshadowed'' them.  This is much easier with an example: the bijection applied to the above monomial data results in the following colored row tableau (which happens to contain two empty rows):
$$
\ytableausetup{smalltableaux}
\begin{ytableau}
\\
\color{blue} 7 \\
 \\
\color{blue} 5 & \color{red} 5 \\
\color{red} 3 \\
\color{blue} 3 \\
\color{blue} 1
\end{ytableau}
$$
Note that the red content is precisely the block corresponding to $4\in \bS_5$, and the blue content is the block corresponding to $4\in \bS_3$.  Moreover if $p$ consisted of, say, just the red partition then the resulting row tableau is the red part of the above tableau.  This shows that it suffices to consider the case where $|\bR|=1$, and show that the resulting row tableau is overshadowing.

To this end, suppose $\bR_{n-i}=\{k\}$ (so the other multisets $\bR_j$ are empty), and in the monomial data $p$ the partition $\la=(\la_1\geq\cdots\geq\la_i\geq 0)$ corresponds to $k$.  Then for $j=1,...,i$, $T$ has content $k+i-2j+1$ going down the rows, which is manifestly overshadowing.  This proves that  $T\in \RowRtc$ for any $p\in\B(\bR)$.

To prove that the bijection $\B(\bR) \to \RowRtc$ is surjective, given $T\in \RowRtc$ choose an $\bR$-coloring of $T$.  This partitions the contents of $T$ into $c$-blocks, and for each such block we can reverse the process above to construct a monomial datum.  If we do this for all blocks at once we obtain a datum in $\B(\bR)\subset \B(\tbR)$. 
\end{proof}

\begin{remark}
  Under this bijection, we obtain a crystal structure on
  overshadowing tableaux.  One can easily work out that this coincides
  with the one induced by Brundan and Kleshchev's crystal structure on row tableau in \cite[\S 4.3]{BK}.
\end{remark}

\section{Proof of Theorem \ref{thm: main theorem}(d): The classical limit}
\label{sec: classical limit}

In this section, we will study the classical limit of our isomorphism 
$$\Phi: Y_\mu^\lambda(\bR)\stackrel{\sim}{\longrightarrow} W(\pi, \fp)_\bR$$  
Our goal is to establish part (d) of Theorem \ref{thm: main theorem}, and show that this classical limit agrees with the Mirkovi\'c-Vybornov isomorphism.  

\begin{remark}
We may immediately save ourselves some work with an observation: it suffices to prove the case of $ \lambda = N\varpi_1$, as in general both isomorphisms are defined by restricting this case to closed subvarieties.
\end{remark}

\subsection{More about slices to nilpotent orbits}
\label{section: MV slices}
In this subsection we let $G$ be a reductive algebraic group over $\C$, with Lie algebra $\fg$.  We will fix throughout a nonzero nilpotent element $e \in \fg$, and an $\mathfrak{sl}_2$--triple $\{e, h, f\}$.  In this appendix, we slightly generalize some of the results on Slodowy slices from \cite{GG}, showing in particular that the classical Slodowy slice and the transverse slice considered in \cite{MV} are Poisson isomorphic.  Since these results may be of independent interest, we provide brief proofs.

\begin{definition}
Let $C \subset \fg$ be an $\text{ad}_h$-invariant subspace such that $\fg = [\fg, e] \oplus C$. Then the affine space $\mathcal{M} = e + C$ is called an \textbf{MV slice}.
\end{definition}

The most natural choice of such a slice is the Slodowy slice, where $C=\fg^f$ (cf. Section \ref{subsection:finiteWalg}).  There are many others however. 

%
\begin{remark}
An MV slice $\mathcal{M} = e + C$ is a transverse slice to the nilpotent orbit $\mathbb{O}_e$ at the point $e$.
\end{remark}

From now on, we assume that $\mathcal{M}$ is an MV slice. Note that the eigenvalues of $\text{ad}_h$ acting on $C$ are necessarily non-positive.  From our $\mathfrak{sl}_2$--triple we get a homomorphism $SL_2 \rightarrow G$, and we will denote by $\gamma(t)$ the image of $\begin{pmatrix} t & 0 \\ 0 & t^{-1} \end{pmatrix} $ in $G$.  We consider the $\C^\times$--action (the Kazhdan action) on $\fg$ defined by
$$ \rho(t)\cdot x = t^2 \left( \text{Ad} \gamma(t^{-1})\right) (x) $$
Note that $\rho$ preserves $\mathcal{M}$ and contracts it to the unique fixed point $e$.


Consider the decomposition $\fg = \bigoplus_{i\in \Z} \fg_i$ into $\text{ad}_h$ weight spaces.  As in Section \ref{subsection:finiteWalg} there is a non-degenerate skew-symmetric form $\langle x,y \rangle = ( e, [x,y])$ on $\fg_{-1}$.  Choose a Lagrangian subspace $\fl \subset\fg_{-1}$ with respect to $\langle \cdot , \cdot \rangle$.

Define the nilpotent Lie subalgebra $\mathfrak{m} = \fl \oplus \bigoplus_{i\leq -2} \fg_i$, and the corresponding unipotent subgroup $M\subset G$. Note that $\mathfrak{m}^\perp = [e, \fl]\oplus \bigoplus_{i\leq 0} \fg_i $ is the orthogonal complement of $\mathfrak{m}$ with respect to the Killing form.  The following result is a generalization of Lemma 2.1 in \cite{GG}.

\begin{lemma} \label{lemma: 2.1 GG}
The adjoint action map $\alpha: M \times \mathcal{M} \rightarrow e + \mathfrak{m}^\perp$ is a $\C^\times$-equivariant isomorphism of affine varieties.  Here $\C^\times$ acts on $e+ \mathfrak{m}^\perp$ by $\rho$, and on $M\times \mathcal{M}$ by
$$ t\cdot (g, x) = \left(\gamma(t^{-1}) g \gamma(t), \rho(t) \cdot x\right) $$
\end{lemma}
\begin{proof}
Since $\mathcal{M}$ is a MV slice we have $C\subset \bigoplus_{i\leq 0} \fg(i) \subset \mathfrak{m}^\perp$, so indeed the image of the adjoint map $M \times \mathcal{M} \rightarrow \fg$ is contained in $e + \mathfrak{m}^\perp$.

Next, since $\fg = [\fg, e] \oplus C$ it follows that $[\mathfrak{m}, e] \cap C = 0$.  We also have $\dim \text{Ker} (\text{ad} f) = \dim C$, since both spaces are complementary to $[\fg, e]$ in $\fg$.  Since $\text{ad} e: \mathfrak{m}\rightarrow [e,\mathfrak{m}]$ is an isomorphism,
\begin{align*}
\dim \mathfrak{m}^\perp &= \dim \mathfrak{m}  +\dim \fg(0) + \dim \fg(-1) = \dim [\mathfrak{m}, e] + \dim \text{Ker}( \text{ad} f) = \\ &= \dim [\mathfrak{m}, e] + \dim C
\end{align*}
So $\mathfrak{m}^\perp = [\mathfrak{m}, e] \oplus C$.  The remainder of the proof proceeds as in \cite{GG}.
\end{proof}
Following Section 3.2 in \cite{GG}: $e$ is a regular value for the moment map $\mu : \fg^\ast \rightarrow \mathfrak{m}^\ast$, $\mu^{-1}(e) = e + \mathfrak{m}^\perp$ (under $\fg^\ast \cong \fg$), and it follows from Lemma \ref{lemma: 2.1 GG} that we have a Hamiltonian reduction of the Poisson structure on $\fg^\ast$ to $\mathcal{M}$.  It is induced from the isomorphisms
$$ \mathcal{M} \cong \mu^{-1}(e)  / M, \quad  \C[\mathcal{M}] \cong \left( \C[\fg] / I( \mu^{-1}(e)) \right)^M$$ 

\begin{theorem}
\label{theorem: MV slices}
There is a $\C^\times$--equivariant isomorphism of affine Poisson varieties between any two MV slices.  
\end{theorem}
\begin{proof}
With $\mathfrak{l}$ and $\mathfrak{m}$ fixed as above, for MV slices $\mathcal{M}_1, \mathcal{M}_2$ we have $\C^\times$--equivariant isomorphisms
$$ M \times \mathcal{M}_1 \cong e + \mathfrak{m}^\perp \cong M \times \mathcal{M}_2 $$
by Lemma \ref{lemma: 2.1 GG}.  The Poisson structures on $\mathcal{M}_1, \mathcal{M}_2$ are both induced by Hamiltonian reduction, giving us the desired Poisson isomorphism.
\end{proof}

\begin{remark}
As in \cite[Section 3.1]{GG}, any MV slice $\mathcal{M}$ also inherits an {\em induced} Poisson structure as a subvariety of $\fg$ (with its Kostant-Kirillov-Souriau	 Poisson structure under $\fg\cong \fg^\ast$), by applying \cite[Proposition 3.10]{Va94}.  This Poisson structure agrees with that given above via Hamiltonian reduction, cf. \cite[Section 3.2]{GG}
\end{remark}

\subsection{More about affine Grassmannian slices}
\label{section: more on gr slices}

Let us briefly recall some aspects of the ``loop group'' description of the slices $\Grlmbar$ in the affine Grassmannian of $G=SL_n$, to supplement the lattice description given in Section \ref{section: slices in the affine Grassmannian}.

\begin{remark}
In Section \ref{section: slices in the affine Grassmannian} we defined the slices in the affine Grassmannian of $PGL_n$.  Here use $G=SL_n$ - this doesn't make a difference.  In fact, we can use any group whose Lie algebra is $\mathfrak{sl}_n$ \cite[Section 2G]{KWWY}.
\end{remark}

 $\Grlmbar$ is a transverse slice to $\Gr^\mu \subset \overline{\Gr^\lambda}$, defined as the intersection
\begin{equation}
\Grlmbar = \Gr_\mu \cap \overline{\Gr^\lambda}
\end{equation}
where $\Gr_\mu = G_1[t^{-1}] t^{w_0 \mu}$ is an orbit for the opposite group $G_1[t^{-1}] = \operatorname{Ker}(G[t^{-1}] \stackrel{t\mapsto \infty}{\rightarrow} G)$.  Every point in $\Gr_\mu$ has a unique representative of the form $ g t^{w_0 \mu}$, where $g = (a_{ij}) \in G_1[t^{-1}]$ satisfies
\begin{equation}
\label{eq: reps of Grm}
a_{ij} = \delta_{ij} + a_{ij}^{(1)} t^{-1} + a_{ij}^{(2)} t^{-2}+ \ldots \in \delta_{ij} + t^{p_i - p_j -1}\C[t^{-1}]
\end{equation}
(Recall that we are denoting $w_0 \mu = (p_1,\ldots p_n)$) In this way, $\Grlmbar$ may be considered as a closed subscheme of $G_1[t^{-1}]$.

\begin{remark}
It is sometimes convenient to work with the group $G_1[[t^{-1}]]$.  One advantage is that elements of this group admit Gauss decompositions,
$$G_1[[t^{-1}]] = U_1^-[[t^{-1}]] T_1[[t^{-1}]] U_1^+[[t^{-1}]]$$
where $U^\pm, T\subset G$ are the subgroups of upper/lower triangular and diagonal matrices.  The varieties $\Grlmbar$ may also be considered as closed subschemes of $G_1[[t^{-1}]]$.
\end{remark}

In particular, we may describe $\Gr^{\overline{N\varpi_1}}_\mu$ as the variety of matrices $g=(a_{ij})$ of the form (\ref{eq: reps of Grm}) (with $\det g =1$), with the additional constraint that $a_{ij}^{(r)} =0$ for $r> p_j$.  Then explicitly $g$ corresponds to the lattice 
\begin{equation}
\Lambda = \operatorname{span}_\cO \left\{ t^p_j e_j + \sum_{i, r} a_{ij}^{(r)} t^{p_j - r} e_i: 1\leq j\leq n\right\}
\end{equation}
allowing us to compare with our previous description (\ref{eq: def of GrNwm}) of $\Gr^{\overline{N\varpi_1}}_\mu$.

Using the above identification of $\Grlmbar \subset G_1[[t^{-1}]]$, we now recall how the classical limit of $Y_\mu^\lambda(\bR)$ is identified with functions on $\Grlmbar$ (as was promised in Section \ref{sec: yangians and functions on slices}).  Following Theorems 3.9, 3.12 and Proposition 4.3 in \cite{KWWY}, 
\begin{align*}
A_i^{(r)} & \mapsto [t^{-r}] \Delta_{\{1,\ldots, i\}, \{1,\ldots, i\}}, \\
E_i^{(r)} & \mapsto [t^{-r}] \frac{\Delta_{\{1,\ldots, i-1, i+1\}, \{1,\ldots, i\} }}{\Delta_{\{1,\ldots, i\}, \{1,\ldots, i\}} }, \\
F_i^{(s)} & \mapsto [t^{-s}] \frac{\Delta_{\{1,\ldots, i\}, \{1,\ldots, i-1, i+1\}}}{\Delta_{\{1,\ldots, i\}, \{1,\ldots, i\}}}
\end{align*}
Here, for $I, J\subset \{1,\ldots, n\}$ we denote by $\Delta_{I,J}$ the minor with rows $I$ and columns $J$, thought of as a map $G_1[[t^{-1}]] \rightarrow \C[[t^{-1}]]$.  Meanwhile, $[t^{-r}]$ extracts the coefficient of $t^{-r}$.  Restricting these functions to the closed subscheme $\Grlmbar \subset G_1[[t^{-1}]]$ gives the desired isomorphism.

The following result is clear from the structure of the GKLO representation \cite[Theorem 4.5]{KWWY}.  It is also follows from similar results for Zastava spaces \cite{FM}, since $\Gr^{\overline{N\varpi_1}}_\mu$ and ${Z}^{N\varpi_1 - \mu}$ are birational.
\begin{proposition}
\label{prop: birational coordinates}
The functions $A_i^{(r)}, E_i^{(r)}$ for $i\in I$, $1\leq r\leq m_i$ are birational coordinates on $\Grlmbar$.
\end{proposition}

\subsection{The MV isomorphism}
\label{sec: the mv isomorphism}
As per usual, fix now $\lambda \geq \mu$ dominant coweights and associated partitions $\tau \geq \pi$ of $N$.

Let $e_\pi\in \mathfrak{gl}_N$ be the nilpotent element with  {\em lower triangular} Jordan type $\pi = (p_1\leq \cdots \leq p_n) \vdash N$, we will consider the {\bf transpose MV slice}
\begin{equation}
\label{eq: def of slice T}
\mathcal{T}_\pi = \left\{ X = (X_{ij}) \in \mathfrak{gl}_N : \begin{array}{cl} (a) & X_{ij} \text{ has size } p_i \times p_j, \\ (b) & X_{ii} \text{ has 1's below diagonal, entries in final} \\ & \ \ \text{column} \\ (c) & X_{ij} \text{ for } i\neq j \text{ has entries in final column,} \\ & \ \ \text{but not below row } p_j \end{array} \right\}
\end{equation}

Recall the description of $\Grlmbar$ from Section \ref{section: slices in the affine Grassmannian}.  With this description and the definition of $\cT_\pi$ in mind, we can now give slightly more precise formulation of the MV isomorphism \ref{theorem:MV}: 

{\em 
For any $\Lambda \in \Grlmbar$, identify $\Lambda_0 / \Lambda \cong \C^N$ via the basis $E_\pi$; in particular we may identify multiplication by $t$ on $\Lambda_0 / \Lambda$ with an element $X\in \mathfrak{gl}_N$.  Then map taking $\Lambda$ to $X\in \mathfrak{gl}_N$ defines an isomorphism $ \Grlmbar \stackrel{\sim}{\longrightarrow} T_\pi \cap \overline{\mathbb{O}_\tau} $. It is compatible with the inclusions of closed subvarieties $\Grlmbar \subset \Gr^{\overline{N \varpi_1}}_\mu$ and $T_\pi \cap \overline{\mathbb{O}_\tau} \subset T_\pi \cap \mathcal{N}_{\mathfrak{gl}_N}$.
}

Following \cite[Section 3.3]{CK16}, it is straightforward to write the above isomorphism down explicitly in coordinates.  On the affine Grassmannian side, we identify $\Lambda$ with $g\in G_1[[t^{-1}]]$ as in the previous section and use the coefficients $a_{ij}^{(r)}$ of the matix entries of $g$ as coordinates.  On the nilpotent cone side, the image is a block matrix $X = (X_{ij})$.  Then under the above isomorphism, the block $X_{ij}$ has interesting entries only in its final column:
$$ 
X_{ij} = \begin{pmatrix} 
0 & 0 &\cdots &0& \cdots & 0& -a_{ij}^{(p_j)}\\
\delta_{ij} & 0 &\cdots&0 &\ddots & \vdots& \vdots \\
0 & \delta_{ij} &\ddots& &\ddots & 0& -a_{ij}^{(1)} \\
0 &0 &\ddots& \ddots&  & 0& 0\\
0 & \ddots&\ddots&\ddots & \ddots& \vdots& \vdots \\
\vdots &\ddots &\ddots& 0& \delta_{ij}& 0 &0 \\
0 &\cdots &0 &0& 0& \delta_{ij}& 0\\
	\end{pmatrix}
$$


\subsection{Completing the proof of Theorem \ref{thm: main theorem}}
\label{sec: completing proof of classical limit}

To finish the proof of the theorem, we will now compare the classical limit of our isomorphism $\Phi: Y_\mu^{N \varpi_1}(\bR) \stackrel{\sim}{\longrightarrow} W(\pi)_\bR$ with the Mirkovi\'c-Vybornov isomorphism.  We interpret the classical limit of $\Phi$ as an isomorphism of coordinate rings
\begin{equation}
\Phi: \C[ \Gr^{\overline{N \varpi_1}}_\mu] \stackrel{\sim}{\longrightarrow} \C[ \cT_\pi \cap \cN_{\mathfrak{gl}_N}]
\end{equation}
In the notation of the previous section, suppose that $g \in \Gr^{\overline{N\varpi_1}}_\mu \subset G_1[[t^{-1}]]$ maps to $X \in \cT_\pi\cap \cN_{\mathfrak{gl}_N}$ under the Mirkovi\'c-Vybornov isomorphism (both being closed points).  To complete the proof of the theorem, it is sufficient to prove that 
\begin{equation}
f(g) = \Phi(f)( X), \qquad \forall f\in \C[\Gr^{\overline{N\varpi_1}}_\mu]
\end{equation}
\begin{remark}
\label{remark: birational coordinates}
Both sides are irreducible algebraic varieties, so in fact it is sufficient to prove that this equation holds for $f$ ranging over the birational coordinates described in Proposition \ref{prop: birational coordinates}.
\end{remark}

The isomorphism $\gr Y_\mu^{\overline{N\varpi_1}}(\bR) \cong \C[ \Gr^{\overline{N\varpi_1}}_\mu]$ was described explicitly in Section \ref{section: more on gr slices}.  We now recall Brundan and Kleshchev's identification of the classical limit of $W(\pi)$ with functions on $\cT_\pi$, following \cite[Sections 3.3--3.4]{BK}. More precisely, they give an explicit isomorphism 
$$ W(\pi) \stackrel{\sim}{\longrightarrow} U(\fp)^{\fm}  \subset  U(\mathfrak{gl}_N)$$
In the classical limit, we identify $S(\mathfrak{gl}_N) \cong \C[ \mathfrak{gl}_N]$ via the trace pairing, and $S(\fp)^\fm \cong \C[e_\pi + \fm^\perp]^M $ (see \cite[\S 8]{BK3} for details).  Since $\cT_\pi$ is an MV slice, by Lemma \ref{lemma: 2.1 GG}, there is an isomorphism 
$$\C[e+\fm^\perp]^M \stackrel{\sim}{\rightarrow} \C[ \cT_\pi]$$
by restriction.  In other words, the isomorphism $\gr W(\pi) \cong \C[ \cT_\pi]$ comes by the composition 
\begin{equation}
\gr W(\pi) \hookrightarrow \C[ \mathfrak{gl}_N] \twoheadrightarrow \C[\cT_\pi]
\end{equation}
where the first arrow is Brundan-Kleshchev's embedding and the second is restriction.

Brundan-Kleshchev's embedding is defined via explicit elements $T_{ij; 0}^{(r)} \in U(\mathfrak{gl}_N)$, defined in \cite[\S 9]{BK3} (see also \cite[\S 3.3]{BK}).  Forming the $n\times n$--matrix $T(u) = (T_{ij}(u)) $ whose entries are formal series of functions $T_{ij}(u) = \delta_{ij} + \sum_{r>0} T_{ij; 0}^{(r)} u^{-r}$, we take its formal Gauss decomposition
$$ T(u) = F(u) D(u) E(u) $$
where $D(u)$ is diagonal and $E(u)$ (resp. $F(u)$) is upper (resp. lower) unitriangular.  Denote the diagonal entries of $D(u)$ by $D_i(u) = 1 + \sum_{r>0} D_i^{(r)} u^{-r}$, and the super-diagonal entries of $E(u)$ by $E_i(u) = \sum_{r>0} E_i^{(r)} u^{-r}$.  Then the elements $D_i^{(r)}, E_i^{(r)} \in U(\mathfrak{gl}_N)$ are the images of the same-named elements of $W(\pi)$ (and similarly for the $F_i^{(s)}$).

By abuse of notation, let us denote by $T_{ij;0}^{(r)}, D_i^{(r)}$, etc. the corresponding elements of the associated graded algebras.
\begin{lemma}
Suppose $\Gr^{\overline{N\varpi_1}}_\mu \ni g \mapsto X\in \cT_\pi \cap \mathcal{N}_{\mathfrak{gl}_N} $ are closed points corresponding under the Mirkovi\'c-Vybornov isomorphism.  Then
$$ T_{i,j; 0}^{(r)} (X) = (-1)^r a_{ji}^{(r)} $$
\end{lemma}
\begin{proof}
Recall our conventions on the pyramid $\pi$ from Section \ref{section: conventions on pyramids}.  By definition,
$$ T_{ij;0}^{(r)} = \sum_{s=1}^r (-1)^{r-s} \sum_{\substack{k_1,\ldots, k_s \\ \ell_1,\ldots, \ell_s}} e_{k_1, \ell_1} \cdots e_{k_s,\ell_s} $$
where the sum is over sequences with $1\leq k_t, \ell_t\leq N$, satisfying conditions (a), (b), (c), (e) and (f) from \cite[\S 3.3]{BK}.  In the classical limit, this is a function on $\mathfrak{gl}_N$ via the trace pairing.  When restricted to $\cT_\pi$, conditions (a), (b), (e) and (f) imply that the value at $X$ has the form
$$ \sum_{s=1}^r (-1)^{r-s} \sum_{\substack{x_2,\ldots, x_s \\ r_1+\ldots +r_s = r}} (-1)^s a_{j x_2}^{(r_1)} a_{x_2 x_3}^{(r_2)}\cdots a_{x_s i}^{(r_s)} $$
where the sum is over all sequences where $1\leq x_t\leq n$.  However, condition (c) implies that only the term with $s=1$ contributes. This proves the claim.
\end{proof}

Now, from (\ref{eq: main yangian map}) and Lemma \ref{lemma: defining s 1} it follows that the classical limit of $\Phi: Y_\mu^{N\varpi_1}(\bR) \stackrel{\sim}{\rightarrow} W(\pi)_\bR$ sends
$$ A_i^{(r)} \mapsto (-1)^r Q_i^{(r)}, \qquad E_i^{(r)} \mapsto (-1)^r E_i^{(r)} $$
where $Q_i(u) = D_1(u)\cdots D_i(u)$.  Therefore, by Remark \ref{remark: birational coordinates} the following result completes the proof of the main theorem:
\begin{proposition}
With notation as in the previous lemma, suppose that $g \mapsto X$.  Then for $i =1,\ldots, n-1$ we have equality of evaluations
\begin{align*}
A_i^{(r)} ( g) & = (-1)^r Q_i^{(r)}(X), \\
E_i^{(r)} (g) & = (-1)^r E_i^{(r)} (X)
\end{align*}
for the functions $A_i^{(r)}, E_i^{(r)} \in \C[ \Gr^{\overline{N\varpi_1}}_\mu]$ and $D_i^{(r)}, E_i^{(r)} \in \C[ \cT_\pi \cap \cN_{\mathfrak{gl}_N}]$, respectively.
\end{proposition}

\begin{proof}
If we take the Gauss decomposion of $T(u)$ and then evaluate the result at the point $X$, we will get the same result as first evaluating $T(u)$ at $X$ and then taking Gauss decomposition.  

By the previous lemma, if we evaluate $T(u)$ at $X$ we get the matrix $ g(-u)^T $ (i.e. the transpose of $g = g(t) \in G_1[[t^{-1}]]$ evaluated at $t = -u$).  Using the relation between the minors of a matrix and its transpose, we observe that
\begin{align*}
E_i^{(r)} (g) & = [t^{-r}] \frac{\Delta_{\{1,\ldots, i-1, i+1\}, \{1,\ldots, i\} }}{\Delta_{\{1,\ldots, i\}, \{1,\ldots, i\}} } \big( g(t) \big)  \\
 &= [t^{-r}] \frac{ \Delta_{\{1,\ldots, i\}, \{1,\ldots, i-1, i+1\} }}{\Delta_{\{1,\ldots, i\}, \{1,\ldots,i\}} } \big( g(t)^T \big) \\
 & = (-1)^r [u^{-r}] \frac{ \Delta_{\{1,\ldots, i\}, \{1,\ldots, i-1, i+1\} }}{\Delta_{\{1,\ldots, i\}, \{1,\ldots,i\}} } \big( g(-u)^T \big)
\end{align*}
The latter precisely extracts the superdiagonal entries of the ``$E$'' part of the Gauss decomposition of $g(-u)^T$.  Hence $E_i^{(r)}(g) = (-1)^r E_i^{(r)}(X)$, as claimed.  A similar calculation applies to $A_i^{(r)}$.
\end{proof}


\bibliographystyle{amsalpha}
\bibliography{monbib}

\end{document}